\documentclass{amsart}

\usepackage{amsthm} 
\usepackage{color}

\definecolor{refkey}{rgb}{1,0,1} 
\definecolor{labelkey}{rgb}{1,0,1}

\usepackage{graphicx,epsfig}
\usepackage{amsmath, amssymb, latexsym, euscript}

\usepackage{url}
\usepackage[all]{xy}
\usepackage{psfrag}
\usepackage{hyperref}

\definecolor{darkgreen}{rgb}{0.0, 0.5, 0.0}

\newtheorem{theorem}{Theorem}[section]
\newtheorem{lemma}[theorem]{Lemma}
\newtheorem{corollary}[theorem]{Corollary} 
 
\newtheorem{proposition}[theorem]{Proposition}
\newtheorem{remark}[theorem]{Remark}

\def\gap{\vspace{.3cm}\noindent}
\def\gp{\par\vspace{.15cm}\noindent}

\def\nb{$\bullet\ \ \ $}
\def\smallskip{\vspace{.15cm}}
\def\medskip{\vspace{.3cm}}
\def\text{\mbox}
\def\rh2{{\mathbb R}{\mathbb H}^2}
\def\ch2{{\mathbb C}{\mathbb H}^2}
\def\RP{{\mathbb RP}}

\def\RPn{{\mathbb RP}^n}
\def\RP2{{\mathbb RP}^2}

\def\RR{\SL(H,p)}
\def\GG{\mathcal G}
\def\GGHP{{\mathcal G}(H,p)}
\def\cOmega{\overline{\Omega}}
\def\bOmega{\partial\cOmega}
\def\rad{{\mathcal D}}
\def\Affn{{\text{Aff}}({\mathbb A}^n)}

\DeclareMathOperator{\CR}{CR}

\DeclareMathOperator{\SL}{SL}
\DeclareMathOperator{\inj}{inj}
\DeclareMathOperator{\diam}{diam}
\DeclareMathOperator{\Volume}{Volume}
\DeclareMathOperator{\thick}{thick}
\DeclareMathOperator{\thin}{thin}
\DeclareMathOperator{\Isom}{Isom}
\DeclareMathOperator{\Fix}{Fix}
\DeclareMathOperator{\im}{Im}
\DeclareMathOperator{\diag}{diag}
\DeclareMathOperator{\Par}{Par}
\DeclareMathOperator{\tr}{tr}
\DeclareMathOperator{\Int}{int}

\begin{document}

\title{On Convex Projective Manifolds and Cusps}

\author{Daryl Cooper}
\author{Darren Long}
\author{Stephan Tillmann}

\address{DC \& DL: Department of Mathematics, University of California, Santa Barbara, CA 93106, USA}
\address{ST: School of Mathematics and Statistics, The University of Sydney, NSW 2006, Australia}
\address{}

\email[]{cooper@math.ucsb.edu}
\email[]{long@math.ucsb.edu}
\email[]{tillmann@maths.usyd.edu.au}

\thanks{{Cooper is supported in part by NSF grant DMS--0706887. \\
\indent Long  is supported in part by NSF grant DMS--1005659.\\
\indent Tillmann  is supported in part by ARC grant DP1095760.\\
}}

\begin{abstract}
This  study of properly or strictly convex real projective manifolds introduces notions of {\em parabolic, horosphere} and {\em cusp}. Results include a Margulis lemma  and in the strictly convex case a thick-thin decomposition. Finite volume cusps are shown to be projectively equivalent to cusps of hyperbolic manifolds.  This is proved using a characterization of ellipsoids in projective space.

Except in dimension $3$,  there are only finitely  many topological types of strictly convex manifolds with  bounded volume. In dimension $4$ and higher, the diameter of a closed strictly convex manifold is at  most $9$ times the diameter of the thick part. There is an algebraic characterization of strict convexity in terms of relative hyperbolicity.
\end{abstract}

\maketitle

Surfaces are ubiquitous throughout mathematics; in good measure because of the {\em geometry}
of Riemann surfaces. Similarly, Thurston's insights into the geometry of 3--manifolds have led to many developments in diverse areas. This paper develops the bridge between real projective geometry and low dimensional topology.

Real projective geometry is a rich subject with many connections. In recent years it has
been combined with topology in the study of projective structures on manifolds.
Classically it provides a unifying framework as it contains the three constant curvature geometries
as subgeometries. In dimension $3$ it contains the eight Thurston geometries (up to a subgroup
of index $2$ in the case of product geometries) and there are paths of projective structures that correspond to transitions between different Thurston geometries on a fixed manifold. Moreover, there is a link between real projective deformations and complex hyperbolic deformations of a real hyperbolic orbifold (see \cite{CLT1}). Projective geometry therefore offers a general and versatile viewpoint for the study of 3--manifolds.

Another window to projective geometry: The symmetric space $SL(n,{\mathbb R})/SO(n)$ is isomorphic to
the group of projective automorphisms of the convex set in projective space
obtained from the open cone of positive definite quadratic forms in $n$ variables.
This set is {\em properly convex}: its closure is a 
compact convex set, which is disjoint from some projective
hyperplane. The boundary of the closure has a rich structure as it
consists of semi-definite forms and, when $n=3,$ contains
a dense set of flat $2$-discs; each corresponding to a family of semi-definite
forms of rank $2$ which may be identified with a copy of the hyperbolic plane.
 
From a geometrical point of view there is a crucial distinction between {\em strictly
convex} domains, which contain no straight line segment in the boundary, and the more general
class of properly convex domains. The former behave like manifolds of negative sectional curvature
and the latter like arbitrary symmetric spaces.  However, projective manifolds are more general: 
Kapovich \cite{Ka} has shown that there are 
closed strictly convex 4--manifolds which do not admit a hyperbolic structure.
 
The {\em Hilbert metric} is a complete Finsler metric on a properly convex set $\Omega.$ This is the hyperbolic metric in the Klein model when $\Omega$ is a round ball. A simplex with the Hilbert metric
is isometric to a normed vector space, and appears in a natural geometry on the Lie algebra 
${\mathfrak sl}_n.$ A singular version of this metric arises in the study of certain limits
of projective structures. The Hilbert metric has a Hausdorff {\em measure} and hence a notion of {\em finite volume}.
 
If a manifold of dimension greater than $2$ admits a finite volume complete hyperbolic metric,
then by Mostow-Prasad rigidity that metric is unique up to isometry. In dimension $2$ there is a finite
dimensional Teichm\"uller space of deformations, parameterized by an algebraic variety.
In the context of strictly convex structures on {\em closed} manifolds the deformation space 
is a semi-algebraic variety. There are closed hyperbolic $3$-manifolds for which this deformation space
has arbitrarily large dimension.
Part of the motivation for this work is to extend these ideas to the context
of finite volume structures, which  in turn is motivated by the study of these (and other still mysterious)
examples which arise via deformations of some finite volume non-compact convex projective
3-orbifolds. (See \cite{CLT2} and \cite{CLT1}.)

In the Riemannian context, there is a Margulis constant $\mu>0$ with  the following property: If $\Gamma$ is a discrete group of isometries of a
Hadamard space with curvature $-1\le K\le 0$ 
generated by isometries all of which move a given point a distance at most $\mu,$ then $\Gamma$ 
is virtually nilpotent, \cite{Ballmann} (9.5) p. 107.

\begin{theorem}[properly convex Margulis---see \S \ref{margulissection}]
\label{projectivemargulis} 
For each dimension $n\ge 2$ there 
is a Margulis constant $\mu_n>0$ with the following property. If $M$ is a properly convex projective 
$n$-manifold and $x$ is a point in $M,$ then the subgroup of $\pi_1(M,x)$ generated by loops 
based at $x$ of length less than $\mu_n$ is virtually nilpotent. 

In fact, there is a nilpotent subgroup  of index bounded above by  $m = m(n)$. Furthermore, if $M$ is strictly convex
and finite volume, this  nilpotent subgroup is abelian. If $M$ is strictly convex and closed, this nilpotent 
subgroup is trivial or  infinite cyclic. 
\end{theorem}

For complete Riemannian manifolds with pinched negative curvature $-b^2\le K\le-a^2<0$ there is a   {\em thick-thin} decomposition \cite{Ballmann} \S 10.
Each component of the {\em thin part} (where the injectivity radius is 
less than $\mu/2$) consists of {\em Margulis tubes} (tubular neighborhoods of short 
geodesics) and cusps, 

\begin{theorem}[strictly convex thick-thin -- see \S \ref{thickthinsection}  and also Proposition \ref{convexthin}]
\label{thickthin}  Suppose that $M$ is a  strictly convex projective  $n$-manifold. 
Then $M=A\cup B,$ where $A$ and $B$ are smooth 
submanifolds  and $\overline{A}\cap \overline{B}=\partial A=\partial B,$ and $B$ is nonempty, and $A$ is a possibly empty submanifold with 
the following properties:
\begin{enumerate}
\item If $\inj(x)\le \iota_n,$ then $x\in A,$ where $\iota_n=3^{-(n+1)}\mu_n$.
\item If $x\in A,$ then $\inj(x)\le\mu_n/2.$
\item Each component of $A$ is a Margulis tube or a cusp.
\end{enumerate}
\end{theorem}

We refer to $B$ as the {\em thickish part} and $A$ as the {\em thinnish part}. The injectivity radius on
 $\partial A$ is between $\iota_n$ and $\mu_n/2$. It follows from the description 
of the thinnish part  that the thickish part is connected in dimension greater than $2$.
 {\em Strictly} convex is necessary because when $M$ is properly convex, there is a properly convex
structure on $M \times S^1,$ where the circle factor is arbitrarily short. In this case the whole manifold is thinnish.

The proof of \ref{thickthin} requires a good understanding of  isometries  in the projective setting. 
A projective transformation which preserves an open properly convex set
is {\em elliptic} if it fixes a point. Otherwise it is {\em hyperbolic} or {\em parabolic} according
to whether or not the infimum of the distance that points are moved is positive. 
 A  study of isometries, with an emphasis on parabolics, 
in \S \ref{newparabolics} leads to the introduction of {\em algebraic horospheres} in \S \ref{horospheresection}.  
After discussing {\em elementary groups} in \S \ref{elementarysection}  a {\em cusp group}  is 
defined in \S  \ref{sectioncusps} as a discrete group that preserves some {\em algebraic horosphere}. 
Cusp groups are elementary and virtually nilpotent (\ref{productcusp}). Every infinite discrete group without hyperbolics  is a cusp group (\ref{virtparaboliccusp}). Informally, a {\em cusp} is a nice neighborhood of a convex suborbifold, with holonomy a cusp group, sitting inside a  properly convex, projective orbifold. 

\begin{theorem}[see \ref{productcusp}]\label{cuspisproduct} Every cusp is diffeomorphic to the product of an affine orbifold with virtually nilpotent holonomy and a line.
\end{theorem}

 A {\em maximal rank cusp} is a cusp with compact boundary.
These are the only cusps which appear in the finite volume setting. 
The projective orbifold $SL(3,{\mathbb Z})\backslash SL(3,{\mathbb R})/SO(3)$ is properly, but not
strictly, convex and has finite volume. The end is not a cusp. However an immediate consequence of \ref{thickthin} is:

\begin{theorem}
\label{endsaremaxcusps} 
Each end of a strictly convex projective manifold or orbifold of finite volume has
a neighborhood which is a maximal rank cusp. 
\end{theorem}

It follows that a finite volume strictly convex manifold is diffeomorphic to the interior
of a compact manifold.  
 Two cusps are {\em projectively equivalent} if their holonomies are conjugate.  
Given the diversity of parabolics, the next result is very surprising:

\begin{theorem}[see \S \ref{maximal_cusps_hyperbolic}]
\label{maximalcuspsstandard}
A maximal rank cusp in a properly convex  real projective manifold is projectively equivalent to a hyperbolic cusp of the same dimension. 
\end{theorem}
Thus the fundamental group of a maximal rank cusp is
virtually abelian, in contrast to the fact \ref{swan} that every finitely generated nilpotent group is the fundamental
group of some properly convex cusp. It follows that every parabolic and every elliptic in the holonomy of a strictly convex orbifold  of finite volume is 
conjugate into $PO(n,1)$.  This is not true in general for hyperbolic elements in strictly convex manifolds 
or for parabolics in infinite volume manifolds. 
A crucial ingredient for the study of maximal rank cusps is:
\begin{theorem}[see \ref{ellipsoid}]\label{charellipsoid}  Suppose that $\Omega$ is strictly convex. Then $\bOmega$ is an ellipsoid if and only if there
is a point $p\in\bOmega$ and a nilpotent group $W$ of projective transformations
 which   acts simply-transitively on $\bOmega\setminus p.$ 
\end{theorem}
A common fallacy is that since any two Euclidean structures on a torus are affinely equivalent it follows that
all hyperbolic cusps with torus boundary are projectively equivalent. However the projective and hyperbolic classification of maximal rank cusps coincide:

\begin{theorem}[See \ref{hyperboliccusps}]\label{hyperboliccusps1} 
Two hyperbolic cusps of maximal rank are projectively equivalent if and only if their holonomies are conjugate in $PO(n,1)$.
\end{theorem}

Benzecri's compactness theorem  implies  the set of balls of fixed radius in properly convex domains with the  Hilbert metric  is compact (\ref{hilbertballs}). Thus there is a lower bound on the volume of a component of the thinnish part, depending only on the dimension. Then \ref{thickthin} implies a result that is familiar in the setting of pinched negative curvature:

\begin{theorem}\label{finite volume}  
A  strictly convex projective manifold has finite volume  if and only if the thick part  is compact.
\end{theorem}

The Wang finiteness theorem \cite{wang} states  that there are a finite number of conjugacy classes 
of lattices  of bounded covolume in a semisimple Lie group without compact or three-dimensional 
factors. The Cheeger finiteness theorem \cite{Ch} bounds the number of topological types of manifolds   
given curvature, injectivity radius, and diameter bounds. The finiteness theorems in the projective setting
lie somewhere between these two.

\begin{theorem}[strictly convex finiteness---see \S \ref{top_finiteness}]\label{topfinite}
 In every dimension there are at most finitely many 
homeomorphism types for the thick parts of 
strictly convex projective manifolds with volume at most $V.$ Moreover:
\begin{enumerate}
\item In dimension $n\ne 3$ there are at most finitely many homeomorphism classes of 
strictly convex projective manifolds of dimension $n$ and volume at most $V.$
\item Every strictly convex projective $3$-manifold of volume at most $V$ is obtained by Dehn-filling one of 
finitely many $3$-manifolds, which depend on $V$.  
\end{enumerate}
\end{theorem}

Though there are only finitely many homeomorphism classes,
 the earlier discussion of moduli means there are infinitely many  projective equivalence classes in
 every dimension  greater than $1$. This finiteness result does not extend to {\em properly} convex manifolds because
the product of any compact properly
convex manifold and a circle has a properly convex structure of arbitrarily small volume; however:

\begin{proposition}[properly convex finiteness---see \S \ref{top_finiteness}]\label{properly_convex_finiteness} Given $d,\epsilon>0$, 
there are only finitely many homeomorphism  classes of closed properly
convex  $n$-manifolds  with diameter less than $d$ and containing a point with 
injectivity radius larger than $\epsilon.$
\end{proposition}

A key ingredient for these finiteness theorems is a version for the Hilbert metric of a standard tool from Riemannian 
geometry with {\em pinched curvature}:  
\begin{proposition}[decay of injectivity radius---see Theorem \ref{decay}]
\label{decayannounce} 
If $M$ is a properly convex projective $n$-manifold and $p,q$ are two points in $M,$ 
then $\inj(q)>f(\inj(p),d_M(p,q)),$ where $f$ depends only on the dimension.
\end{proposition} 
The {\em depth} of a Margulis tube is the minimum  distance of points on the boundary of the tube 
from the core geodesic.  Two more consequences of \ref{decayannounce} are: 
\begin{theorem} [Volume bounds diameter---see Theorem \ref{boundedvolume_gives_bounded_diamater}]
\label{volumediameter}
If $n\ge 4$ there is $c_n>0$ such that if $M^n$ is a closed, strictly convex real projective manifold  
then $\diam(M)\le 9\cdot \diam(\thick(M))\le c_n\cdot \Volume(M)$. 
\end{theorem}

\begin{proposition}[uniformly deep tubes---see Theorem \ref{deeptube}]\label{deeptubes} 
For each dimension $n$ there is a 
decreasing function $d:(0,\mu_n]\longrightarrow{\mathbb R}_+$ with $\lim_{x\to0} d(x)=\infty$ 
such that  a Margulis tube in a properly convex manifold with core geodesic of length less 
than $\epsilon$ has depth  greater than $d(\epsilon).$  \end{proposition} 

Another ingredient  of \ref{topfinite} is related to Paulin's \cite{Paulin} equivariant Gromov-Hausdorff topology, with a key difference being that in \cite{Paulin} the group remains fixed.

\begin{theorem}[see \S \ref{top_finiteness}]
\label{GHcompact} Given $\epsilon>0$  let ${\mathcal H}$ be the 
set of isometry classes of pointed metric spaces $(M,x),$ where $M$ is a 
properly convex projective $n$-manifold with the Hilbert metric and $\inj(x)\ge\epsilon.$ 

Then ${\mathcal H}$ is compact in the pointed Gromov-Hausdorff topology.
\end{theorem}

The next result is due to Benoist \cite{B0A} in the closed case. Choi has obtained a similar result
with different hypotheses.
\begin{theorem}[relatively hyperbolic---see Theorem \ref{generalized_benoist}]\label{relhyp} 
Suppose $M=\Omega/\Gamma$ is a properly 
convex manifold of finite volume which is the interior of a compact manifold $N$ and 
the holonomy of each component of $\partial N$ is parabolic. Then the following are
equivalent:
\begin{enumerate}
\item $\Omega$ is strictly convex,
\item $\bOmega$ is $C^1,$ 
\item $\pi_1N$ is hyperbolic relative to the subgroups of the boundary components.
\end{enumerate}
\end{theorem}

There has been a lot of work on {\em compact} manifolds of the form
$\Omega/\Gamma,$ where $\Omega$ is the interior of a strictly convex compact
set in Euclidean space and $\Gamma$ is a discrete group of real-projective
transformations that preserve $\Omega.$ We mention Goldman
\cite{Goldsurfaces}, \cite{goldman1}, Benoist \cite{B1},
\cite{B2},\cite{B3}, \cite{B4}, Choi \cite{C1}, \cite{Choicox} and Choi and
Goldman \cite{CG1}.

The thick-thin decomposition was obtained in dimension $2$ by Choi \cite{choi1}, 
where he asked if it could be extended to arbitrary
dimensions.  During the course of this work, Choi obtained some results
similar to some of ours (see \cite{choiopen1}), and we learnt of Marquis
\cite{marquis1},\cite{marquis2},\cite{marquis3} who has studied finite area
projective surfaces and constructed examples of cusped non-hyperbolic real
projective manifolds in all dimensions.  Recently he and Crampon proved a
Margulis lemma \cite{Marqpp}.  In another recent paper,
Crampon discusses parabolics and cusps in the
$C^1$ setting in \cite{crampon2}. This avoids many complications. 
 Our proof of the Margulis lemma in the
properly convex case occupies section \ref{margulissection} and does not
depend on the earlier sections.  The enhanced result in the finite volume
strictly convex case follows from \ref{maximalcuspsstandard}.

The picture which seems to be emerging from the work herein is that 
finite-volume {\em strictly} convex
manifolds behave like {\em hyperbolic} manifolds, {\it sans} Mostow rigidity. However
they are more general.
There are similarities between the notions {\em properly} convex and {\em pinched non-negative curvature}.
This is related to Benzecri's compactness theorem \ref{benzecricpct} which
provides a {\em compact family} of charts around each point.  The proof
 that finite volume cusps are hyperbolic starts with the observation that far
out in the cusp the holonomy is almost dense in a Lie group, which must be
nilpotent by the Margulis lemma. Then one uses the theory of nilpotent
Lie groups. The reader should be aware that despite
the {\em {\color{blue} parallels}}, many
familiar facts from hyperbolic geometry do not hold in the projective
context.\\

\textbf{Acknowledgements} The authors thank Olivier Guichard for comments on a previous draft of this paper.

\section{Projective Geometry and Convex Sets}

If $V$ is a finite dimensional real vector space, then ${\mathbb P}(V)=V/{\mathbb R}^{\times}$ is the projectivization and $PGL(V)$ is the group of projective transformations. A {\em projective subspace} is the image ${\mathbb P}(U)\subseteq {\mathbb P}(V)$ of a vector subspace $U\subseteq V$, and is called a {\em (projective) line} if $\dim U=2$. If $\dim V = n$ a {\em projective basis} of ${\mathbb P}(V)$ is an $(n+1)$--tuple of distinct points ${\mathcal B}=(p_0,p_1,\cdots,p_{n})$ in ${\mathbb P}(V)$ such that no subset of $n$ distinct points lies in a projective hyperplane. The set of all projective bases is an open subset ${\mathcal U}\subset{\mathbb P}(V)^{n+1}$.
\begin{proposition}\label{projbasis} 
For ${\mathcal B}_0\in{\mathcal U}$ the map $PGL(V)\longrightarrow {\mathcal U}$ given by $\tau\mapsto\tau{\mathcal B}_0$ is a homeomorphism.
\end{proposition}

To refer to eigenvalues it is convenient to work
with the double cover of projective space
  ${\mathbb S}(V)=V/{\mathbb R}_+$ 
with automorphism group  $SL(V),$ which in this paper is \emph{the group of matrices of determinant $\pm 1.$} 
We write ${\mathbb R}P^n={\mathbb P}({\mathbb R}^{n+1})$ and $S^n={\mathbb S}({\mathbb R}^{n+1})$.

 The set $C\subseteq{\mathbb P}(V)$ is {\em convex} if the intersection of every line with $C$ is connected.
 An {\em affine patch} is a subset of ${\mathbb R}P^n$ obtained by deleting a codimension-1 projective 
 hyperplane. A  convex subset $C\subseteq{\mathbb R}P^n$ is 
 {\em properly convex} if  its closure is contained in an affine patch.
The point $p\in\partial\overline{C}$ is a {\em strictly convex} point if
it is not contained in a line segment of positive length in $\partial\overline{C}$.
The set  $C$ is {\em strictly convex} if it is properly convex 
and strictly convex at every point in $\partial\overline{C}$. 

Let $\pi:S^n\longrightarrow{\mathbb R}P^n$ denote the double cover. If $\Omega$ is a properly convex subset of ${\mathbb R}P^n,$
then $\pi^{-1}\Omega$ has two components, each with closure contained in an open hemisphere. 
We choose one as a lift and refer to it as $\Omega,$ and we will always assume that \emph{$\Omega$ is open.}

We use the notation $SL(\Omega)$ for the subgroup of $SL(n+1,{\mathbb R})$ which preserves 
$\Omega.$ It is naturally isomorphic to the subgroup $PGL(\Omega)\subset PGL(n+1,{\mathbb R})$ which 
preserves $\Omega.$ It is convenient to switch back and forth between talking about projective space
and its double cover, and between talking about $PGL(\Omega)$ and $SL(\Omega).$ This allows
a certain economy of expression and should not cause confusion. 

A subset ${\mathcal C} \subset {\mathbb R}^{n+1}$  is a {\em  cone} if $\lambda \cdot {\mathcal C} = {\mathcal C}$  for all  $\lambda>0$, and is {\em sharp} if it contains no affine line. A properly convex domain $\Omega\subset S^n$ determines a  sharp convex cone
${\mathcal C}(\Omega)={\mathbb R}_+\cdot\Omega\subset{\mathbb R}^{n+1}.$ 
Then $SL({\mathcal C})=SL(\Omega)$ is the subgroup of $SL(n+1,{\mathbb R})$ which preserves ${\mathcal C}$.
 
The dual of  the vector space $V$  is denoted $V^*$. A codimension-1 vector subspace
$U\subset V$ determines a $1$-dimensional subspace of $V^*$. This gives a natural
bijection  called {\em duality} between codimension-1 projective hyperplanes in ${\mathbb P}(V)$ and points in ${\mathbb P}(V^*)$. There is a natural action of $\SL(V)$ on $V^*$. Using a basis of $V$ and the dual basis of $V^*$ if $T\in \SL(V)$ has matrix $A$ then the matrix for the action of $T$ on $V^*$ is $A^*=\text{transpose}(A^{-1})$.

 If $\Omega\subset {\mathbb S}(V)$ is a properly convex set 
the {\em dual} is $\Omega^*\subset{\mathbb S}(V^*),$ which is the projectivization of the {\em dual cone}
$${\mathcal C}^*(\Omega)=\{\ \phi\in V^*\ :\ \forall v\in\overline{\Omega}\ \  \phi(v)>0\ \}.$$
 A point $[\phi]\in \bOmega^*$ is dual to a supporting hyperplane to $p=[u]\in \bOmega$ iff $\phi(u)=0$.
Hence the subset of ${\mathbb P}(V^*)$ dual to supporting hyperplanes at  $p=[u]\in\bOmega$ is the projectivization of the cone 
$${\mathcal C}^*(\Omega,p)=\{\ \phi\in\overline{{\mathcal C}^*(\Omega)}\ :\ \phi(u)=0\ \},$$
from which one easily sees
\begin{proposition}\label{normalcone} If $\Omega\subset{\mathbb S}(V)$ is properly convex  the subset  ${\mathbb S}({\mathcal C}^*(\Omega,p))\subset{\mathbb S}(V^*)$ dual to supporting hyperplanes to $p\in\bOmega$
is  compact and properly convex.
\end{proposition}

 A group, $G$,  of homeomorphisms of a locally compact Hausdorff space $X$  acts {\em properly discontinuously} if for every compact $K\subset X$ the set $K\cap gK$ is nonempty for at most finitely many $g\in G$. 
\begin{proposition}\label{discreteproperdiscts} Suppose $\Omega$ is properly convex and $\Gamma\subset PGL(\Omega)$. Then
$\Gamma$ is a discrete subgroup of $PGL(n+1,{\mathbb R})$ iff $\Gamma$ acts properly discontinuously on $\Omega$.
\end{proposition}
\begin{proof} Suppose there is a sequence of distinct elements
$\gamma_i\in \Gamma$ converging to the identity in $PGL(n+1,{\mathbb R})$. Let $K\subset\Omega$ be a compact set containing $[v]$ in its interior.
Then $\gamma_i[v]\in K$ for all sufficiently large $i$ so $\Gamma$ does not act properly discontinuously. Conversely, suppose $K\subset\Omega$ is compact and there is a sequence of distinct elements $\gamma_i\in\Gamma$ with
$K\cap\gamma_iK\ne \phi$. Choose a projective basis ${\mathcal B}=(x_0,\cdots,x_n)\subset\Omega$ with $x_0\in K$. After taking a subsequence we may assume $\gamma_i{\mathcal B}$ converges to a subset of $\Omega$. The sequence $\delta_i=\gamma_{i+1}^{-1}\gamma_i\in \Gamma$ has the property $\delta_i{\mathcal B}\to {\mathcal B}$ because $\delta_i$ is an {\em isometry}.
By \ref{projbasis}, this implies $\delta_i$ converges to the identity. 
\end{proof}

A {\em properly convex projective orbifold} is $Q=\Omega/\Gamma,$ where $\Omega$ is an open
 properly convex set and $\Gamma\subseteq SL(\Omega)$
 is a discrete group. Similarly for {\em strictly convex}. This orbifold is a manifold iff $\Gamma$ is torsion
 free. 
Since points in $\Omega^*$ are the duals of hyperplanes disjoint
 from $\Omega$ it follows that under the dual action $SL(\Omega)$ preserves $\Omega^*.$ 
 Thus given a properly convex projective orbifold $Q,$ there is a dual orbifold $Q^*=\Omega^*/\Gamma^*.$
Two  orbifolds $\Omega/\Gamma$ and $\Omega'/\Gamma'$
 are {\em projectively equivalent} if there is a homeomorphism between them
which is covered by the restriction of a projective transformation mapping $\Omega$ to $\Omega'$.
In general $Q$ is not projectively equivalent to $Q^*$, see \cite{CD}. 

 \begin{proposition}[convex decomposition]
 \label{convexdecompose} 
 If $\Omega$ is an open convex subset  of ${\mathbb R}P^n$ which contains no projective line, then 
 it is a subset ${\mathbb A}^k\times C$ of  some affine patch 
${\mathbb A}^k\times{\mathbb A}^{n-k}\subset {\mathbb R}P^n,$ 
  where $k\ge 0$ and 
$C\subset{\mathbb A}^{n-k}$ is a properly convex set. One factor might be a single point. The set $C$ is unique 
up to projective isomorphism.
\end{proposition}

\begin{proof} In \cite{GV} it is shown there is  an affine patch 
${\mathbb A}^n={\mathbb R}P^n\setminus H$ which contains $\Omega$.  
Choose an affine subspace ${\mathbb A}^k\subseteq \Omega$ of maximum dimension  $k\ge 0$. Then
$k=0$ iff  $\Omega$ contains no affine line. Since $\Omega$ is convex and open, it follows that 
$\Omega={\mathbb A}^k\times C$ for some open convex set $C\subset{\mathbb A}^{n-k}.$ Since $k$ is 
maximal it follows that $C$ contains no affine line. 

The closure $\overline{C}\subset{\mathbb R}P^{n-k}$ contains no projective line. By \cite{GV}  it is disjoint from some 
projective hyperplane $H'\subset{\mathbb R}P^{n-k}$. Thus $\overline{C}$ is a compact subset of the affine patch 
${\mathbb R}P^{n-k}\setminus H',$ so $C$ is properly convex. Uniqueness of $C$ up to projective isomorphism 
follows from the fact that a projective transformation sends affine spaces to affine spaces.
\end{proof}

Suppose $U\subseteq V$ is a $1$-dimensional subspace. 
The set of lines in ${\mathbb P}(V)$ containing the point $p=[U]$ is the projective space 
${\mathbb P}(V/U)$ and is called the {\em space of directions} at $p$.  
{\em Radial projection towards $p$} is 
  ${\mathcal D}_p:{\mathbb P}(V)\setminus \{p\}\longrightarrow{\mathbb P}(V/U)$ given by ${\mathcal D}_p[v]=[v+U]$.
   The image of a subset 
  $\Omega\subseteq {\mathbb P}(V)$ is denoted ${\mathcal D}_p\Omega$ and
is called the {\em  space of directions} of $\Omega$ at $p$.

 A projective transformation $\tau\in PGL(V)$ which fixes
$p$ induces a projective transformation $\tau_p$ of ${\mathbb P}(V/U)$. If $A\in GL(V)$ represents $\tau$ then $A(U)=U$
and $\tau_p([v])=[Av+U]$. 

Passing to  double covers of these projective spaces, 
${\mathbb S}(V/U)$ is the set of {\em oriented} lines containing a lift of $p$ and is also  called the {\em space of directions}.
Suppose that $A\in\SL(\Omega)\subseteq \SL(V)$
 fixes $p\in\bOmega$. Then $A$ preserves the orientations of lines through $p$ 
and so induces $A_p\in \SL(V/U)$. We will make frequent use of:

\begin{proposition}\label{directioneigenvalues} Suppose $\Omega\subset S^n$ is  properly convex, $p\in\bOmega$ and $A\in \SL(\Omega)$ fixes $p.$ Choose a basis of ${\mathbb R}^{n+1}$ with first vector $e_1$ representing $p;$ thus
$Ae_1=\lambda_1 e_1$ and $\lambda_1>0.$  Then $A_p=\sqrt[n]{1/\lambda_1}B$ where $B$
is the $n\times n$ submatrix obtained from the matrix $A$ by ommiting the first row and column. In particular,
if $\lambda_1=1$ then
the eigenvalues counted with multiplicity of $A_p$ are the subset of the eigenvalues of $A,$ where the algebraic multiplicity of $\lambda_1$ is reduced by $1$.
\end{proposition}
 
If $\Omega$ is a properly convex domain and $p\in\bOmega,$ then ${\mathcal D}_p\Omega$
 is open and convex because $\Omega$ is, and it is contained in an affine patch  given by the complement
of the image of any supporting hyperplane of $\Omega$ at $p.$
A subset $U\subset{\mathbb R}P^n$ is {\em starshaped at $p$} if $p\in\overline{U}$ and the intersection
with $\overline{U}$ of every line containing $p$ is connected. 

 At a point $p\in\bOmega$ locally $\bOmega$ is the graph of a function
  defined on a neighborhood of $p$ in a supporting hyperplane $H$. 
 By (2.7 of \cite{gruber}) this function is $C^1$ at $p$
 iff $H$ is the unique supporting hyperplane at $p$ iff the dual point $H^*$ is a strictly convex
  point in $\partial\overline{\Omega^*}$. The point $p$ is called a {\em round} point of $\bOmega$ if $p$ is both a $C^1$ point and a strictly convex point of $\bOmega$. Round points play an
important role in the study of cusps. 

 \begin{corollary}\label{spacedirections}  Suppose $\Omega^n$ is properly convex and $p\in\bOmega$.
\begin{enumerate}
\item  ${\mathcal D}_p\Omega$ is projectively equivalent to ${\mathbb A}^k\times C$ where $C$ is a properly convex open set
 and $\dim C=n-k-1$. One of the factors might be a single point.
\item $p$  is a $C^1$ point iff
${\mathcal D}_p\Omega={\mathbb A}^{n-1}$. 
\item $p$ is a strictly convex point iff $\rad_p|(\bOmega\setminus\{p\})$ is injective.
\item $p$ is a round point iff  the restriction of $\rad_p$ is a homeomorphism from 
$\bOmega\setminus \{p\}$ to ${\mathbb A}^{n-1}$.
\end{enumerate}
 \end{corollary}

The {\em Hilbert metric} $d_{\Omega}$ on a
properly convex open set $\Omega$ is  $d_{\Omega}(a,b)=\log |\CR(x,a,b,y)|,$
where $x,y\in\bOmega$ are the endpoints of a line segment in $\Omega$ containing $a$ and $b$
such that $a$ lies between $x$ and $b$ on the line segment and 
$$\CR(x,a,b,y)=\frac{\|b-x\|\cdot \|a-y\|}{\| b -y\|\cdot \|a-x\|}$$
 is the {\em cross ratio}. This is a complete
Finsler metric with:
$$ds = \log |\CR(x,a,a+da,y)| =\left(\frac{1}{|a-x|}+\frac{1}{|a-y|}\right)da.$$
 This gives {\em twice} the hyperbolic metric when $\Omega$ is the interior of an ellipsoid.  Every
segment of a projective line in $\Omega$ is length minimizing, and in the strictly convex
case these are the only geodesics. This metric defines a Hausdorff-measure on $\Omega$ which is denoted
$\mu_{\Omega}$ and is absolutely continuous with respect to Lebesgue measure.

Since projective transformations preserve 
cross ratio, $\SL(\Omega)$ is a group of isometries of the Hilbert metric. The inclusion $\SL(\Omega) \le \text{Isom}(\Omega, d_\Omega)$ may be strict. The Hilbert metric and associated measure descend to $Q=\Omega/\Gamma$  giving a {\em volume} $\mu_{\Omega}(Q)$.

\begin{figure}[h]
 \begin{center}
\psfrag{x}{$x$}
\psfrag{y}{$y$}
\psfrag{z}{$z$}
	 \includegraphics[scale=0.5]{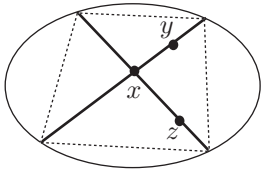}
 \end{center}
\caption{Comparing to a quadrilateral}	 \label{lines2}
\end{figure}

\begin{lemma}\label{convexballs} 
If $\Omega$ is properly (resp. {\em strictly}) convex, then metric balls of the Hilbert metric are  convex (resp. {\em strictly convex}).
\end{lemma}

\begin{proof} 
 Refer to Figure \ref{lines2}. Suppose $R=d(x,y)=d(x,z).$ We need to show that for every 
$p\in[y,z],$ we have $d(x,p)\le R.$ The extreme case is obtained by taking the quadrilateral $Q\subset\Omega$ 
which is the convex hull of the four points on $\bOmega,$ where the extensions of the segments 
$[x,z]$ and $[x,y]$ meet $\bOmega.$ Then $d_{\Omega}\le d_Q$ and  the ball of radius $R$ in $Q$ center $x$ is a convex quadrilateral.
\end{proof}

Example E(ii) below shows metric balls might not be strictly convex.
In this case geodesics are not even locally unique.
 A function defined on a convex set is {\em convex} if the 
restriction to every line segment is 
convex. The statement that metric balls centered at the point $p$ 
are convex is equivalent to the statement 
that the function on $\Omega$ defined by $f(x)=d_{\Omega}(p,x)$ is convex.
Soci{\'e}-M{\'e}thou \cite{Socie} showed that $d_{\Omega}(x,y)$ is not a 
geodesically convex function,
 in contrast to the situation in hyperbolic and Euclidean space. 
However, the following lemma leads to a maximum principle for the distance function.

\begin{lemma}[4 points]\label{4points} Suppose $a,b,c,d$ are points in a properly convex
 set $\Omega$ and that $R=d_{\Omega}(a,b)=d_{\Omega}(c,d).$ Then  every point on $[a,c]$ is within distance $R$ of $[b,d].$ 
\end{lemma}
\begin{proof} Refer to Figure \ref{simplexpic}. Let $A,B$ be the points in $\bOmega$ such that the line $[A,B]$ contains $[a,b].$ 
Define $[C,D]$ similarly. Let $\sigma$ be the interior of the convex hull of $A,B,C,D.$ 
Then $\sigma\subset\Omega,$ so $d_{\sigma}\ge d_{\Omega}.$ The formula for the Hilbert metric
on $\sigma$ makes sense for pairs of points on the same edge in the 1-skeleton of $\sigma$.
Then, by construction 
$d_{\sigma}(a,b)=d_{\Omega}(a,b)$ and $d_{\sigma}(c,d)=d_{\Omega}(c,d).$ Thus it 
suffices to prove the result when $\Omega=\sigma.$ 

\begin{figure}[ht]
 \begin{center}
	 \psfrag{om3}{$\Omega$}
	 \psfrag{a}{$a$}
	 \psfrag{b}{$b$}
	 \psfrag{ell}{$\ell$}
	 \psfrag{c}{$c$}
	 \psfrag{d}{$d$}
	 \psfrag{A}{$A$}
	 \psfrag{B}{$B$}
	 \psfrag{C}{$C$}
	 \psfrag{D}{$D$}
 \psfrag{x}{$x$}
	 \psfrag{y}{$y$}
	 \psfrag{X}{$X$}
	 \psfrag{Y}{$Y$}
	 \includegraphics[scale=0.6]{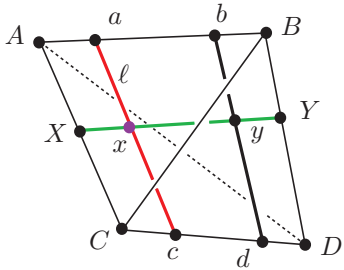}
	 \caption{The Simplex $\sigma$}
	 \label{simplexpic}
 \end{center}
\end{figure}

We may therefore assume that $\Omega=\sigma$ is a possibly degenerate $3$-simplex. 
The  degenerate case follows from the non-degenerate case by a continuity argument.

The identity component $H$ of $SL(\sigma)$ fixes the vertices of $\sigma$
 and acts simply transitively
on $\sigma.$ If we choose coordinates so that the vertices of $\sigma$
are represented by basis vectors, then $H$ is the group of positive diagonal matrices with
determinant $1.$ A point $x$ in the interior of $\sigma$ lies on a unique line segment, $\ell=[a,c]$,  in $\sigma$ with 
one endpoint  $a\in (A,B)$ and the other  $c\in (C,D)$. It follows that the subgroup of $H$  that preserves $\ell$
is a one-parameter group which acts simply-transitively on $\ell$.

The point $x$ also lies on a unique segment $[X,Y]$
with $X\in(A,C)$ and $Y\in(B,D)$.
Let $G=G_1\cdot G_2$ be the two parameter subgroup of $H$ that is the product of the stabilizers,
$G_1$ of $[a,c]$ and $G_2$ of $[X,Y]$.
The $G$-orbit of $x$  is a doubly ruled surface: a hyperbolic paraboloid.
The $G_1$-orbit of the line $G_2\cdot x=(X,Y)$
 gives one ruling. The $G_2$-orbit of the line $G_1\cdot x=(a,c)$  gives the other ruling.
This surface is the interior of a twisted square with corners $A,B,C,D$. 
Since $G$ acts by isometries
and $d_{\sigma}(a,b)=d_{\sigma}(c,d),$ it follows that $[a,c]$ is sent to $[b,d]$ by an element of $G.$
Thus $[b,d]$ intersects $[X,Y]$ at a point $y$.
The segment $[x,y]$
can be moved by elements of $G$ arbitrarily close to both $[a,b]$ and to $[c,d]$. 
Furthermore, $d_{\sigma}(g\cdot x,g\cdot y)$ is independent of $G.$ 
It follows by continuity of cross-ratio
that this constant is $d_{\sigma}(a,b).$ \end{proof}

A point $x$ in a set $K$ in Euclidean space is an {\em extreme point} if it is not contained in the
interior of a line segment in $K$.
It is clear that the extreme points of a compact set $K$ must lie on its frontier and that 
 $K$ is the convex hull of its extreme points, \cite{KM}.
 If $\Omega$ is properly convex, a function $f:\Omega\longrightarrow{\mathbb R}$ 
satisfies the {\em maximum principle} if for every compact subset $K\subset\Omega$ the 
restriction $f|K$ attains its maximum at an extreme point of $K.$ 
\begin{corollary}[Maximum principle]
\label{maxprinciple} 
If $C$ is a closed convex set in a properly
convex domain $\Omega,$ then the distance of a point in $\Omega$ from $C$ satisfies 
the maximum principle.
\end{corollary}
\begin{proof} The function $f(x)=d_{\Omega}(x,C)$  is $1$-Lipschitz, therefore continuous. 
Let $K\subset\Omega$ be a compact  set then $f|K$ attains its maximum at some point $y.$ 
There is a finite minimal set, $S$, of extreme points of $K$ such that $y$ is in their convex hull.
Choose $y$ to minimise $|S|$. If $S$ contains more than one point then
 $y$  is in the interior of a segment $[a,b]\subset K$ with $a\in S$ and $b$ in the convex hull of $S'=S\setminus y.$
  Since $C$ is closed and $f$ 
is continuous there are $c,d\in C$ with $f(a)=d_{\Omega}(a,c)$ and $f(b)=d_{\Omega}(b,d).$ 
Since $C$ is convex $[c,d]\subset C.$ 

Assume for purposes of contradiction that   $f(y)>f(a) = d_{\Omega}(a,[c,d])$ and 
$f(y)>f(b)=d_{\Omega}(b,[c,d]).$ Then we may find $a',b'$ on $[a,b]$ 
such that $y\in [a',b']$ and $f(a')=f(b')< f(y) .$ By the $4$-points lemma
$d_{\Omega}(y,[c,d])\le f(a').$ However, $[c,d]\subset C$ and so 
$f(y)\le d_{\Omega}(y,[c,d]),$ giving the contradiction $f(y)\le f(a').$ 
\end{proof}
\begin{corollary}[convexity of $r$-neighborhoods]
\label{convexnbhds} 
If $C$ is a closed convex set in a properly convex 
domain $\Omega$ and $r>0,$ then the $r$-neighborhood of $C$ is convex. 

In particular, an $r$-neighborhood of a line segment is convex.
\end{corollary}

\begin{lemma}[diverging lines]
\label{divergeline} Suppose $L$ and $L'$ are two distinct line segments in  a 
strictly convex domain $\Omega$ which start at $p\in\partial\Omega.$ Let $x(t)$ and $x'(t)$  be  parameterizations
of $L$ and $L'$ by arc length so that increasing the parameter moves away from $p.$ 

Then $f(s)=d_{\Omega}(x(s),L')$ is a monotonic increasing homeomorphism $f:{\mathbb R}\longrightarrow (\alpha,\infty)$ for some $\alpha\ge 0$. 
Furthermore $\alpha=0$ if $p$ is a $C^1$ point.
\end{lemma}

\begin{proof} Refer to figure \ref{dpic}. We may reduce to two dimensions by intersecting with a 
plane containing the two lines.
The function is $1$-Lipschitz, thus continuous. 
Let $x'(s')$ be some point on $L'$ closest to $x(s)$, and let $\Omega_s$ be the subdomain of $\Omega$ which is the triangle  with vertices $p, q(s), r(s)$ shown dotted. The following 
facts are evident. The distance between $x(s)$ and $x'(s')$ is the same in both $\Omega$ and $\Omega_s$. 
For $t>0$ we have $f(s-t) \le d_{\Omega_{s-t}}(x(s-t),x'(s'-t)).$ Finally $d_{\Omega_s}(x(s-t),x'(s'-t))$ is 
constant for $t>0$. The obvious comparison applied to triangular domains  $\Omega_{s}$ and  $\Omega_{s-t}$
gives the monotonicity statement. 

\begin{figure}[ht]
 \begin{center}
 \psfrag{om}{$\Omega$}
 \psfrag{omm}{$\Omega_s$}
\psfrag{x}{$p$}
\psfrag{y}{$x(s)$}
\psfrag{z}{$x'(s')$}
\psfrag{x}{$p$}
\psfrag{a}{$ $}
\psfrag{b}{$ $}
\psfrag{c}{$q(s)$}
\psfrag{d}{$r(s)$}
	 \includegraphics[scale=0.7]{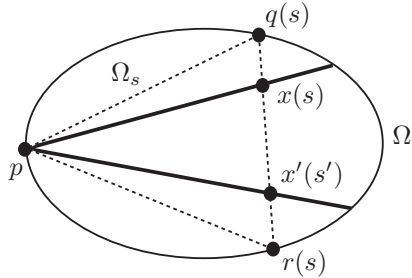}
	 \end{center}
 \caption{Diverging Lines} \label{dpic}
\end{figure}

If now $p$ is a $C^1$ point, then there is an unique tangent line to $\partial \Omega$ at $p$ and the
triangular domains  have the angle at $p$ increasingly close to $\pi$. This implies that the distance tends to zero.

It only remains to show $f$ is not bounded above. Let $a(s)=|q(s)-x(s)|$ and $b(s)=|r(s)-x'(s')|$.
If $f(s)=d_{\Omega}(x(s), x'(s'))$ is bounded above
as $s\to\infty$ then, using the cross ratio formula for distance and the fact $|x(s)-x'(s')|$ is bounded away from zero,  $a(s)$ and $b(s)$ are bounded away from $0$.
Using the fact that $\Omega$ is convex,  the limit as $s\to\infty$ of the segment with endpoints $q(s)$ and $r(s)$  is  a line segment in $\bOmega$.
\end{proof}

\section{Projective Isometries}\label{newparabolics}

Let $\Omega\subseteq S^n$ be an open properly convex domain. An element $A\in SL(\Omega)$ is
called a {\em projective isometry}. If $\Omega$ is strictly convex then every isometry of the 
Hilbert metric is of this type. If $A$ fixes a point in $\Omega$ it is called {\em elliptic}.
If $A$ acts freely on $\Omega$ it is {\em parabolic} if every eigenvalue has modulus $1$ and {\em hyperbolic} otherwise. The main results are  summarized in  \ref{fixedpoints}, \ref{isometryfoliation} and
\ref{characterizeparabolics}.
The {\em translation length} of  $A$ is 
$$t(A) = \inf_{x\in X}d_{\Omega}(x,A x).$$
The subset of $\Omega$ for which this infimum is attained is called the {\em minset} of $A$. It might be empty.
 Later we derive the following algebraic formula for translation length which implies hyperbolics have positive translation length and parabolics have translation length zero. The following result is proved at the end of this section for elliptic and hyperbolic isometries, and in \ref{nonhypsmalldist} for parabolics.
\begin{proposition}\label{translationlength}
$t(A)=\log |\lambda/\mu|,$ where $\lambda$ and $\mu$ are eigenvalues of $A$ of maximum and 
minimum modulus respectively.
\end{proposition}

For future reference, and to illustrate the diversity, we present some key examples of
{\em homogeneous domains}, i.e.\thinspace domains $\Omega$
on which $\SL(\Omega)$ acts transitively. These have been classified by Vinberg \cite{vin} and include:
\begin{itemize}
\item[E(i)]  {\em The projective model of hyperbolic space} ${\mathbb H}^n$ is  identified 
with the unit ball  $D^n\subseteq {\mathbb R}P^n$ and $SL(D^n)\cong PO(n,1)$
\item[E(ii)] The {\em Hex plane} $\Omega=\Delta$ is the interior of 
an open $2$-simplex and $\SL(\Delta)$ consists of the semi-direct product
of positive diagonal matrices of determinant $1$ and permutations of the vertices. This is isometric to a normed
vector space, where the unit ball is a regular hexagon, \cite{Harpe}. Since the unit ball is not strictly convex
geodesics are not even locally unique. The minset of a hyperbolic is $\Delta$. Also $\SL(\Delta)$
 has index $2$ in $Isom(\Delta)$
\item[E(iii)]   $\Omega =D^2* \{p\}\subset{\mathbb R}P^3$ is the open cone on a 
round disc $D^2$. The restriction of the Hilbert metric to $D^2\times \{x\}\subset\Omega$ is E(i).  
Restricted to the cone  on a line in $D^2$ gives E(ii). There is an isomorphism
 $\SL(D^2* \{p\})\cong \Isom_+({\mathbb H}^2\times{\mathbb R})$; the latter is 
isometries which preserve the ${\mathbb R}$-orientation. A certain parabolic $A$  fixes a line $[p,x]$ in the boundary
where $x\in\partial\overline{D}$. The cone point $p$ is fixed by the subgroup $\Isom({\mathbb H}^2)$.
\item[E(iv)] {\em Real Siegel upper half space} $\Omega=Pos\subset {\mathbb R}^{n(n+1)/2}$
 is the projectivization of the open convex 
cone in $M_n({\mathbb R})$ of positive definite symmetric matrices. 
Points in $Pos$ correspond to 
homothety classes of positive definite quadratic forms, and points on the
boundary to positive semi-definite forms. The group $\SL(n,{\mathbb R})$ acts via 
 $B\mapsto A^t\cdot B\cdot A.$ 
Thus $SL(Pos)$ contains the image of the 
irreducible representation $\sigma_2:\SL(n,{\mathbb R})\longrightarrow \SL(n(n+1)/2,{\mathbb R})$. For $n=2$
this gives the hyperbolic plane E(i). For $n\ge 3$ this example shows
 there are many possibilities for the Jordan normal form of an element of $SL(\Omega)$ when $\Omega$
is  properly  but not strictly convex.
\end{itemize} 

\gap

If $p\in\overline{\Omega},$ then  $SL(\Omega,p)\subseteq SL(\Omega)$ is defined as the subgroup which fixes $p.$ It is easy to see that if $p \in \Omega$, then this group is compact, i.e.

\begin{lemma}[Elliptics are standard]
\label{elliptics}  If   $\Omega$ is a properly convex domain, then $A\in \SL(\Omega)$ is elliptic iff
it is conjugate in $SL(n+1,{\mathbb R})$  into $O(n+1)$. Furthermore, if $p\in\Omega,$ then $SL(\Omega,p)$ is conjugate in
$SL(n+1,{\mathbb R})$ into $O(n+1)$.\qed
\end{lemma}

Points  in projective space fixed by  $A\in SL(n+1,{\mathbb R})$ correspond to real eigenvectors
of $A.$ Thus the set of points in projective space fixed by $A$ is a finite set of disjoint projective subspaces, each of which is the projectivization of a real eigenspace. 

\begin{lemma}[invariant hyperplanes]\label{invarianthyperplane} 
If $\Omega$ is a properly convex domain and $A\in SL(\Omega)$ fixes a point $p\in\bOmega,$ then there is a supporting hyperplane $H$ to $\Omega$ at $p$ which is preserved by $A$.
\end{lemma}
\begin{proof} 
By \ref{normalcone} the set of hyperplanes which support $\Omega$ at $p$ is dual to a compact properly convex set, $C,$ in the dual projective space. By Brouwer, the dual action of $A^*$ fixes at least one point in $C$ and this point is dual to $H$.
\end{proof}

An immediate consequence of \ref{directioneigenvalues} that will be used in the study of elementary groups is:
\begin{lemma}\label{directionnothyperbolic} 
Suppose $\Omega\subset S^n$ is properly convex and $p\in\bOmega$. If $A\in SL(\Omega,p)$ is not hyperbolic, then the induced map $A_p\in\SL({\mathcal D}_p\Omega)$ on the space of directions is not hyperbolic.
\end{lemma}

The next step is to describe the fixed points in $\bOmega$ and the dynamics of a projective isometry.
By the Brouwer fixed point theorem the subset  $\Fix(A)\subseteq\overline{\Omega}$ of all points fixed 
by $A\in SL(\Omega)$ is not empty. 
 If $\Omega\subset S^n$ is properly convex and $A\in SL(\Omega)$ fixes
 a point in $\overline{\Omega}$ then the corresponding eigenvalue is {\em positive}.
Let $V_{\lambda}$ be the $\lambda$-eigenspace and $\Fix(A,\lambda)=\overline{\Omega}\cap {\mathbb P}(V_{\lambda}).$ This set   is either empty or compact and properly convex. Then $\Fix(A)=\bigsqcup_{\lambda} \Fix(A,\lambda)$  where $\lambda$ runs over the positive eigenvalues of $A$. 
 
The {\em $\omega$--limit set} $\omega(f,U)$ of the subset $U\subseteq X$ under $f:X\longrightarrow X$ is the union of the sets of accumulation points of the forward orbits $\{f^n(u):n>0\}$ of points $u\in U$. If $A\in \SL(\Omega)$ is not elliptic, then it generates an infinite discrete group. It follows from  \ref{discreteproperdiscts} that $A$ acts properly discontinuously on $\Omega,$ thus $\omega(A,\Omega)\subseteq\bOmega$.

 The $\omega$-limit set of {\em generic} points in projective space under $A\in\SL(n+1,{\mathbb R})$ is
 determined firstly by the eigenvalues of largest modulus and secondly by
 the Jordan blocks of largest size amongst these eigenvalues. 

Consider  the  dynamics of  $T\in GL(V)$ with a single Jordan block of size $\dim V=k+1$. 
Then  $T=\lambda\cdot(I+ N)$  with
  $N^{k+1}=0$ and 
  $N^k\ne0.$ For $p\ge k$ 
 $$T^p=\lambda^p(I+N)^p=\lambda^p\left[1 + \left(\begin{array}{c} p\\1\end{array}\right) N  + \left(\begin{array}{c} p\\2\end{array}\right)
 N^2 +\cdots + \left(\begin{array}{c} p\\k\end{array}\right)N^k\right]$$
 For $p$ large the last term dominates.  Let $e_{k+1}\in V$ be a cyclic vector for the ${\mathbb R}[T]$-module $V$. 
 This gives a
   basis $\{e_1,\cdots, e_{k+1}\}$ of $V$ with $e_i=N(e_{i+1})$ for $1\le i\le k$ and $N(e_1)=0$.
Observe that $T$ has a one-dimensional eigenspace  $E={\mathbb R}e_1$. 
Define a polynomial $h(t)=(t-\lambda)^k,$ then $E=\im\ h(T)$ is the eigenspace and
 $K=\ker N^k=\ker h(T)$ is the unique proper invariant subspace of maximum dimension.
   Call a point $x\in {\mathbb P}(V)$ {\em generic} if it is not in the hyperplane ${\mathbb P}(K)$.
 If  $x$ is generic, then $T^px\to {\mathbb P}(E)$ as $p\to\infty$. 
 Thus $\omega(T,{\mathbb P}(V)\setminus{\mathbb P}(K))={\mathbb P}(E)$ is a single point. 
 
 If instead $T$ has Jordan form $(I+re^{i\theta} N)\oplus (I+re^{-i\theta}N),$
 similar reasoning shows there is a projective line ${\mathbb P}(E)$ on which $T$ acts by rotation by 
 $2\theta$ and generic points  converge
 to this line under iteration.  In fact using the definitions of $E$ and $K$ above but with the
  polynomial $h(t)=(t^2-2tr\cos\theta+r^2)^k$ one obtains similar conclusions.  As before, generic points
  are those not in the codimension-2 hyperplane ${\mathbb P}(K)$. Now for the general case.
 
To a $k\times k$ Jordan  block $\lambda I+N$ with eigenvalue $\lambda$ assign the ordered pair $(|\lambda|,k),$
called the {\em power} of the block. Two Jordan blocks with the same power are called {\em power equivalent}.
Lexicographic ordering of these pairs is an ordering  on power equivalence classes of
 Jordan block matrices. Given a linear map $T\in GL(V)$ the {\em power} of $T$ is
 the maximum of the powers of the Jordan blocks of $T$. If the power of $T$ is larger than the power
 of $S,$ we say $T$ is {\em more powerful} than $S$.    The {\em spectral radius} $r(T)$  is the maximum modulus of the eigenvalues of $T$.
 
The power of $T\in GL(V)$ is $(r(T),k),$ where $k\ge 1$ is the size of the most powerful blocks.
Let $p(t)$ be the characteristic polynomial of $T$. Let ${\mathcal E}$ be the set of eigenvalues of Jordan blocks
of maximum power in $T$ and set $q(t)=\prod_{\lambda\in{\mathcal E}}(t-\lambda)$.
Observe that the linear factors of $q(t)$ are all {\em distinct} and that $q(t)$ has real coefficients.
Define $h_T(t)=h(t)=p(t)/q(t)$ and  two  linear subspaces 
 $E=E(T)=\im\ h(T)$ and $K=K(T)=\ker h(T)$.
The next proposition implies that  points in ${\mathbb P}(V)\setminus {\mathbb P}(K)$ limit on ${\mathbb P}(E)$ under forward iteration of $[T]$.

\begin{lemma}[power attracts]\label{powerattracts} Suppose $T\in GL(V)$ and
 $W\subseteq{\mathbb P}(V)\setminus {\mathbb P}(K)$
has nonempty interior.
Then   $\omega([T],W)$ 
is a subset of ${\mathbb P}(E)$ with nonempty interior. Moreover, the action of $T$ on ${\mathbb P}(E)$ is conjugate
into the orthogonal group.
\end{lemma}
\begin{proof}[Sketch proof] Extend $T$ to $T_{\mathbb C}$ over 
$V_{\mathbb C}=V\otimes_{\mathbb R}{\mathbb C}$.
Take the Jordan decomposition of $T_{\mathbb C}=\bigoplus T_i$ corresponding to an invariant
 decomposition $V_{\mathbb C}=\bigoplus V_i$. Use the analysis above in each block. After
projectivizing only the most powerful blocks
 contribute to the $\omega$--limit. The subspace $K\otimes{\mathbb C}$ contains 
 those $V_i$
for blocks that do not have maximum power. It also contains the maximal proper invariant
subspace of those $V_i$ for each Jordan block of maximum power. 
The subspace $E\otimes{\mathbb C}$ is the space spanned
by the eigenvectors from the most powerful blocks. The action of $T$ on this subspace is diagonal with
eigenvalues $re^{i\theta}$ with $r=r(T)$ fixed but $\theta$ varying. \end{proof}

  \begin{proposition}\label{realpower} If $\Omega$ is properly convex and $T\in\SL(\Omega)$ is not elliptic then 
   $T$ has a most powerful Jordan block with real eigenvalue $r=r(T)$ and $\Fix(T,r)\subseteq\bOmega$ is  nonempty. Furthermore, if $\Omega$ is strictly convex, then $T$ contains a 
unique Jordan block of maximum power.
  \end{proposition}
  \begin{proof} Set $K=K(T)$ and $E=E(T)$.
   By \ref{powerattracts}  $H_+=\omega([T],\Omega\setminus{\mathbb P}(K))\subseteq{\mathbb P}(E)$ contains a nonempty
    open subset
  of ${\mathbb P}(E)$. The $\omega$--limit set of $\Omega$ is in $\bOmega$ 
  so $H_+\subseteq \bOmega$ hence 
  $G=\overline{\Omega}\cap{\mathbb P}(E)\supset H_+$ is
  a nonempty, compact convex set preserved by $T$. By the Brouwer fixed point theorem $T$ fixes some point 
  in $G$. This corresponds to an eigenvector with positive eigenvalue that is maximal, and is therefore $r$. Hence $\Fix(T,r)$ is not empty. Since $T$ is not elliptic $F=\Fix(T,r)\subseteq\bOmega$.  
   
 The number of Jordan blocks of maximum power is $\dim E$.   Since  $H_+$ contains an open set
 in ${\mathbb P}(E)$, if $\dim E>1$, then it contains a nondegenerate interval. But $H_+\subseteq\bOmega$ hence $\Omega$ is not strictly convex. 
    \end{proof}

If $A$ is hyperbolic, then $r(A)>1$ and the points in $F_+(A)=\Fix(A,r(A))$ are called  {\em attracting} fixed points and
  are represented by  eigenvectors with  eigenvalue $r(A)$. Similarly, points
  in $F_-(A)=F_+(A^{-1})$ are {\em repelling} fixed points. The union of the remaining sets $\Fix(A, \lambda)$ is denoted $F_0(A).$

\begin{proposition}  
\label{fixedpoints}
Suppose $\Omega$ is a properly convex domain and $A\in SL(\Omega)$.  
\begin{enumerate}
\item If $A$ is parabolic or elliptic then $\Fix(A)=\Fix(A,1)$ is convex.
\item If $A$ is hyperbolic then $\Fix(A)= F_+(A)\sqcup F_-(A)\sqcup F_0(A)$ and $F_{\pm}(A)$ are nonempty compact convex sets. In particular, $\Fix(A)$ is not connected.
\end{enumerate}
\end{proposition}

{\bf Example} 
Referring to E(iii)  consider the hyperbolic  $A\in\SL(D^2*\{p\})$ which is the composition of a rotation by $\theta$ in $D^2$ together
with a hyperbolic given by $\diag(2,2,2,1/8)$ which moves points towards $D^2$ and away from $p$. The forward and backward $\omega$--limits sets are $H_+=D^2$
and $H_-=p$. There is a unique fixed point $F_+(A)$ in $D^2$: the center of the rotation. 
 
A real matrix with unique eigenvalues of
maximum and minimum modulus is {\em positive proximal} \cite{benoistsurvey} if  these eigenvalues are positive.
\begin{proposition}[strictly convex isometries]\label{fixedpts} Suppose $\Omega$ is a strictly convex domain and
$A\in SL(\Omega)$. 
If $A$ is parabolic, it fixes precisely one point in $\bOmega$. If $A$ is hyperbolic, it  is positive proximal
and fixes precisely two points in $\bOmega$. The line segment in $\Omega$ with these endpoints is called the {\bf axis} and consists of all points moved distance $t(A)$.
\end{proposition} 

\begin{proof} 
Each $\Fix(A,\lambda)$ is a single point because $\Omega$ is strictly convex. The result for parabolics now follows from \ref{fixedpoints}. Otherwise for a hyperbolic $F_-=[v_-]$ and $F_+=[v_+]$ are single points. 

The eigenvectors $v_{\pm}$ have eigenvalues $\lambda_{\pm}$ of maximum and minimum modulus. 
By \ref{invarianthyperplane} there are invariant
 supporting hyperplanes  $H_{\pm}$
to $\Omega$ at these points. Since $\Omega$ is strictly convex, these hyperplanes are distinct so that their intersection
is a codimension-2 hyperplane. Thus $A$ preserves a codimension-2 linear subspace that contains neither $v_{\pm}$.
 It follows that the corresponding Jordan blocks have size $1$. By \ref{realpower} the most powerful block is
 unique, so the eigenvalue $\lambda_+$ has algebraic multiplicity one. The same remarks apply to $\lambda_-$ because  $A^{-1}$ 
is also hyperbolic. Thus $A$ is positive proximal.

The line segment $[v_-,v_+]\subseteq \overline\Omega$ 
 meets $\bOmega$ only at its endpoints and $A$ maps this segment to itself. The restriction of $A$ to the two dimensional subspace spanned by $v_{\pm}$ is 
given by the diagonal matrix $\diag(\lambda_+,\lambda_-).$  The action of $A$ on this segment is translation 
by a Hilbert distance of $\log(\lambda_+/\lambda_-).$ It follows from \ref{isometryfoliation} that points not on this axis are moved a larger distance (the discussion up to and including 2.11 does not use this characterisation of the axis).
\end{proof}
 \noindent{\bf Example} (A hyperbolic with no axis) The domain
$\Omega=\{(x,y): xy>1\}$ is projectively equivalent to a properly convex subset of the Hex plane $\Delta$.
There is $A\in\SL(\Omega)$ given by $A(x,y)= (2x,y/2)$ with translation length $\log 4$ which is not attained, 
so the minset is empty.\gap

\noindent{\bf Examples of Parabolics} Every 1-parameter subgroup of parabolics in $SO(2,1)$ is conjugate to
 $$\left(\begin{array}{ccc}1 & t & t^2/2\\0 & 1 & t\\0 & 0 & 1\end{array}\right).$$
 The orbit of $[0:0:1]$ is the affine curve in ${\mathbb R}P^2$ given by $[t^2/2:t:1].$ The 
completion of this curve is a projective quadric. One may regard this as the boundary of the parabolic 
model  $\{\ (x,y)\ :\ x>y^2/2\ \}\subseteq  {\mathbb R}^2$ of the hyperbolic plane (see later).
  
  \gap
The {\em index} $i_A(\lambda)$ of an eigenvalue $\lambda$ is the size of the largest Jordan block for
  $\lambda$. This equals the degree of the factor $(t-\lambda)$ in the minimum polynomial of $A$. If
  $\lambda$ is not an eigenvalue of $A,$ then define $i_A(\lambda)=0$.  
  The {\em maximum index} of $A$ is $i_A=\max_{\lambda} i_{A}(\lambda)$.
   Every element  $A\in O(n,1)$ is conjugate into $O(n-2)\oplus O(2,1)$.
If $A$ is parabolic, then  $i_A=i_A(1)=3$  and all other eigenvalues are semisimple. 
 
For  the  Siegel upper half space, we have $\SL(Pos)\supset\sigma_2\left(SL(n,{\mathbb R})\right)$. The image of a matrix given by
 a single Jordan  block of size $n$  contains one Jordan block of each  of the sizes $2n-1,2n-5,\cdots, 3$ or $1$.
 In particular, a unipotent matrix of this type gives a parabolic $A$ with $i_A=i_A(1)=2n-1$.
 
 As a final example let $N$ denote a nilpotent $3\times 3$ matrix with $N^2\ne 0$ 
 so that $$B=(I+N)\oplus e^{i\theta}(I+N)\oplus e^{-i\theta}(I+N)\in GL(9,{\mathbb C})$$ is the Jordan form of an element
 $A\in SL(9,{\mathbb R})$ with $i_A=i_A(1)=i_A(e^{\pm i\theta})=3$.
  Then $E=E(A)$ is a $3$-dimensional invariant subspace. The action of $A$ on $E$
 is  rotation by $\theta$ around an axis corresponding to the real eigenvector for $A$.
 The image of the axis is the unique fixed point $x\in{\mathbb R}P^8$ for the action of $A$.
 The set ${\mathbb P}(E)\subseteq{\mathbb R}P^8$  is the $\omega$--limit set for $A$. 
 The convex hull of the orbit of a suitable small open set near $x$ disjoint from ${\mathbb P}(E)$ is a properly convex set $\Omega$ preserved by $A$.
 Under iteration points in $\Omega$ converge to ${\mathbb P}(E)$ so that
 $\overline{\Omega}\cap{\mathbb P}(E)$ is a small 2--disc centered on $x$ which is rotated by $A$.
 In particular, $\Omega$ is not strictly convex.
 
   \begin{proposition}[JNF for parabolics]
 \label{parabolicchar} Suppose $\Omega$ is a properly convex domain and $T\in \SL(\Omega)$   
 is a parabolic. Then there is a Jordan block of maximum power with eigenvalue $1$
 and the block size $i_T(1)\ge 3$ is odd. If $\Omega$ is
 strictly convex, this is the only block of maximum power. 
\end{proposition}
 \begin{proof} Except for the statement concerning $i_T(1)$ this follows from \ref{realpower}.
First consider the case that $T=I+N$ consists of a single Jordan block 
 of size $n+1$. Then $N^n\ne0$ and $N^{n+1}=0$.
 Using a suitable basis $[0:0:\cdots:1]\in\Omega,$ and the image of 
$(0,0,\cdots,1)$ under $(I+N)^p$ is
$(x_0,x_1,\cdots, x_n)=(1 ,\left(\begin{array}{c} p\\1\end{array}\right), \left(\begin{array}{c} p\\2\end{array}\right),\cdots , 
\left(\begin{array}{c} p\\n\end{array}\right))$ provided $p\ge n$.

For $p$ large $x_n$ dominates. If $n$ is odd the sign of $x_n$ is the sign of $p$. Thus
as $p\to\pm\infty$ this implies $(0,\cdots,0,\pm1)\in\partial\Omega$. 
These are antipodal points in $S^n$ and contradict that $\Omega$ is strictly convex.
Hence $n$ is even so $i_T(1)$ is odd. If $i_T(1)=1$ then every eigenvalue
of $T$ has multiplicity $1$ thus $T$ is  elliptic. Hence $i_T(1)\ge 3$. 
This argument is simpler than the original and is credited to Benoist 
by Crampon and Marquis.

For the general case choose $[v]\in\Omega$ and let $V\subseteq{\mathbb R}^{n+1}$ be the
cyclic ${\mathbb R}[T]$-module generated by $v$. Then $T|V$ has a single
Jordan block. By choosing $v$ generically it follows that $\dim V$ is the size of a largest
Jordan block of $T$. Furthermore $\Omega'=\Omega\cap {\mathbb P}(V)$ is a nonempty, properly convex open set,
that is preserved by $T$.  The result  follows from the special case.
\end{proof}

\begin{corollary}[low dimensions]\label{parabolic2and3} Suppose $A\in \SL(n+1,{\mathbb R})$ is a 
parabolic for a properly convex domain.
If $n=2$ or $3$ then $A$ is conjugate into $O(n,1)$. If $n=4$ then
$A$ is conjugate into $O(4,1)$ or $O(2,1)\oplus \SL(2,R)$.
\end{corollary}

Using this, with a bit of work one can show that in dimension $3$ a rank-2 discrete
free abelian group consisting of parabolics for a properly convex domain is conjugate into $O(3,1)$.
However, in dimension $3$ there is a rank-2 free abelian group $\Gamma$ with the property that
every non trivial element of $\Gamma$ is a parabolic for {\em some} properly convex domain,
but $\Gamma$ is not conjugate into $O(3,1).$ 
\gp
If $C$ is a codimension-2 projective subspace then the set of codimension-1 projective hyperplanes
containing $C$ is called a {\em pencil of hyperplanes} and $C$ is the {\em center} of the pencil. The hyperplanes in the pencil
are dual to a line $C^*$ in the dual projective space. The next result gives a good picture of the dynamics of
a projective isometry.
\begin{proposition}[isometry permutes pencil]\label{isometryfoliation} 
Suppose that $\Omega$ is a properly convex domain and $A\in \SL(\Omega)$ is 
a parabolic or hyperbolic. 

Then there is a pencil of  hyperplanes that is preserved by $A.$ 
The intersection of this pencil with $\Omega$ is a foliation and no leaf is stabilized by $A.$ 
Thus  $M=\Omega/\langle A\rangle$ is a  bundle over the circle  with fibers subsets of hyperplanes. 
\end{proposition}

\begin{proof} The desired conclusion is equivalent to the  existence of a projective line $C^*$ 
in the dual projective space with the properties
\begin{enumerate}
\item $C^*$ is preserved by the dual action of $A$, and this action on $C^*$ 
is non-trivial;
\item $C^*$ intersects the closure of the dual domain  $\overline{\Omega^*}.$ 
\end{enumerate}

The reason is that a hyperplane $H$ meets $\Omega$ if and only if the dual point $H^*$ is disjoint from $\overline{\Omega^*}$. Thus the condition that $C^*$ meets $\overline{\Omega^*}$ ensures that the center, $C$, of the pencil does not intersect $\Omega$, which  in turn ensures the hyperplanes foliate $\Omega.$
 
\begin{figure}[ht]	 
\begin{center}
	 \psfrag{H0}{$H_+$}
	 \psfrag{H}{$H$}
	 \psfrag{H1}{$H_-$}
	 \psfrag{p+}{$p_+$}
	 \psfrag{p-}{$p_-$}
	\psfrag{S}{$C$}
	\psfrag{ell}{$\ell$}
	 \psfrag{p}{$C$}
	 \psfrag{Om}{$\Omega$}
		 \psfrag{hyp}{Hyperbolic}
	 \psfrag{par}{Parabolic}
	 \psfrag{Pt}{$P_t$}
	 \includegraphics[scale=0.6]{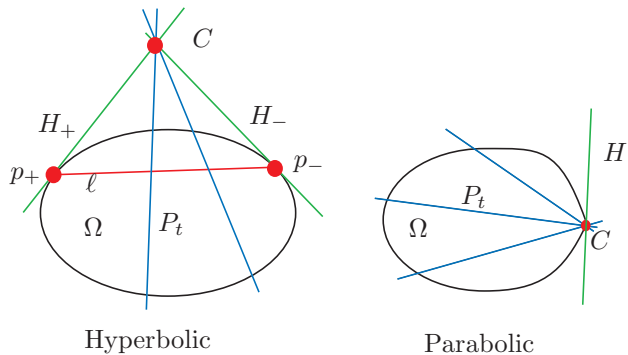}
\end{center}
  \caption{Pencils of hyperplanes}\label{pencil}
\end{figure}

First consider the case that $A$ is hyperbolic. Then there are distinct points $H_{\pm}^*\in\partial\overline{\Omega^*}$
 which are respectively an attracting and a repelling fixed point for the dual action of $A^*.$ In this case 
 we may choose $C^*$ to be the line containing these points. 
 The points $H_{\pm}^*$ are dual  to 
supporting hyperplanes $H_{\pm}$ to $\Omega$ at some attracting and repelling fixed points.

The second case is that $A$ is parabolic. In this case $i_{A^*}(1)=i_A(1)\ge 3$.
 There is a 2-dimensional invariant subspace $V^*$ in the dual projective space 
 coming from a  Jordan block of size $i_{A^*}(1)$ with eigenvalue $1$ for $A^*$ and the restriction of $A^*$ to this subspace is a non-trivial parabolic in $\SL(2,{\mathbb R})$. 
We may choose $V^*$ so that the projective line $C^*={\mathbb P}(V^*)$  
contains a parabolic fixed point $H^*$ in $\bOmega^*.$ This is dual to a supporting 
hyperplane, $H,$ to $\Omega$ at some parabolic fixed point $p$ which is preserved by $A.$
\end{proof}

From this and \ref{divergeline} it  easily follows that:
\begin{corollary}\label{unboundeddisplacement} If $\Omega$ is strictly convex and $A\in\SL(\Omega)$ is not elliptic, then $f(x)=d_{\Omega}(x,Ax)$
is not bounded above.
\end{corollary}

\begin{proof}[Proof of \ref{translationlength}]
If $A$ is elliptic, then $t(A)=0$ and the result follows from \ref{elliptics}. The parabolic case follows from Lemma \ref{nonhypsmalldist}. The hyperbolic case follows from \ref{isometryfoliation}. The pencil gives an $A$--equivariant projective map of $\Omega$ onto the interval $[H_-^*,H_+^*]\subseteq C^*$. There is a Hilbert metric on this interval. The projection is distance non-increasing. The action of $A$ on the interval is translation by $\log(\lambda_+/\lambda_-)$. The result follows.
 
We remark that in the case $\Omega$ is strictly convex, there is a natural identification of this interval with the axis, $\ell$, of $A$ in $\Omega$ and the projection corresponds to projection along leaves of the pencil onto this axis.
\end{proof}

\begin{proposition}\label{characterizeparabolics} Suppose $\Omega$ is a properly convex domain and $A\in SL(\Omega,p)$ is not elliptic. The following are equivalent:
\begin{itemize}
\item $A$ is parabolic,
\item every eigenvalue has modulus $1$,
\item every eigenvalue has modulus $1$ and the eigenvalue $1$ has largest index, which is odd $\ge3$,
\item the translation length $t(A)=0$ (see Lemma \ref{nonhypsmalldist}),
\item the subset of $\bOmega$ fixed by $A$ is non-empty, convex and connected,
\item $A$ preserves some horosphere (see Proposition \ref{Hphorosphere}).
\end{itemize}
\end{proposition}
\section{Horospheres}\label{horospheresection}

Given a ray $\gamma$ in a path metric space $X$  Busemann \cite{Bus55} defines
a  function $\beta_{\gamma}$ on $X$ and a horosphere
to be a level set of $\beta_{\gamma}.$
 We consider this for  the Hilbert metric on a properly convex domain $\Omega$.
If $\gamma$ converges to a $C^1$ point  $x\in\partial\overline{\Omega},$
then these horospheres depend only on $x$ and not on the choice of $\gamma$
converging to $x$. This is the case for hyperbolic space ${\mathbb H}^n$, but
in general horospheres depend on the choice of $\gamma$ converging to $x$.
See Walsh \cite{hilberthoro} for an extensive discussion.

 {\em Algebraic horospheres} are defined below.  These
 coincide with Busemann's horospheres at $C^1$ points. 
We will subsequently refer to the latter
as {\em Busemann-horospheres} and the term {\em horosphere} will henceforth mean algebraic horosphere.
Of course the convention will be applied to horoballs and to all {\em horo} objects: they refer
to the algebraic definitions below.

It turns out that every
parabolic preserves certain horospheres and these are
used to foliate cusps in section \ref{sectioncusps}. The construction
depends on both $x$ and a choice of {\em supporting hyperplane} $H$ to $\Omega$ at $x$
rather than a choice of ray $\gamma$. 

 Let  $\widetilde{H}$ be a codimension-1 vector subspace of 
${\mathbb R}^{n+1}$ and  $\tilde{p}\in\widetilde{H}$ a non-zero vector.
Let $p\in H\subset S^n$   be their images under projection. Define
 $\SL(H,p)$ to be the subgroup $\SL(n+1,{\mathbb R})$   which preserves both $H$ and $p$.
 This is the subgroup of the affine group $\Affn$ which preserves a {\em direction}.
  Given  $A\in SL(H,p)$ let $\lambda_+(A)$ be the eigenvalue for the eigenvector $\tilde{p}$.
 If $\tilde{v}\in{\mathbb R}^{n+1}\setminus \widetilde{H},$ then $A\tilde{v}+\widetilde{H}=\lambda_-\tilde{v}+\widetilde{H}$ and $\lambda_-=\lambda_-(A)$
 is another eigenvalue of $A$ which does not depend on the choice of $\tilde{v}$.
There is a homomorphism $\tau:SL(H,p)\longrightarrow({\mathbb R}^*,\times)$ given by
$$\tau(A)=\lambda_+(A)/\lambda_-(A).$$
 
 \noindent Define the subgroup 
$\GG=\GGHP\subset \SL(H,p)$ to be those elements $A \in \SL(H,p)$ which satisfy:
\begin{enumerate}
\item $A$ acts as the identity on $\widetilde{H}$, and
\item $A(\ell)=\ell$ for every line $\ell$ in ${\mathbb R}P^n$ which contains $p.$
\end{enumerate}
Notice that (1) and (2) imply:
\begin{enumerate}\setcounter{enumi}{2}
\item $A$ acts freely on $\ell \setminus \{ p\}$.
\end{enumerate}

It is clear that in fact $\GG$ is a normal subgroup of $\RR$. Moreover,  all elements of $\GG$ have
the form $Id + \phi\otimes\tilde{p}$, where  
$\phi \in ({\mathbb R}^{n+1})^*$ and $\phi(\widetilde{H})=0$. Suppose $\ell$ is a line containing $p$ that is not contained in
$H$. Then $\GG$ acts by parabolics on $\ell$ fixing $p$. Denoting $\Par(\ell,p)$ the group of parabolic transformations of $\ell$ fixing $p,$ this gives an isomorphism $\GG\longrightarrow \Par(\ell,p).$
Since $\GG\cong \Par(\ell,p)\cong({\mathbb R},+),$ it follows
that there is a canonical identification $Aut(\GG)\equiv({\mathbb R}^*,\times)$.

\begin{proposition}\label{Gnormal} The action by conjugacy
of $SL(H,p)$ on the normal subgroup $\GGHP$ is given by 
$\tau:SL(H,p)\longrightarrow Aut(\GGHP)\equiv({\mathbb R}^*,\times)$.
\end{proposition}

In the sequel we assume $\Omega$ is a properly convex domain, $p\in\partial\overline{\Omega}$ and
$H$ is a supporting hyperplane to $\Omega$ at $p$. Define ${\mathcal S}_0\subset\partial\overline{\Omega}$
to be the subset of $\partial\overline{\Omega}$ obtained by deleting $p$ and all line segments in 
 $\partial\overline{\Omega}$ with one endpoint at $p$.  Thus
${\mathcal S}_0$ satisfies the {\em radial condition} that ${\mathcal D}_p|{\mathcal S}_0$ is a homeomorphism
onto ${\mathcal D}_p\Omega$.  If $p$ is a strictly 
convex point of $\partial\overline{\Omega},$ then ${\mathcal S}_0=\partial\overline{\Omega}\setminus p$.
 A {\em generalized horosphere centered on $(H,p)$} is the image of ${\mathcal S}_0$ 
under an element  $\GGHP$.
An {\em algebraic horosphere} or just {\em horosphere} is a generalized horosphere  contained in $\Omega$.
Property (3) implies $\Omega$ is foliated by horospheres.  
Similarly, a {\em generalized horoball centered on $(H,p)$} is the image of ${\mathcal B}_0=\Omega\cup{\mathcal S}_0$ under an element of $\GGHP$ and an {\em algebraic horoball} or just {\em horoball} is a generalized horoball contained in $\Omega.$

Parabolics preserve certain horospheres: If $A\in\SL(\Omega,p)$ is parabolic,
then by \ref{invarianthyperplane} it preserves some supporting hyperplane $H$ at $p$.
Define $SL(\Omega,H,p)=\SL(\Omega)\cap \SL(H,p)$. Then $A\in\SL(\Omega,H,p)$.
Observe that if $p$ is a $C^1$ point of $\partial\overline{\Omega}$
then $H$ is unique and $\SL(\Omega,H,p)=SL(\Omega,p)$. 

Since $\SL(\Omega,H,p)$ preserves $\partial\overline{\Omega}$ it
also preserves the foliation of $\Omega$ by horospheres.  
For $A\in\GGHP$ define the horosphere ${\mathcal S}_A=A({\mathcal S}_0)$. The element $B\in \SL(\Omega,H,p)$ acts on horospheres by
$$B({\mathcal S}_A)=BA({\mathcal S}_0)=BAB^{-1}(B{\mathcal S}_0)=BAB^{-1}({\mathcal S_0})={\mathcal S}_{BAB^{-1}}$$

Choose an isomorphism from $({\mathbb R},+)$ to $\Par(\ell,p)$ given by
$t\mapsto A_t$ and define $${\mathcal S}_t = A_t({\mathcal S})$$  
This isomorphism can be chosen so that ${\mathcal S}_t\subset\Omega$ for all $t>0$.
Then  the horoball ${\mathcal B}_t=\cup_{s\ge t}{\mathcal S}_s$  is a union of horospheres, and
 $\partial{\mathcal B}_t={\mathcal S}_t$. Combining these remarks with \ref{Gnormal}:

\begin{proposition}\label{Hphorosphere2} If $B\in \SL(\Omega,H,p),$ then $B({\mathcal S}_t)={\mathcal S}_{\tau(B)t}$.
\end{proposition}

The {\em horosphere displacement function} is the homomorphism
$$h:\SL(\Omega,H,p)\longrightarrow({\mathbb R},+)$$ given by $h(B)=\log\tau(B).$

\begin{proposition}\label{Hphorosphere} Suppose $B\in \SL(\Omega,H,p)$.
If $B$ is elliptic or parabolic, then $h(B)=0$ and $B$ preserves every  generalized 
horosphere centered on $(H,p)$.
 If $B$ is
hyperbolic and $\Omega$ is properly convex, then $h(B)=\pm t(B)$ is the {\em signed translation length}
with the  $+$ sign iff $B$ translates {\em towards} $p$.
\end{proposition}
\begin{proof} If every eigenvalue of $B$ has modulus $1,$ then $\tau(B)=1$ which this gives the result
for elliptics and parabolics. Suppose $B \in  \SL(\Omega, H,p)$ is  hyperbolic and $\tilde{p}$ 
is an eigenvector with 
largest eigenvalue $\lambda_+$ so that $B$ translates {\em towards} $p$.
The other endpoint $q \in \partial\overline{\Omega}$ of the axis of $B$ corresponds to 
the eigenvalue of smallest modulus $\lambda_-$ and since $\tilde{q}\notin\widetilde{H}$
from the definition of $\tau$ we see that $\tau(B)=\lambda_+/\lambda_-$. The formula
for translation length \ref{translationlength} completes the proof.
\end{proof}

This is most easily understood using   {\em parabolic coordinates} on a properly convex open 
set $\Omega$ described below.
This is done for the Klein model of hyperbolic space in \cite{wpt1}  2.3.13.
Choose another point $r\in\partial\overline{\Omega}$ such 
that the interior of the segment $[p,r]$ is in $\Omega.$ Let  $H_r\subset\RPn$ be some supporting hyperplane at $r$,
and for clarity let $H_p\subset\RPn$ denote $H$.  
 Identify the affine patch $\RPn\setminus H_p$  with ${\mathbb R^n}$ so 
that $p$ corresponds to the direction given by the $x_n$ axis and so that $r$ is the origin and $H_r$ is the 
hyperplane $x_n=0.$ These are called {\em parabolic coordinates centered on $(H,p)$.}

\begin{figure}[ht]	 
\begin{center}
	 \psfrag{Hp}{$H_p$}
	 \psfrag{p}{$p$}
	 \psfrag{Hq}{$H_r$}
	 \psfrag{q}{$r$}
	 \psfrag{Om}{$\Omega$}
		 \psfrag{pom}{${\color{darkgreen}{\mathcal S}_t=\partial\overline{\Omega}+te_{n}}$}
	 \psfrag{tp}{$ $}
	 \psfrag{q0}{$r=0$}
	 \psfrag{xn}{$x_{n }$}
	 \includegraphics[scale=0.7]{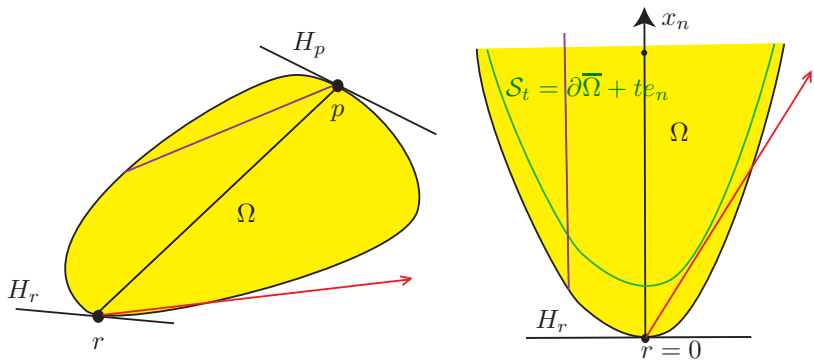}
\end{center}
  \caption{Horospheres.}\label{horoballs}
\end{figure}

In these coordinates, \emph{rays in $\Omega$ converging to $p$} are the {\em vertical} rays parallel to the $x_n$ axis.
Radial projection ${\mathcal D}_p$ from $p$ corresponds to vertical projection onto $H_r$.
An element $A\in \SL(\Omega,H,p)$ acts affinely on
this affine patch sending vertical lines to vertical lines.
The generalized horosphere ${\mathcal S}_0\subset\bOmega$ is the subset of
$\partial\overline{\Omega}\cap{\mathbb R}^n$ obtained by deleting all vertical line segments in $\bOmega$.
There are no such segments if $p$ is a strictly convex point of $\bOmega$.
The horosphere ${\mathcal S}_0$ is  the graph of a continuous convex 
function $h:U\longrightarrow{\mathbb R}_+$ 
defined on an open  convex subset   $U\subset H_r.$ Observe that ${\mathcal D}_pU\cong{\mathcal D}_p\Omega$
and $U=H_r$ iff $p$ is a $C^1$ point.
 
The positive $x_n$-axis is contained in 
$\Omega.$ Rays contained in $\overline{\Omega}$ starting at $r$ correspond to points of $H_p\cap\partial\overline{\Omega}.$ 
If $\Omega$ is strictly convex at $r$ then the positive $x_n$-axis  is the unique  ray in $\Omega$ 
starting at $r.$  Let ${\bf e}_n$ denote a vector in the direction of the $x_n$ axis.
There is an isomorphism 
$({\mathbb R},+)\cong\GGHP$ given by  $t\mapsto A_t$ so that the action of the group $\GG$ on ${\mathbb R}^n$ is by 
{\em vertical translation}  $A_t(x)= x+t e_n$. Then in parabolic coordinates horospheres are given
by translating ${\mathcal S}_0$ vertically upwards:
$${\mathcal S}_t={\mathcal S}_0+te_n.$$

\begin{proposition}
 \label{horosphereprops} Suppose $\Omega$ is properly convex and $H$ is a supporting 
 hyperplane to $\Omega$ at $p$.
 In what follows horoballs and horospheres are always understood in the algebraic sense and centered on $(H,p)$, and:
  \begin{itemize}
 \item[(H1)] Radial projection ${\mathcal D}_p$ is a homeomorphism from a horosphere to the open ball ${\mathcal D}_p\Omega$.
 \item[(H2)] Every horoball is convex and homeomorphic to a closed ball with one point removed from the boundary.
\item[(H3)]  The boundary of a horoball is a horosphere. 
 \item[(H4)] If $\Omega$ is strictly convex at $p$ then each horoball limits 
 on only one point in $\partial\overline{\Omega},$ the center of the horoball.
 \item[(H5)] The horospheres centered on $(H,p)$  foliate $\Omega.$
 \item[(H6)] The rays in $\Omega$ asymptotic to $p$ give a transverse foliation ${\mathcal F}$.
 \item[(H7)] If $p$ is a $C^1$ point  and $x(t),x'(t)$ are
 two vertical rays parameterized so $x(t),x'(t)$ are both
  on ${\mathcal S}_t$  then $d_{\Omega}(x(t),x'(t))\to0$ monotonically as $t\to\infty$.
  \item[(H8)] The distance between two horospheres is constant and equals the
 Hilbert length of every arc in a leaf of ${\mathcal F}$ connecting them.
\end{itemize}
\end{proposition}
\begin{proof} 
These statements follow by considering parabolic coordinates.
\end{proof}

We compare this to the classical geometrical approach to Busemann-horospheres using Busemann functions. To this
 {\em\color{blue} end}, let $\gamma:[0,\infty)\rightarrow\Omega$  be a projective 
 line segment in $\Omega$ parameterized by arc 
length and so that $\lim_{t\to\infty}\gamma(t)=p.$ 
The {\em Busemann function} $\beta_{\gamma}:\Omega\longrightarrow{\mathbb R}$ is
 $$\beta_\gamma(x)=\lim_{t\to\infty}\left(d_{\Omega}(x,\gamma(t))-t\right)$$
The limit exists because $d_{\Omega}(x,\gamma(t)) - t $ is a non-increasing function of $t$
that is bounded below. It is easy to see that 
$$|\beta_{\gamma}(x)-\beta_{\gamma}(x')|\le d_{\Omega}(x,x')\qquad\qquad\text{and}\qquad\qquad \lim_{x\to p} \beta_{\gamma}(x)=-\infty$$
 Suppose that $p\in\partial\overline{\Omega}$ is a $C^1$ point.  
 If two rays converge to $p$ then approaching $p$ the distance between them goes to zero.
 It follows that the Busemann functions they define differ
only by a constant.  
In this case the level sets of $\beta_{\gamma}$ are algebraic horospheres:
\begin{lemma} 
\label{busemanncoords}
Suppose that $p$ is a $C^1$ point and $\gamma$ is a ray in 
$\Omega$ asymptotic to $p$. Then in parabolic coordinates   the level sets of $\beta_{\gamma}$ are 
$(\partial\overline{\Omega}\cap{\mathbb R}^n)+t e_n$ for $t>0$. Furthermore $|\beta_{\gamma}(q)-\beta_{\gamma}(r)|$
is the minimal Hilbert distance between points on the horospheres containing $q$ and $r$.
\end{lemma}
\begin{proof}   
There are parabolic coordinates so that $\gamma(t) = e^t {\bf e}_n$.
 Suppose $q \in \Omega$ is not on the $x_n$-axis. Let $y$ be the point on $\partial\overline{\Omega}$ vertically below $q$. The straight line  $\ell$ through $\gamma(t)$ and $q$ has two intercepts on $\partial\overline{\Omega}$; denote the intercept on the $q$
side by $k(t)$ and the other by $\tau(t)$. See Figure \ref{horocalc}.

\begin{figure}[ht]
 \begin{center}
\psfrag{gx}{$g_x$}
\psfrag{ki}{$y$}
\psfrag{pki}{$y_n$}
\psfrag{th}{\small$\theta$}
\psfrag{pgt}{$0$}
\psfrag{taut}{$\tau(t)$}
\psfrag{ptaut}{$e^{t+s}$}
\psfrag{gt}{$\gamma(t)=e^t$}
\psfrag{kt}{$k(t)$}
\psfrag{pkt}{$k_n(t)$}
\psfrag{ek}{$x_n$}
\psfrag{q}{$q$}
\psfrag{ell}{$\ell$}
\psfrag{pq}{$q_n$}
\psfrag{pom}{$\partial\overline{\Omega}$}
         \includegraphics[scale=0.7]{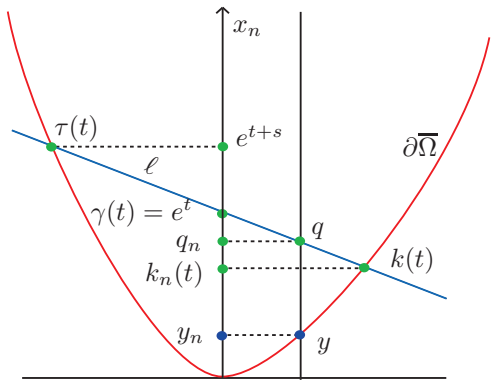}
 \end{center}
\caption{Busemann function at a round point}        
 \label{horocalc}

\end{figure}
Denote the $x_n$-coordinate of $q$ by $q_n$, of $y$ by $y_n$ and of $\tau(t)$ by $e^{t+s}$.
Projection onto the $x_n$--coordinate axis preserves cross ratios, 
so
\begin{align*}
d_\Omega( \gamma(t) , q) -t & = \log |CR(k_n(t),q_n,e^t,e^{t+s})|-t \\
 & = 
\log\ \left|\displaystyle\frac{e^{t}- k_n(t)}{e^t-e^{t+s}}\cdot\frac{q_n-e^{t+s}}{q_n- k_n(t)}\cdot e^{-t}\right|\\
& = 
\log\ \left|\displaystyle\frac{1- e^{-t}k_n(t)}{e^{-s}-1}\cdot\frac{e^{-(t+s)}q_n-1}{q_n- k_n(t)}\right|.
\end{align*}
Observe that $k(t)\to y$ as $t\to\infty,$ so $k_n(t)\to y_n$.
Since $p$ is a round point, as $t$ tends to infinity, the point $\tau(t)$ moves  arbitrarily far from the $x_n$-axis
and this implies $s\to\infty$ as $t\to\infty$.
Taking the limit as $t\to\infty$ gives
$$ \beta_\gamma(q) =\lim_{t\to\infty}\left(d_\Omega( \gamma(t) , q)-t\right) =  -\log\left| q_n-y_n \right| $$
It follows that  the level sets of $\beta_{\gamma}$ are  $(\partial\overline{\Omega}\cap{\mathbb R}^n)+te_n$ 
given by $q_n-y_n=e^{-t}$  for fixed $t>0$.
\end{proof}

\begin{corollary} 
Suppose $p\in\partial\overline{\Omega}$ is a $C^1$ point and $\beta_p$ a Busemann function for a ray asymptotic to $p$.  Then the horosphere displacement function $h:SL(\Omega,H,p)\longrightarrow{\mathbb R}$ is given by $h(A)=\beta_p(x)-\beta_p(Ax)$ for every $x\in\Omega$.
\end{corollary}

\begin{corollary}[parabolic quotient]\label{parabolicquotient} Suppose $\Omega$ is a properly convex 
domain and $\Gamma\subset \SL(\Omega,H,p)$ is a group of parabolics. Then $\Omega/\Gamma$ is not compact.
\end{corollary}
\begin{proof} Since $\Gamma$ preserves $(H,p)$ horospheres there is a continuous  surjection $\Omega/\Gamma\longrightarrow{\mathbb R}$ given by collapsing each horosphere to a point.
\end{proof}
\section{Elementary Groups}\label{elementarysection}

A subgroup $G\le SL(\Omega)$ is {\em parabolic} if every element in $G$ is
{\em parabolic}.  Similar definitions apply for the terms {\em
nonparabolic, elliptic, nonelliptic, hyperbolic, nonhyperbolic}.  The
subgroup is {\em elementary} if it fixes some point
$p\in\overline{\Omega}$.  It is {\em doubly elementary} if it fixes some $p\in\bOmega$
and if in addition it also preserves a supporting hyperplane $H$ to
$\cOmega$ at $p$.  The latter condition is equivalent to fixing the dual point
$H^*$  in $\partial\overline{\Omega^*}$ and is important for the study
of parabolic groups.  The main results in this section are:
\begin{itemize}
\item Every nonhyperbolic group is elementary (\ref{nonhyperbolicelementary}).
\item In the {\em strictly convex} case,  every nonelliptic elementary group is doubly 
elementary (\ref{doubleelementary}).
\item  For discrete groups in the {\em strictly convex} case  elementary coincides with virtually nilpotent (\ref{elementarypartition}).
\end{itemize}

\begin{theorem}
\label{nonhyperbolicelementary} 
If $\Omega$ is properly convex, then every nonhyperbolic subgroup 
of $\SL(\Omega)$ is elementary.
\end{theorem}

Some lemmas are needed for the proof of Theorem \ref{nonhyperbolicelementary}.

\begin{lemma}
\label{absirrep}
Suppose that $G$ is an irreducible subgroup of  $SL(n, {\mathbb C})$ and the trace function is bounded
on $G$. Then $G$ has compact closure.
\end{lemma}

\begin{proof} Since $G$ is an irreducible subgroup of  $SL(n, {\mathbb C})$, Burnside's theorem (\cite{Lang} p.648 Cor. 3.4) implies
that we can choose $n^2$ elements of $G$, $\{ g_i\;\;|\;\; 1 \leq i \leq n^2\}$ which are a basis for $M(n,  {\mathbb C})$.

The trace function defines a nondegenerate bilinear form on $M(n,  {\mathbb C})$, so we can choose elements $g_i^*$ which 
are dual to the $g_i$'s, i.e. $\tr(g_i\cdot g_j^*) = \delta_{ij}$. 
These dual elements also form a basis, so that given any $g \in G$
we have
$$ g = \sum_i a_i g_i^*.$$
This gives
$$ \tr(g.g_j) = \tr( \sum_i a_i g_i^*g_j)=  \sum_i a_i \tr(g_i^*g_j) = a_j.$$
By hypothesis traces are bounded on $G$, so $G$ is a bounded subgroup of  $M(n,  {\mathbb C})$, and therefore has compact closure  in $SL(n, {\mathbb C})$.
\end{proof}

\begin{lemma}\label{cpctfp}
Suppose $\overline{\Omega}$ is properly convex and $G\le \SL(\Omega)$ is compact. 

Then $G$ fixes some point in $\Omega$.
\end{lemma}

\begin{proof} Consider the set ${\mathcal S}$ of compact convex $G$-invariant non-empty subsets of $\Omega$.
Since $G$ is compact the convex hull of the $G$-orbit of a point in $x \in \Omega$ is an element of ${\mathcal S}$; so this set is nonempty.

There is a partial order
given by $A< B$ if $A\supset B.$ Then every chain is bounded above by the intersection of the elements of
the chain. By Zorn's lemma there is a maximal element $K$ of ${\mathcal S}$. 
If $K$ is not a single point and  is convex, there is
a point $y$ in the relative interior of $K.$ By considering the Hilbert metric on the interior of $K$ one sees that
 the closure of the $G$-orbit of $y$ is a proper subset of $K$ contradicting
maximality.
\end{proof}

\begin{lemma}
\label{realdecomp}
Suppose that $\rho : G \longrightarrow GL(n, {\mathbb R})$ is irreducible and $\rho\otimes{\mathbb C}$ is reducible.

Then  $\rho\otimes{\mathbb C}=\sigma \oplus \overline{\sigma}$, where
$\sigma$ is an irreducible complex representation of $G$.
\end{lemma}

\noindent
{\bf Proof.} Suppose that $\sigma$
is a complex irreducible subrepresentation  of $\rho \otimes  {\mathbb C}$ with image $U\subseteq{\mathbb C}^n$. 
Since $\rho$ is real it follows that the complex-conjugate representation $\overline{\sigma}$ is
also a subrepresentation  of $\rho \otimes  {\mathbb C}$ with image  $\overline{U}$. 
Now  $U \cap \overline{U}$ is $G$-invariant
and preserved by complex conjugation, so it is of the form $V\otimes {\mathbb C}$ for some subspace
$V\subseteq{\mathbb R}^n$. Since $\rho$ is $ {\mathbb R}$-irreducible, $V=0$. Thus $\sigma\oplus\overline{\sigma}$ is a 
representation with image $U \oplus \overline{U}$ 
that is invariant under complex conjugacy. Arguing as before, the image must be all of ${\mathbb C}^n$. \qed \\[\baselineskip]

{\bf Proof of \ref{nonhyperbolicelementary}.} Suppose $\rho: G\longrightarrow SL(n,{\mathbb R})$ is the representation
given by the inclusion map of a nonhyperbolic subgroup $G\ < \SL(\Omega)$.
The hypothesis $\rho$ is nonhyperbolic implies $|\tr\rho|\le n$ thus $\rho$ has bounded trace.
If $\rho$ is {\em absolutely irreducible} (i.e.\thinspace irreducible over ${\mathbb C}$) 
then we are done by Lemmas \ref{absirrep} and \ref{cpctfp}.
If $\rho$ is not absolutely irreducible, but is ${\mathbb R}$-irreducible, then \ref{realdecomp} shows that 
$\rho\otimes{\mathbb C}=\sigma\oplus\overline{\sigma}$ 
with $\sigma$ irreducible.  Now
 $\sigma$ has bounded trace so \ref{absirrep} implies
$\sigma$ and hence $\rho$ have image with compact closure giving a fixed point as before.

It remains to consider a nontrivial decomposition ${\mathbb R}^{n+1} \cong A \oplus B$, where 
$A$ is $G$-invariant. 
First suppose $\dim B =1.$ Since $\dim A=n,$ it follows that $H={\mathbb P}(A)$ is a hyperplane which is preserved by $G.$ Thus $G$ preserves the complement of this hyperplane and hence is an affine group. If $\overline{\Omega}$ is disjoint from $H,$ then it is a compact convex set in affine space preserved by $G.$ This implies that $G$ is compact and therefore \ref{cpctfp} implies that $G$ has a fixed point. 

Hence suppose $\dim B >1.$ This remaining case is proved by induction on $n$.
If ${\mathbb P}(A)$ meets $\overline{\Omega}$, then  ${\mathbb P}(A) \cap \overline{\Omega}$
 is a properly convex $G$-invariant set of lower dimension and, by
induction and the previous case, there is a fixed point for $G$ in $\overline{\Omega} \cap {\mathbb P}(A)$. 
So we may assume that they are disjoint. We claim that  the image of $\overline{\Omega}$ under the projection
\begin{equation}\tag{$*$}
\pi :  {\mathbb R}P^n - {\mathbb P}(A) \longrightarrow {\mathbb P}(B)
\end{equation}
is a properly convex subset of ${\mathbb P}(B)$.  
 
 Assuming this,
 consider the action, $\rho'$, of $G$ on ${\mathbb P}(B)$, given 
by the  action on $B\cong{\mathbb R}^n/A$. 
This corresponds to a block decomposition of the matrices in $\rho$ so
the eigenvalues of $\rho'$ are a subset of those for  $\rho$. Thus $\rho'$ has no hyperbolics.
By induction there is a fixed point $p\in\partial(\pi\overline{\Omega})$  for $\rho'$.
  Then $\overline{\Omega'}=\pi^{-1}(p)\cap\overline{\Omega}$ is a nonempty properly convex $G$-invariant set
of smaller dimension and the result follows by induction.

It only remains to prove the claim. Choose a hyperplane in ${\mathbb
R}P^n$ disjoint from $\bOmega$ and in general position with respect to
${\mathbb P}(A)$.  The complement is an affine patch ${\mathbb A}^n$ which
contains $\overline{\Omega}$ and the affine part $A_{\mathbb A}={\mathbb
P}(A)\cap{\mathbb A}^n$.  Both these sets are convex, so we may apply the separating hyperplane
theorem (4.4 of \cite{gruber}) to deduce that there is an affine hyperplane
$H_{\mathbb A}$ in ${\mathbb A}^n$ which separates $\overline{\Omega}$ from
$A_{\mathbb A}$ inside ${\mathbb A}^n$.  The affine subspaces $A_{\mathbb
A}$ and $H'_{\mathbb A}$ are disjoint.  Since $H'_{\mathbb A}$ is a
hyperplane, $A_{\mathbb A}$ is parallel to a subspace of $H'_{\mathbb A}$.
Therefore we can move $H'_{\mathbb A}$ away from $\overline{\Omega}$ to a parallel affine hyperplane
$H_{\mathbb A}$ which contains $A_{\mathbb A}$ and is disjoint from $\overline{\Omega}$.  Thus there is a projective
hyperplane $H$ in ${\mathbb R}P^n$ which contains $H_{\mathbb A}$, and thus
${\mathbb P}(A)$, and misses $\overline{\Omega}$.

We claim  $\pi( {\mathbb R}P^n - H)\subseteq {\mathbb P}(B)-H$: 
suppose $\pi(x)\in{\mathbb P}(B) \cap H$, then by definition of the projection,
there is a straight line containing $x$ with one endpoint  $y\in{\mathbb P}(A)\subseteq{\mathbb P}(H)$ and the other
 endpoint at $\pi(x)$. If $\pi(x) \in H$, then the entire line is in $H$, thus $x \in H$.

Thus $\pi(\cOmega)$ is a compact convex set in the affine part ${\mathbb P}(B)-H$ of ${\mathbb P}(B)$ and 
therefore properly convex. This completes the proof.\qed

\begin{corollary}
\label{parabolicdouble} If $\Omega$ is properly convex, then 
every nonhyperbolic group is either elliptic or doubly elementary.
\end{corollary}
\begin{proof} A nonhyperbolic group fixes a point $p\in\overline{\Omega}$ by \ref{nonhyperbolicelementary}.
Either $p\in\bOmega$ or the group is elliptic. In the first case
the set of supporting hyperplanes to $\Omega$ at $p$ is a compact, properly convex subset, $K$, of the dual
projective space. The dual action of the group on $K$ is by nonhyperbolics and 
so fixes a point in $K$ by \ref{nonhyperbolicelementary}.
\end{proof}

\begin{proposition}\label{hypfpC1} 
If $\Omega$ is strictly convex and $p\in\bOmega$ is fixed by a hyperbolic, then  $p$ is a $C^1$ point of $\bOmega$.
\end{proposition}

\begin{proof} Suppose $A\in\SL(\Omega,p)$ is hyperbolic. 
 Since $\Omega$ is strictly convex, \ref{fixedpts} implies that $A$ has unique eigenvalues $\lambda_{\pm}$ of 
 largest and smallest modulus and these are positive reals.
  
 Now $A$ acts  on ${\mathcal D}_p{\mathbb R}P^{n}\cong{\mathbb R}P^{n-1}$ as some projective 
 transformation $B$. It follows that the eigenvalues of $B$ are those of $A$ with the eigenvalue corresponding
 to $p$ omitted.
We may assume the eigenvalue for $p$ is $\lambda_-$ so that $\lambda_+$ is the unique eigenvalue of 
$B$ of largest modulus.

By \ref{invarianthyperplane}, there is a supporting hyperplane $H$ to $\Omega$ at $p$ that is preserved by $A$, so that 
$A$ acts as an affine map on the affine space ${\mathbb A}^n=\RPn\setminus H$ and preserves 
the point $\pm p$ at infinity.  Thus $B$ restricts to an affine map,
also denoted $B$, on ${\mathbb A}^{n-1}={\mathcal D}_p{\mathbb A}^n$.

Let $q\in\bOmega$ be the other fixed point of $A$.  The line
$\ell\subseteq\RPn$ containing $p$ and $q$ gives a point $[\ell]\in{\mathbb
R}P^{n-1}$.  Because $\ell$ intersects $\Omega$ in a segment,
$[\ell]\in{\mathcal D}_p\Omega\subseteq{\mathbb A}^{n-1}$.  It follows this
is the unique fixed point for that the action of $B$ on ${\mathbb A}^{n-1}$
and it is an attracting fixed point: every point in ${\mathbb A}^{n-1}$
converges to it under iteration of $B$.  The closure $C$ of ${\mathcal
D}_p\Omega\subseteq {\mathbb A}^{n-1}$ is invariant under 
$B.$ Now $[\ell]$ is in the interior of $C$ and if $\partial C \neq \emptyset,$ there is a point on
$\partial C$ closest to $[\ell]$ and which converges to $[\ell]$ under iteration.
Since $\partial C$ is preserved by $B$ it must therefore be empty, so ${\mathcal
D}_p\Omega={\mathbb A}^{n-1}$ and $p$ is a $C^1$ point.
\end{proof}

\noindent
{\bf Remark.} 
The cone point of example E(iii) is fixed by $O(2,1).$ This shows
 that \ref{hypfpC1} and the next result both fail for {\em properly convex} domains.

\begin{corollary}\label{doubleelementary} 
If $\Omega$ is strictly convex, then every 
elementary subgroup of $\SL(\Omega)$ is elliptic or doubly elementary.
\end{corollary}

\begin{proof}  
If $G$ contains a hyperbolic, then by \ref{hypfpC1}, $p$ is a $C^1$ point. So there is a unique supporting hyperplane to $\Omega$ at $p$ which therefore must be preserved by $G$.
Otherwise $G$ is nonhyperbolic. If it is not elliptic, \ref{parabolicdouble} implies that it is doubly elementary.
\end{proof}

We are now in a position to prove that parabolics have translation length $0$.
\begin{proposition}\label{nonhypsmalldist} Suppose $\Omega$ is properly convex and $G\le \SL(\Omega)$ is 
nonhyperbolic. If $\epsilon>0$ and $S\subseteq G$ is finite, there is $x\in\Omega$
 such that $d_{\Omega}(x,Ax)<\epsilon$ for all $A\in S$.
 \end{proposition}
\begin{proof} 
By \ref{nonhyperbolicelementary} and
\ref{doubleelementary} $G$ is  elementary elliptic or doubly elementary. If $G$ is elementary elliptic,
then there is a point $x\in\Omega$ fixed by $G$. This leaves the case $G\subseteq \SL(\Omega,H,p)$.

First assume $p$ is a $C^1$ point. Given $y\in\Omega$ let $\ell$ be the ray in $\Omega$ from $y$ to $p$.
The result holds for every point $x$  on $\ell$ close enough to $p$. The reason is that
the finite set of lines $S\cdot\ell$ is asymptotic to $p$. The point $x$ lies
on some $(H,p)$-horosphere ${\mathcal S}_t$. Since $G$ contains no hyperbolics, it
preserves each horosphere, thus
 $S\cdot x={\mathcal S}_t\cap (S\cdot \ell)$.
Moving $x$ vertically upwards corresponds to moving the horosphere ${\mathcal S}_t$ vertically upwards.
Since $p$ is a $C^1$ point \ref{horosphereprops}(H7) implies  the diameter of $S\cdot x$ goes to $0$.

We proceed by
induction on dimension $n=\dim\Omega$. When $n=1$ the result is trivially true.
The space
of directions of $\Omega$ at $p$ is a product ${\mathcal D}_p\Omega\cong \Omega'\times{\mathbb A}^k$
with $\Omega'$ properly convex. One of these factors might be a single point.
 Observe that $\dim \Omega'\le \dim\Omega-1.$

If $\Omega'$ is a single point then $\Omega$ is $C^1$ at $p$ and the result follows from
the above. Otherwise $G$ induces an action on $\Omega'$ which is nonhyperbolic.
By \ref{nonhyperbolicelementary} there is a fixed point $w\in\overline{\Omega'}.$ The first
case is that $w\in\Omega'$. The preimage of $w$ under the projection $\Omega\rightarrow\Omega'$
is the intersection of $\Omega$ with a projective subspace. This is a properly convex
$\Omega''\subseteq\Omega$ which is preserved by $G$. By induction there is $x\in\Omega''$
with the required property.

The remaining case is that $w\in\bOmega'$.   
By induction there is $y'\in\Omega'$ (close to $w$) which is moved at most $\epsilon/2$ by every element of $S$.
Choose $y\in\Omega$ which projects to $y'$. As in the $C^1$ case let $\ell$ be the ray in $\Omega$ from $y$ to $p$.
We show that every point $x$ on $\ell$ close enough to $p$ 
 is moved less than $\epsilon$ by every element of $S$. This will complete the inductive step.
 
 Given $s\in S$  the points $y',sy'\in\Omega'$ lie on a line segment  $[a',b']\subseteq\overline{\Omega'}$ 
 with endpoints $a',b'\in\bOmega'$. Choose $A',B'$ in the interior of this segment
  with $A'$ close to $a'$ and $B'$ close to $b'$
 so that the cross-ratios of $(a',y',sy',b')$ and $(A',y',sy',B')$ are very close, then
 $d_{\Omega'}(y',sy')<\epsilon$.
  If $x$ is a point on $\ell$ close enough to $p$ then the line segment $[A,B]$ in $\overline{\Omega}$
   with $A,B\in\bOmega$ containing
 $x$ and $sx$ has image which {\em contains} $[A',B']$. This projection is projective and thus preserves cross-ration.
 It follows that
 $d_{\Omega}(x,sx)< d_{\Omega'}(y',sy')<\epsilon$.
 \end{proof}
 
For  the parabolic $A$ discussed in example E(iii) if $y\in D^2$  then all the points on a line $[p,y]$ 
 are moved the {\em same} distance. To produce a point $q$ near
 $p$ moved a small distance  $q$ must approach $p$ along an arc
becoming tangential to $[p,x]$ as it approaches $p$. 
 
 \begin{proposition}\label{elementaryimpliesvirtnil} 
 If $\Omega$ is 
 properly convex, then every  discrete nonhyperbolic group is virtually nilpotent.
 \end{proposition}
\begin{proof}  
Suppose $G$ is a nonhyperbolic group.  
By \ref{nonhypsmalldist} if $S$ is a finite subset of $G$ there is $x\in\Omega$
so that the elements of $S$ all move $x$ less than $\mu$. It follows from the Margulis lemma \ref{projectivemargulis-technical} that 
the subgroup of $G$ generated by $S$ contains a nilpotent  subgroup of index at most $m$.
Then \ref{nilpotentlemma} below implies that $G$ is virtually nilpotent.\end{proof}

\begin{lemma}\label{nilpotentlemma} 
If $G$ is a linear group and every finitely generated subgroup of $G$ 
contains  a nilpotent subgroup of index at most $m,$ then $G$ contains a nilpotent subgroup of finite index.
\end{lemma}

\begin{proof} Suppose $S\subseteq G$ is finite and let $S'$ denote the set of $k$-th powers   of elements in $S$ where $k=m!$.
 The group $H=\langle S'\rangle \subseteq \langle S\rangle$ generated by $S'$ is nilpotent.  Since $G\le GL(n,{\mathbb R})$ 
 it follows that $H$  is conjugate into the Borel subgroup of upper triangular matrices in $GL(n,{\mathbb C})$. Hence there
 is a uniform bound, $c$, on the nilpotency class of every such $H$ and every $c$-fold iterated commutator of $k$-th powers of 
 finitely many elements in $G$ is trivial. 
 
This is an algebraic condition on the elements of $G$,  therefore the Zariski closure, $\overline{G}$, 
of $G$ in $GL(n,{\mathbb C})$ 
also has this property. 

Let $W$ denote the connected component of the identity in
$\overline{G}$. There is a neighborhood, $U$, of the identity in $W$ which is in
the image of the exponential map. Every element in $U$ is a $k$-th power.
Hence every $c$-fold iterated commutator of elements in $U$ is trivial. Since $U$
generates $W$ it follows that $W$ is nilpotent.
 The algebraic
group $\overline{G}$ has finitely many connected components. Thus $W$ has finite index in $\overline{G}$.
\end{proof}

 \begin{proposition}\label{hypparob} If $\Omega$ is strictly convex, then every discrete elementary 
 torsion-free group is virtually nilpotent and
either hyperbolic or parabolic.
\end{proposition}
\begin{proof}  If  $G$ is hyperbolic, discreteness implies  $G$ is infinite cyclic hence
virtually nilpotent.

If $G$ is nonhyperbolic the result follows from \ref{elementaryimpliesvirtnil}.
We claim that  these are the only possibilities for $G$.

\begin{figure}[ht]	 
\begin{center}
	 \psfrag{Sr}{${\mathcal S}_r$}
	 \psfrag{St}{${\mathcal S}_t$}
	 \psfrag{dn}{$d_n$}
	 \psfrag{bx}{$\beta^n x$}
	 \psfrag{x}{$x$}
	 \psfrag{ax}{$\alpha x$}
	 \psfrag{abx}{$\alpha\beta^n x$}
	 \psfrag{babx}{$\beta^{-n}\alpha\beta^n x$}
	 \psfrag{L}{$\ell$}
	 \psfrag{dL}{$\alpha\ell$}
	\psfrag{p}{$p$}
	\psfrag{Om}{$\Omega$}
	 \includegraphics[scale=0.6]{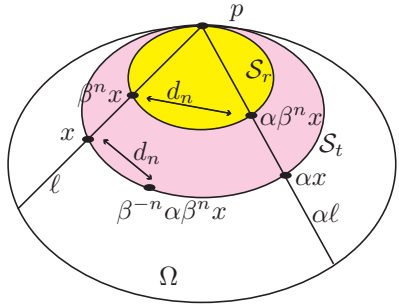}
\end{center}
  \caption{Conjugate of parabolic by a hyperbolic}\label{hypconjpar}
\end{figure}

Refer to Figure \ref{hypconjpar}. Suppose that $\alpha,\beta\in G$ and $\beta$ is hyperbolic with axis $\ell$ and $\alpha$ is parabolic.
Let $x$ be a point on $\ell$.
 The points $x$ and $\alpha x$ lie on a 
horosphere ${\mathcal S}_t,$ and their images under $\beta^n$ lie on another horosphere ${\mathcal S}_r$. 
The points $x$ and $\beta^n x$ are both on $\ell$ so  $\alpha x$ and $\alpha\beta^n x$ are both on 
$\alpha\ell$. Furthermore $\beta^nx\to p$ as $n\to\infty$. By \ref{hypfpC1} $p$ is a $C^1$ point and this 
implies $d_n=d_{\Omega}(\beta^n x,\alpha\beta^n x)\to 0$ as
$n\to\infty$. Since $\beta^n$ is an isometry  $d_{\Omega}( x,\beta^{-n}\alpha\beta^n x)=d_n\to 0$. Then
 \ref{horosphereprops}(H7) implies
$G$ does not act properly discontinuously on $\Omega$ and \ref{discreteproperdiscts} implies 
$G$ is not discrete.\end{proof}

\begin{proposition}[virtually nilpotent $\Rightarrow$ elementary]
\label{elementary} 
Suppose $\Gamma$ is a virtually nilpotent group  of isometries of 
a strictly convex domain and $\Gamma$ is nonelliptic. 
Then $\Gamma$    is  elementary.
\end{proposition}

\begin{proof}  
The given group $\Gamma$ contains a finite-index infinite nilpotent
subgroup $\Gamma_0 \subseteq \Isom(\Omega)$. Hence $\Gamma_0$ contains a
nontrivial central element $\gamma$. By \ref{fixedpts} $\gamma$ fixes
exactly one or two points in $\bOmega$. Since $\gamma$ is central it
follows that each element of $\Gamma_0$ permutes these fixed points.  Hence
there is a subgroup, $\Gamma_1$ of $\Gamma_0$ of index at most two which
fixes a fixed point, $x$, of $\gamma$ and is thus elementary.

It follows that $\Gamma$ itself is elementary.  For suppose  that $\gamma$ is a
nontrivial element of $\Gamma$. Then some power $\gamma^n$ with $n\ne0$ is
in $\Gamma_1$,  and this power must fix $x$. By hypothesis $\gamma$ is not elliptic
so it is parabolic or hyperbolic. The subset of the boundary of a strictly
convex domain fixed by a parabolic or hyperbolic is not changed by taking
powers of the element.  Hence $\gamma$ also fixes $x,$ and $\Gamma$ is an
elementary group as required.
 \end{proof}

The next result is the basis of the thick-thin decomposition.
\begin{corollary}\label{elementarypartition} Suppose that $\Omega$ is strictly convex and $G\le\SL(\Omega)$ is torsion-free and discrete.
 Then
 \begin{itemize}
 \item $G$  is elementary iff it is virtually nilpotent.
\item The maximal elementary subgroups of $G$ partition the nontrivial elements of $G$.
\end{itemize}
\end{corollary}
\begin{proof} This follows from  \ref{hypparob}
and \ref{elementary} together with the observation that if two elementary groups have nontrivial intersection
then they are both hyperbolic or both parabolic. In either case they have the same fixed points and are therefore
the same group.
\end{proof}

\section{Cusps}\label{sectioncusps}

This section describes cusps in properly convex projective manifolds in terms of algebraic horospheres.
 Cusps of maximal rank play a key role, since these are the only cusps that arise in finite volume projective manifolds. The
main results of this section are Theorem  \ref{productcusp}, 
which implies that  cusps are products of the form $P\cong [0,1)\times\partial P$; 
and Proposition \ref{maxparabolicround}, which states that the parabolic fixed point corresponding to a 
maximal rank cusp is a round point of $\bOmega.$ We define four variants: {\em full cusp, convex cusp, open cusp} 
and {\em horocusp}. They differ in respect of whether or not they have boundary or are convex. The starting point are not cusps, but \emph{cusp groups}.

A {\em cusp group} is a discrete infinite subgroup $\Gamma\subseteq \SL(\Omega)$ which preserves some algebraic horosphere. Thus $\Gamma\subseteq\SL(\Omega,H,p),$ where 
$p\in\partial\overline{\Omega}$ is called the {\em parabolic fixed point}
 and $H$ is a supporting hyperplane to $\Omega$ at $p$ and both are preserved by $\Gamma$.

A {\em full cusp} is $N=\Omega/\Gamma,$ where  $\Omega$ is a properly convex domain and
$\Gamma\subseteq \SL(\Omega)$ is a cusp group.

\smallskip

 The next result explains why algebraic horospheres are used instead of Busemann's horospheres.
 From \ref{parabolicdouble} we get:
\begin{proposition}\label{virtparaboliccusp} If $\Omega$ is properly convex, then an infinite discrete group $\Gamma\subseteq\SL(\Omega)$ is a cusp group
iff it contains no hyperbolics.
\end{proposition}

 To simplify terminology in what follows,  we only discuss the case where $\Gamma$ is torsion free.  The obvious generalizations are true for orbifolds.

A {\em convex cusp} $W$ is an open submanifold of a properly convex manifold $N$ such that $W$
 is projectively equivalent to a full cusp. This implies $W$ is a convex submanifold of $N$ so $\tilde{W}$ is
 a properly convex subdomain of $\tilde{N}$.
 In general a component of the thin part of a manifold is not convex, even for hyperbolic manifolds.
This motivates the following.

Suppose $\Omega'\subset\Omega$ are both properly convex and both preserved
by a discrete group $\Gamma$. Let $W=\Omega'/\Gamma$ and $N=\Omega/\Gamma$.
If  $W\subset P\subset N$ and $P$ is connected then $W$ is a {\em convex core} of 
$P$ and $P$ is a {\em thickening} of $W$. We do not require $P$ is $W$ plus a collar, only
that they have the same holonomy.

   Suppose $N=\Omega/\Gamma'$ is a properly convex manifold. An {\em open cusp} in $N$
 is a connected open submanifold $M\subset N$ which is a thickening of a convex cusp $W$. In addition
 we require there is a parabolic fixed point $p\in\bOmega$ for $W$ and  a component
 $\tilde{M}\subset\Omega$  of the preimage of  $M$ which is starshaped at $p$. 
  
 A {\em cusp} in a properly convex manifold $N$ is a submanifold $P\subset N$ with nonempty boundary 
 $\partial P=\overline{P}\cap\overline{N\setminus P}$ such that the interior of $P$ 
 is an open cusp and so that every ray asymptotic to $p$ which contains a point
 in $P$  intersects $\partial P$ transversally at one point. 
 It follows that $P\cong [0,1)\times\partial P$.
 
A {\em horocusp} is a cusp covered by a horoball.
 The boundary of a horocusp is the quotient of a horosphere and is called
  a {\em horoboundary}. Usually we require $\partial P$ is a smooth submanifold, however this may
  not be true for horocusps.

\begin{theorem}[structure of open cusps]\label{productcusp} Suppose $M=\tilde{M}/\Gamma$ is an open 
 cusp  in a properly convex manifold $N=\Omega/\Gamma'$ with $\Gamma\subset\SL(\Omega,H,p)$.
 \begin{itemize}
 \item[(C1)] There is a diffeomorphism $h=(h_1,h_2):M\longrightarrow {\mathbb R}\times X$.
 \item[(C2)] $X$ is an affine $(n-1)$-manifold called the {\em cusp cross-section}.
\item[(C3)] Fibers of $h_2$ are the rays in $M$ asymptotic to $p$ and $h_1\to-\infty$ moving toward $p$.
 \item[(C4)] $M$ is an affine manifold.
\item[(C5)] If $V\subset M$ is an open cusp and $h_2(M\setminus V)=X$ then $V\subset h_1^{-1}(-\infty,0]$
  for some choice of $h_1$. 
\item[(C6)] In this case $P=h_1^{-1}(-\infty,0]$ is a  closed cusp.
\item[(C7)] $h_2|:\partial P\longrightarrow X$ is a diffeomorphism.
\item[(C8)] $\pi_1M$ is virtually nilpotent.
\end{itemize}
\end{theorem} 
\begin{proof}   
With reference to Figure \ref{horoballs}, parabolic coordinates centered on $(H,p)$ give an affine patch 
${\mathbb R}^{n-1}\times{\mathbb R}={\mathbb R}^n={\mathbb R}P^n\setminus H$ on which $\Gamma$
acts affinely preserving this product structure. The ${\mathbb R}$-direction is called {\em vertical} and moving {\em upwards} is moving towards $p$.
Since $\tilde{M}$ is a subset of this patch 
$M=\tilde{M}/\Gamma$ is an affine manifold proving (C4).
Now $M$ is starshaped at $p,$ so if $x\in\tilde{M}$ and $y$ is vertically above $x,$ then $y\in\tilde{M}$.

Radial projection from $p$ 
 corresponds to vertical projection of ${\mathbb R}^{n-1}\times{\mathbb R}$ onto the first factor.
This gives a diffeomorphism from  ${\mathcal D}_p\tilde{M}$ 
onto  an open set $U\subset{\mathbb R}^{n-1}$. Since $\Gamma$ preserves the product structure
it acts affinely on ${\mathbb R}^{n-1}$. Thus $p$ covers a submersion $h_2:M\longrightarrow X$
where $X=U/\Gamma\cong{\mathcal D}_p\tilde{M}/\Gamma$ is an affine manifold,  proving (C2). 

There is a $1$-dimensional foliation, ${\mathcal F}$, of $M$ covered
  by vertical lines in ${\mathbb R}^n$. This foliation is transverse to the codimension-$1$ foliation
  of $M$ covered by horospheres.  
To prove (C1) and (C3)  it suffices to show that there is  a  smooth map $f:M\longrightarrow{\mathbb R}$ whose
restriction to each line in ${\mathcal F}$  is a diffeomorphism oriented correctly.

Choose a complete smooth Riemannian metric, $ds,$ on $M.$ Given a point
$q\in M$ there is a smooth $(n-1)$-disc $D_q$ containing $q$ and contained in the interior
of another smooth $(n-1)$-disc $D_q^+$ in $M$ 
 transverse to ${\mathcal F}$ and meeting each line in ${\mathcal F}$ at most once.
Choose a smooth non-negative
function, $\psi_q,$ on $D_q^+$ which equals $1$ on $D_q$  and is zero in a neighborhood of $\partial D_q^+.$ 

We use this to define a smooth non-negative function $f_q$ on $\Int(M)$ supported 
inside the set of rays in
${\mathcal F}$ that meet $D_q^+.$ If $\ell$ is such a ray which intersects
$D_q^+$ at $x$ and $y$ is a point on $\ell$ then 
$$f_q(y)=\psi_q(x)\cdot d_{\ell}(x,y),$$ 

where $d_{\ell}(x,y)$ is the {\em signed} $ds$-length of
the segment of $\ell$ between $x$ and $y.$ The sign is positive iff $x$
lies between $y$ and $p.$ 

The function $f_q$ is smooth.  Each ray is either mapped to $0$ or onto
${\mathbb R}.$ It is a diffeomorphism on each ray on which it is not
constant, increasing as the point moves away from $p.$ 

Since $N$ is paracompact  there is a subset $Q\subset M$ so that
every ray in
${\mathcal F}$ meets at least one of the sets $\{D_q:q\in Q\}$ and at most finitely many
of the sets $\{D_q^+:q\in Q\}$.
The function $h_1=\sum_{q\in Q} f_{q}$  is smooth  because near each point in $M$ the sum is finite. It is
strictly monotonic on each ray of ${\mathcal F}.$ 
To prove (C5), since $h_2(M\setminus V)=X$ one can choose each $D_q^+\subset M\setminus V$
 then $f_q(V)\le 0$ because $V$ is starshaped from $p$.  Thus $h_1(V)\le 0$ so $V\subset P$.
 Since $V$ is an open cusp it, and hence $P$, contains a convex cusp.
The remaining conditions for $P$ to be a cusp
are readily checked, yielding  (C6). Clearly $(C1)+ (C5)\Rightarrow (C7)$. 
(C8) follows from \ref{elementaryimpliesvirtnil}.
\end{proof}

\begin{proposition}[$C^1$ open cusps] Suppose $M$ is an open cusp with a $C^1$ parabolic fixed point $p\in\bOmega$
and cusp cross-section $X$. Then
\begin{itemize} 
\item[(P1)] $X$ is a complete affine manifold.
\item[(P2)]  $X$ is homeomorphic to a horoboundary.
\item[(P3)]  $M$ is diffeomorphic to a full cusp.
\item[(P4)] For every $\epsilon>0$ and finite subset $S\subset\pi_1M$ there is a point in $M$ so that
every element of $S$ is represented by a loop based at $x$ of length less than $\epsilon$.
\end{itemize}
\end{proposition}
\begin{proof} With reference to the proof of \ref{productcusp},
 the condition $p$ is a $C^1$ point is equivalent to $U=\mathbb R^{n-1}$ and implies $X$ is diffeomorphic
to the complete affine manifold ${\mathbb R}^{n-1}/\Gamma$ proving (P1). (P2) and (P3) follows easily from considering
parabolic coordinates. (P4) follows from \ref{nonhypsmalldist}.
\end{proof}

The following implies that a cusp component of the thin part of
a strictly convex manifold must have nonempty boundary.
 \begin{lemma}\label{fatcusps}
  If $M$ is a strictly convex  
complete cusp and $\ell$ is a ray in $M$ asymptotic
to the parabolic fixed point $p$ then moving along $\ell$ away from $p$ the injectivity radius increases to infinity.
\end{lemma}
\begin{proof} Let $M=\Omega/\Gamma$. Because $\Gamma$ is discrete, it acts properly discontinuously on $\Omega$.
Therefore, at a point $x$ on $\ell$ given $r>0$
there are at most finitely many elements $\gamma_1,\cdots\gamma_n\in\Gamma$ which move $x$  distance less than $r$. This gives finitely many lines $\ell_i=\gamma_i\ell$. By \ref{divergeline} 
if $y$ is sufficiently far away from $x$ in the direction away from $p$ then 
$d_{\Omega}(y,\ell_i)>r$ for each $i$. If $\gamma\in\Gamma$ moves $y$ less than $r$ then by 
\ref{horosphereprops}(H7)
it also moves $x$ less than $r$. But then $\gamma=\gamma_i$ for some $i$ which is a contradiction. Thus
the injectivity radius at $y$ is at least $r$.
\end{proof}

   Two cusps are {\em projectively equivalent} if they have conjugate holonomy. 
   It is easy to show that every convex cusp is diffeomorphic to a full cusp. Thus equivalent
   {\em convex} cusps are diffeomorphic. It is also easy to show that every maximal rank cusp is  diffeomorphic
   to a full cusp.
    Corollary~\ref{parabolic2and3} implies all $2$--dimensional cusps are projectively equivalent.

A cusp has {\em maximal rank} if the boundary is compact. 
There are several equivalent formulations
which will be useful.
The {\em Hirsch rank} of a finitely generated nilpotent group $G$ is the sum of the ranks of 
the abelian groups $G_{i}/G_{i+1}$ for any central series $1=G_n<G_{n-1}<\cdots <G_1=G.$
 This equals the virtual cohomological dimension of $G.$ The {\em rank} of a cusp, $M,$ is the Hirsch rank of any nilpotent 
 subgroup of finite index in  $\pi_1M$ and is thus at most $1$ less than the topological dimension of $M.$ 
  Following Bowditch \cite{Bow} a  point $p\in\bOmega$ is called a {\em bounded parabolic point} of a discrete
 group of parabolics $\Gamma\subset\SL(\Omega,p)$ if 
$(\bOmega\setminus p)/\Gamma$ is compact. 

\begin{proposition}[maximal cusps]\label{maximalcusps}
Suppose $M$ is a cusp in $N=\Omega/\Gamma'$ with parabolic fixed point $p$ and holonomy $\Gamma$. The following
are equivalent:
 \begin{itemize}
 \item[(M1)]\label{cusp1} $M$ has maximal rank.
  \item[(M2)]\label{cusp2} $\partial M$ is compact. 
\item[(M3)]\label{cusp3} ${\mathcal D}_p\Omega/\Gamma$ is compact.
\item[(M4)]\label{cusp4} $\Gamma$ has Hirsch rank $\dim(M)-1$.
  \item[(M5)]\label{cusp5} $p$ is a  bounded parabolic point for $\Gamma$.
 \end{itemize}
  \end{proposition}

  \begin{proof} $M1\Leftrightarrow M2$ by definition. 
  Let $\partial\tilde{M}\subset\Omega$ be the pre-image of $\partial M$.
  Radial projection from $p$ embeds ${\mathcal D}_p\partial\tilde{M}$
  as an open subset of ${\mathcal D}_p\Omega$.
 This identification is $\Gamma$-equivariant.  So
$\partial M\subset {\mathcal D}_p\Omega/\Gamma$. The identification of ${\mathcal D}_p\Omega$ with a horosphere
shows that action of $\Gamma$
on ${\mathcal D}_p\Omega$ is properly discontinuous. Therefore these are  {\em Hausdorff} 
manifolds of the same dimension
and the inclusion induces an isomorphism of fundamental groups. If $\partial M$ is compact
then it is a closed manifold
so ${\mathcal D}_p\Omega/\Gamma$ is a closed manifold hence compact, proving $(M2)\Rightarrow (M3)$. 
Conversely, if  
${\mathcal D}_p\Omega/\Gamma$ is compact, then it is a closed manifold and also a $K(\Gamma,1)$.
Since $M$ is a cusp it contains a convex core $W$ and inclusion induces $\pi_1M\cong\pi_1W$. 
Also radial projection ${\mathcal D}_p$ induces isomorphisms 
$\pi_1\partial M\cong\pi_1M$ and $\pi_1W\cong\pi_1\partial W$. Convexity
implies $\partial W$ is a $K(\Gamma,1)$ also. Hence $\partial W$ is closed and
${\mathcal D}_p$ covers an inclusion
 $\partial W\hookrightarrow {\mathcal D}_p\Omega/\Gamma$
which is a homotopy equivalence of closed manifolds. Thus they are equal, and equal to 
$\partial M$, proving $(M3)\Rightarrow(M2)$.

$M2\Leftrightarrow M4$
  because $\partial M$ is a $K(\Gamma,1)$ hence the 
  virtual cohomological dimension of  $\Gamma$ is $\dim(\partial M)$ if and only if $\partial M$ is a closed manifold.

For $(M1)+(M3)\Rightarrow(M5)$ by Theorem \ref{maxparabolicround} $p$ is a round point. Then
radial projection from $p$ gives a $\Gamma$-equivariant identification of 
$\bOmega\setminus p$ with ${\mathcal D}_p\Omega$.

For $(M5)\Rightarrow(M3)$ let $H$ be a $\Gamma$-invariant supporting hyperplane at $p$.
If  $H\cap\bOmega= p$  then  radial projection from $p$ identifies
$\bOmega\setminus p$ with ${\mathcal D}_p\Omega$ implying $(M3)$. Otherwise 
$X=H\cap\bOmega\setminus p$
is a properly convex set on which $\Gamma$ acts by nonhyperbolics. But $X/\Gamma$ is not compact: a ray
in $X$ converging to $p$ does not converge in $X/\Gamma$. However $X$ is a closed subset of $\Omega\setminus p$
so  $X/\Gamma$ must be compact by $(M5)$. This contradiction completes the proof. \end{proof}

Using \ref{maximalcusps}$(M2)\Rightarrow(M3)$, if $M$ is a maximal cusp with parabolic fixed point $p$ the
hypothesis of the next result is satisfied by the holonomy. 
\begin{theorem}[max parabolic fixed point is round]
\label{maxparabolicround} Suppose $\Omega$ is a  properly convex set and $p\in\bOmega$
and $\Gamma\subset\SL(\Omega,p)$ is parabolic. If ${\mathcal D}_p\Omega/\Gamma$ is
compact then $p$ is a round point of $\bOmega$.
\end{theorem}

\begin{proof} 
By Corollary~\ref{spacedirections}, ${\mathcal D}_p\Omega$ is projectively equivalent to ${\mathbb
A}^{k}\times C,$ where $C$ is properly convex.  Every subspace of ${\mathbb A}^{k}\times C$
projectively isomorphic to ${\mathbb A}^{k}$ is of the form ${\mathbb
A}^k\times \{c\}$ for some $c\in C.$ It follows that every projective
transformation, $[A]\in SL(n+1,{\mathbb R}),$ which preserves ${\mathbb
A}^{k}\times C,$ induces a projective transformation on $C.$ Thus we get an
induced action of $\Gamma$ on $C.$ Then $C/\Gamma$ is a quotient of
${\mathcal D}_p\Omega/\Gamma$ and is therefore compact.

Using a basis of ${\mathbb R}^{k}$ followed by a basis of
${\mathbb R}^{n+1-k},$ we see that
$$A=\left(\begin{array}{cc} M_k & N_{k,n+1-k}\\ 0 &  R_{n+1-k}\end{array}\right).$$ 
The induced map on $C$ is  given by $[R].$ In particular, the eigenvalues of $R$ are a subset of those of $A.$ 
Since $A$ is nonhyperolic, all its eigenvalue have modulus $1.$ Hence $R$ is nonhyperbolic. By 
\ref{nonhyperbolicelementary}
$\Gamma$ fixes a point, $q,$ in $\overline{C}$.

If $q\in C,$ then $C/\Gamma$ is not compact, since the distance of a point in
$C$ from $q$ is preserved by the action, and hence $C/\Gamma$ maps onto
$[0,\infty).$ Whence $q\in\partial C.$ But now Corollary~\ref{parabolicquotient} implies
that the quotient $C/\Gamma$ is not compact.  This contradiction shows that ${\mathcal D}_p\Omega =
{\mathbb A}^{n-1}$.

Applying the same argument to the action on the dual domain $\Omega^*$, it follows
 that $p$ is not contained in a line segment of positive length in
$\partial \Omega$.  \end{proof}

Suppose $M=\Omega/\Gamma$ is a non-compact convex projective manifold which 
contains a convex core $M'$.
The universal cover of $M'$ is a $\pi_1M$-invariant convex subset $\Omega'\subset\Omega.$
It may happen that one of these manifolds is strictly convex and the other is not. For example,
if $M={\mathbb H}^2/\Gamma$ is a full 2-dimensional hyperbolic cusp 
and $x$ is a point in $M$ there is a geodesic segment
$\gamma$ in $M$ starting and ending at $x.$ Let $M'$ denote the 
component of $M\setminus\gamma$ which contains the cusp of $M.$ The universal cover
of $M'$ is convex set bounded by an infinite sided polygon, so it is properly but not strictly convex.
This construction can 
sometimes be reversed:

\begin{proposition}\label{strictlyconvexcusps} 
Suppose that $M=\Omega/\Gamma$ is a full cusp with
$\Gamma\subset\SL(\Omega,H,p)$. 
Then there is a properly convex domain $\Omega'\subset\Omega$ with 
$\overline{\Omega}'\cap H=\overline{\Omega}\cap H$ that is preserved by $\Gamma$. 
Thus  $M'=\Omega'/\Gamma$ is a full cusp that is 
projectively equivalent to $M.$
Moreover,
$\Omega'$ is strictly convex and $C^1$, 
except possibly at $\overline{\Omega}\cap H$. \end{proposition}

\begin{proof}
Refer to Figure~\ref{hsurface}.
 The sublevel sets of the characteristic function $f$ given by Theorem~\ref{Vinbergsurface} are
  strictly convex and real-analytic.
 We may embed ${\mathbb R}^{n+1}$ as an affine patch  in ${\mathbb R}P^{n+1}.$ 
 The closure
 $\overline{{\mathcal C}(\Omega)}$ of ${\mathcal C}(\Omega)$ in ${\mathbb R}P^{n+1}$ is a 
 compact cone. There are coordinates so that the cone point, $q,$ is the origin
  in ${\mathbb R}^{n+1},$ and the base is $\overline{\Omega}\subset{\mathbb R}P^n.$ 
   
Let $K\subset \overline{{\mathcal C}(\Omega)}$ be the closure of a 
sublevel set of $f.$ Then $\partial K=\overline{\Omega}\cup S$ where
 $S$ is a level set of $f.$ Let $\Omega^*$ be the dual domain. The dual action of $\Gamma^*$
 fixes the point $\alpha\in\bOmega^*$ which is dual  $H$.
  
 \begin{figure}[ht]	 
\begin{center}
	 \psfrag{Ht}{$H_t$}
	 \psfrag{H}{$H$}
	 \psfrag{q}{$q$}
	 \psfrag{p}{$p$}
	 \psfrag{K}{$K$}
	 \psfrag{W}{$W$}
	 \psfrag{piW}{$\pi(W)$}
	 \psfrag{C}{$\overline{{\mathcal C}(\Omega)}$}
	 \psfrag{S}{$\color{red}S$}
	 \psfrag{Om}{$\overline{\Omega}$}
	 \includegraphics[scale=1.1]{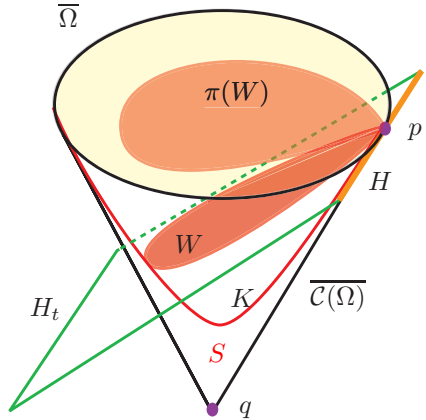}
\end{center}
  \caption{Hilbert hypersurface}\label{hsurface}
\end{figure}

 There is a pencil of hyperplanes $H_t\subset{\mathbb R}P^{n+1}$ with center  $H$ and dual to
 some projective line $L$ in the dual space. The group $SL(C(\Omega),H,p)$ acts projectively on $L$ 
 fixing the points dual to  two hyperplanes, one that contains $\Omega,$
 and the other that contains $q.$ In particular every parabolic in this group acts trivially
 on $L.$
 
 Choose a hyperplane $H_t$ that contains a point in the interior of $\overline{{\mathcal C}(\Omega)}.$ 
 Then $W=K\cap H_t$
 is the intersection of two convex sets and so  is convex.  
 Moreover $\partial W=\partial K\cap H_t=(\overline{\Omega}\cap H_t)\cup(S\cap H_t)$.
 Observe that $\overline{\Omega}\cap H_t=\overline{\Omega}\cap H$.
 Let $\pi:{\mathcal C}(\Omega)\longrightarrow\Omega$ be radial projection centered at $q$.
 Then $\partial(\pi W )=\pi(\partial W)=H\cup \pi(S\cap H_t)$. Now $S$
 is real-analytic and strictly convex, thus so is $S\cap H_t$ and its image under $\pi$.  
 Define $\Omega'$ to be the interior of $W$.
 Since $H_t$ is preserved by $\Gamma,$ so is $W$ and hence $\Omega'.$
\end{proof}

\noindent
{\bf Example.} It follows from \ref{maximalcuspsstandard} that every parabolic in a {\em finite 
volume} strictly convex orbifold
is conjugate into $O(n,1)$.
What follows is an example  of a parabolic isometry of a strictly convex
domain
not conjugate into $O(n,1)$. 
Consider the one-parameter parabolic subgroup  $\Gamma < \SL(5,{\mathbb R})$ 
$$\exp(tN)=\left(\begin{array}{ccccc}
1 & t & t^2/2! & t^3/3! & t^4/4!\\
0 &1 & t & t^2/2! & t^3/3!\\
0 &0 &1 & t & t^2/2!\\
0 & 0 &0 &1 & t \\
0 & 0 & 0 & 0 & 1\end{array}\right).$$
The orbit of $[e_5]$ is the affine curve in ${\mathbb R}P^4$ given by $[t^4/4!:t^3/3!:t^2/2!:t:1]$.
 Let $\Omega$ be the interior of the convex hull  of this curve. Then $\Omega$ is  properly (but not strictly) 
 convex and is preserved by $\Gamma$. 
 The boundary of $\Omega$ is the ruled $3$-sphere consisting of the set 
 of convex combinations  of pairs of points on this curve. The supporting hyperplane $H$ given
 by omitting $e_5$  meets $\overline{\Omega}$ at a single point. It
  follows from \ref{strictlyconvexcusps}
 there is another {\em strictly convex} domain $\Omega'\subset\Omega$  preserved by $\Gamma$ and 
 which is $C^1$ except at $p$.
 
\gap
\begin{remark}\label{swan} By a theorem of Auslander and Swan \cite{Swan}, 
every polycyclic group is a subgroup of $GL(n,{\mathbb Z})$. 
If $G$ is a finitely generated nilpotent group then it is polycyclic. Thus $G$ is the orbifold fundamental
group of a cusp for the Siegel upper half space E(iv). \end{remark}

In contrast
a maximal cusp group is a Euclidean
crystallographic group, and therefore virtually abelian: see section \ref{maximal_cusps_hyperbolic}.

\section{Work of Benz\'ecri and Vinberg}

We shall make frequent use of results of Benz\'ecri \cite{benzecri} and Vinberg \cite{vin}.
Simplified  proofs of these results are in Goldman \cite{goldman1} pages 49--63. 

Let $\mathfrak C$ be the set of all  properly convex compact subsets in ${\mathbb R}P^n$ with non-empty interior 
and equip this with the Hausdorff topology. Let ${\mathfrak C}_*$ be the space of all $(C,p)\in{\mathfrak C}\times{\mathbb R}P^n$ with 
$p$ a point in the interior of $C$ and equipped with the product topology. 

\begin{theorem}[Benz\'ecri compactness]\label{benzecricpct} The quotient of ${\mathfrak C}_*$ by the 
natural action of $PGL(n+1,{\mathbb R})$ is compact.
\end{theorem}
\noindent Given a metric space $X$ with metric $d$  the closed ball in 
 $X$ center $p$ radius $r$ is $$B_r(p;X,d)=\{\ x\in X\ : d(x,p)\le r\ \}.$$
In what follows $B(r)$ denotes the closed ball of Euclidean radius $r$ centered on 
 the origin in Euclidean space. 
\begin{corollary}[Benz\'ecri charts, \cite{goldman1} page 61 C.24]\label{benzecri} For every $n\ge 2$ there is a constant
$R_{\mathcal B}=R_{\mathcal B}(n)>1$ with the following property:

If $\Omega\subset{\mathbb R}P^n$ is a properly convex open set and $p\in\Omega$ then 
there is a projective automorphism $\tau$ called a {\em Benz\'ecri chart} such that 
$B(1)\subset\tau(\Omega)\subset B(R_{\mathcal B})\subset{\mathbb R}^n$ and $\tau(p)=0.$
\end{corollary}

\noindent An open convex set $\Omega$ is called a {\em Benz\'ecri domain} if 
$B(1)\subset\overline{\Omega}\subset B(R_{\mathcal B}(n)).$ It is routine to show: 
\begin{proposition}\label{bdomaincpct} Let  ${\mathcal B}$ 
be the set of all Benz\'ecri domains in ${\mathbb R}^n.$ Then ${\mathcal B}$ is compact with the Hausdorff 
metric induced by the Euclidean metric
on ${\mathbb R}^n.$ 
\end{proposition}
\begin{corollary}[Hilbert balls are uniformly bilipschitz]\label{hilbertballs}
  For every dimension $n\ge 2$ and $r>0$:
\begin{itemize}
\item There is $K=K(n,r)>0$ such that
for every properly convex domain $\Omega\subset{\mathbb R}P^n$ and $p\in\Omega$ there
is a $K$-bilipschitz homeomorphism from $B_r(p;\Omega,d_{\Omega})$ to $B(r).$ \\
\item There is $K_{\mu}=K_{\mu}(n,r)>0$ such that 
if $\Omega$ is a Benz\'ecri domain  and $\mu_{\Omega}$ is the Hausdorff measure on $\Omega$ 
induced by the Hilbert metric and $\mu_L$ is  Lebesgue measure on ${\mathbb R}^n$
 then for every open set $U\subset B_r(0;\Omega,d_{\Omega})$ 
 $$K_{\mu}^{-1}\cdot\mu_L(U)\le \mu_{\Omega}(U) \le K_{\mu}\cdot \mu_L(U).$$
 \end{itemize}
\end{corollary}

\gap
 
Suppose ${\mathcal C}={\mathcal C}(\Omega)\subset V$ is a sharp convex cone and ${\mathcal C}^*\subset V^*$ is the dual cone.
 Let $d\psi$ be a volume form on $V^*$.
The {\em characteristic function} $f:{\mathcal C}\longrightarrow{\mathbb R}$  defined by
$$f(x)=\int_{{\mathcal C}^*}e^{-\psi(x)}d\psi$$
is  real analytic and $f(tx)=t^{-1}f(x)$ for $t>0$. 
For each $t>0$ the level set $S_t=f^{-1}(t)$ is called a {\em Vinberg hypersurface}. It is the boundary
of the {\em sublevel set} ${\mathcal C}_t=f^{-1}(0,t]\subset {\mathcal C}.$ For example, the hyperboloids $z^2=x^2+y^2+t$ are Vinberg hypersurfaces in the cone $z^2>x^2+y^2.$ 

\begin{theorem}[Vinberg \cite{vin}, see also \cite{goldman1} (C1), (C6) pages 51--52]\label{Vinbergsurface} 
The Vinberg hypersurfaces are an analytic foliation of ${\mathcal C}.$
\begin{itemize}
\item The radial projection $\pi:S_t\longrightarrow\Omega$ is a diffeomorphism.
\item ${\mathcal C}_t$ has smooth strictly convex boundary.
\item $S_t$ is preserved by $SL({\mathcal C}).$ 
\end{itemize}
\end{theorem}

At each point $p$ on a Vinberg surface there is a unique supporting 
tangent hyperplane $\ker df_p$. This gives a {\em duality map}
 $\Phi_{\Omega}:\Omega\longrightarrow\Omega^*$. Another description of this map is that $\Phi_{\Omega}(x)$ is the centroid
 of the intersection of ${\mathcal C}^*$ with the hyperplane $\{\ \psi\in V^*\ :\ \psi(x)=n\ \} \subset V^*$. 
 Benz\'ecri's compactness theorem has the following consequences. 
  \begin{theorem}\label{dualitybilip} $\Phi_{\Omega}$ is $K$-bilipschitz with respect to the Hilbert metrics where $K=K(n)$ only depends on $n=\dim\Omega$.
 \end{theorem}
 \begin{corollary}\label{cor:finvol-iff-dual-finvol}
 The duality map descends to a $K$-bilipschitz map between a properly convex orbifold $M$ and its dual $M^*.$
 In particular, $M$ has finite volume if and only if $M^*$ has finite volume.
 \end{corollary}

\section{The Margulis lemma}\label{margulissection}

 \begin{theorem}[Isometry Bound]\label{isometrybound} For every $d>0$ there is a compact subset 
$K\subset SL(n+1,{\mathbb R})$ with the following property. Suppose that  $\Omega$ is a Benzecri domain 
and $A\in SL(\Omega)$ moves the origin a distance at most $d$ in the Hilbert metric on $\Omega$.

 Then $A\in K$.
 \end{theorem}
 
 There is a more invariant version which follows immediately from Theorem \ref{isometrybound} and Theorem \ref{benzecri}: 
 For every $d>0$ there is a compact subset  $K\subset SL(n+1,{\mathbb R})$  so that if $\Omega$ is any properly 
convex domain and $p$ is a point in $\Omega$ and $S=S(\Omega,p,d)$ is the subset of $SL(\Omega)$ 
consisting of all maps that move $p\in\Omega$ a distance at most $d$ in the Hilbert metric on $\Omega$, then $S$ is 
conjugate into $K$. i.e. there is $B\in SL(n+1,{\mathbb R})$ such that $B\cdot S\cdot B^{-1}\subset K.$

 \begin{proof} Let $p$ denote
 the origin. Suppose we have a sequence $(\Omega_k,A_k)$ where each $\Omega_k$ is a Benzecri domain
 and $A_k\in SL(\Omega_k)$ moves $p$ a Hilbert distance at most $d.$ It suffices to show $A_k$ has 
  a convergent subsequence in $SL(n+1,{\mathbb R}).$
 
 By \ref{bdomaincpct} we can pass to a subsequence so that $\Omega_k$ converges to
 a Benzecri domain $\Omega_{\infty}.$
Choose a projective basis ${\mathcal B}=(p_0,p_1,p_2,\cdots,p_{n+1})$ in $B(1/10).$ This ensures that
${\mathcal B}\subset B_1(p;\Omega,d_{\Omega})$  for every Benzecri domain $\Omega.$
We can choose a subsequence so that the projective bases ${\mathcal B}_k=A_k({\mathcal B})$ converge to an
(n+2)-tuple ${\mathcal B}_{\infty}=(q_0,\cdots q_{n+1})\subset\Omega_{\infty}.$
We need to show this set is a projective basis. 

Since every $A_k$ moves $p$ a distance at most $d$,
 it follows that ${\mathcal B}_{\infty}\subset B_{d+1}(p;\Omega_{\infty},d_{\Omega_{\infty}}).$
Let $\sigma_i$ be the $n$-simplex with vertices  ${\mathcal B}\setminus\{p_i\}$. Since metric balls are convex \ref{convexballs}, it follows that
$\sigma_i  \subset B_{d+1}(p;\Omega_{\infty},d_{\Omega_{\infty}}).$ Note
that each $A_i$ has determinant $1$, so preserves Lebesgue measure.

Let $V=\left(K_{\mu}(n, d+1)\right)^{-1}\min_{i} \mu_{L}(\sigma_i)$. It follows from \ref{hilbertballs}
 that $\mu_{\Omega_k}(\sigma_i)\ge V.$
Let $\sigma^{\infty}_i$ be the possibly degenerate $n$-simplex
with vertices the $(n+2)$-tuple ${\mathcal B}_{\infty}$ with $q_i$ deleted. Then
$\sigma^{\infty}_i=\lim_k A_k(\sigma_i).$ 
It is easy to see that $\mu_{\Omega_{\infty}}(\sigma^{\infty}_i)=\lim_k \mu_{\Omega_k}(A_k\sigma_i)\ge V>0.$
 In particular $\sigma^{\infty}_i$ is not degenerate
 therefore ${\mathcal B}_{\infty}$ is a projective basis.
There is a unique element  $A_{\infty}\in SL(n+1,{\mathbb R})$ sending ${\mathcal B}$ to ${\mathcal B}_{\infty}.$
 It is easy to check that $A_{\infty}=\lim A_k.$  \end{proof}

From (6.2.3) in Eberlein \cite{eb} we have:

\begin{proposition}[Zassenhaus neighborhood]\label{znbhd} There is a neighborhood $U$ of the identity in 
$SL(n+1,{\mathbb R})$ such that if $\Gamma$ is a discrete subgroup of $SL(n+1,{\mathbb R})$ then the 
subgroup generated by $\Gamma\cap U$ is nilpotent. 
\end{proposition}

The following statement {\em and proof} is essentially (4.1.16) in Thurston \cite{wpt1}.  However the
hypotheses are different.
\begin{proposition}[short motion almost nilpotent]\label{projectivemargulis-technical} 
For every dimension
$n\ge 2$ there there is an integer $m>0$ and a Margulis constant $\mu>0$
with the following property:

 Suppose that $\Omega$ is a properly convex domain and $p$ is a point in $\Omega$ and 
 $\Gamma\subset SL(\Omega)$ is a discrete subgroup generated by isometries that move $p$ 
 a distance less than $\mu$ in the Hilbert metric on $\Omega.$  Then 
 \begin{enumerate}
\item There is a normal nilpotent 
 subgroup of index at most $m$ in $\Gamma.$
\item $\Gamma$ is
contained in a closed subgroup of $SL(n+1,{\mathbb R})$ with no more than
$m$ components and with a nilpotent identity component.  
\end{enumerate}
\end{proposition}
\begin{proof} 
By Theorem \ref{benzecri} we may assume $\Omega$ is a  Benzecri domain and $p$ is the origin. 
Let $K\subset SL(n+1,{\mathbb R})$ be a compact subset as
provided by \ref{isometrybound} when $d=1$ (for example).  
Since $K$ is compact,
it is covered by some finite number, $m,$ of left translates of the Zassenhaus
neighborhood $U$ given by \ref{znbhd}.  Define $\mu=d/m.$

  Let $W\subset SL(\Omega)$ be the subset of
 all $A$ such that $A$ moves $p$ a distance less than $\mu.$ 
 Then $W=W^{-1}$ and $W^m\subset K.$ 
By hypothesis the  group $\Gamma$ is generated by $\Gamma\cap W$.
 Define $\Gamma_U$ to be the nilpotent subgroup generated by $\Gamma\cap U$. 
 We claim there are at most $m$
 left cosets of $\Gamma_U$ in $\Gamma.$ 
 
  Otherwise there  are $m+1$ distinct left cosets of $\Gamma_U$  
  which have representatives each of which is the product of at most $m$ elements of an
  arbitrary symmetric generating set of $\Gamma$  (see \cite{wpt1}, 4.1.15).
   Choose the symmetric generating set $\Gamma\cap W\subset W.$
 Hence these representatives are in $W^m\subset K.$ But $K$ is covered
 by $m$ left cosets of $U.$ Thus there are two representatives
  $g,g'\in \Gamma\cap W^m$  such that $g,g'$ are in the same left translate of $U.$
Thus $g^{-1}g'\in\Gamma\cap U\subset\Gamma_U,$ hence $g\Gamma_U=g'\Gamma_U$
  which contradicts the existence of $m+1$ 
distinct cosets of $\Gamma_U$ in $\Gamma.$
 It follows that $\Gamma_U$ has index at most $m$ in $\Gamma.$ 
 
 It remains to prove there is a {\em normal} subgroup of index at most $m$ 
 and the statement concerning the closed subgroup. We follow the last
 three paragraphs of Thurston's  proof
 (4.1.16)  \cite{wpt1}
 {\em verbatim}, subject only to the change that he uses $\epsilon$ in place of our $\mu.$ During
 the course of that proof, $m$ is replaced by another constant.
 \end{proof}

The proof of the projective Margulis lemma \ref{projectivemargulis} follows
from this.

\section{thick-thin Decomposition}\label{thickthinsection}

This section contains proofs of  Theorem  \ref{thickthin},  the thick-thin decomposition for strictly convex orbifolds 
and,  in the finite volume case, Theorem  \ref{convexthin}, a variant  where the thinnish components are convex.
The {\em thinnish part} is a certain submanifold constructed below
 such that everywhere on the boundary the injectivity radius lies between two 
constants related to the Margulis constant and depending only on dimensions. The reason for this approach
is that the authors do not know if the set of points moved a distance at most $R$ by a projective isometry
is a convex set.

The proof in outline: When $\Omega$ is {\em strictly} convex  the holonomy of each component of the thin part of $\Omega/\Gamma$ 
is an elementary group \ref{cusplemma}. This follows from the fact \ref{elementarypartition} that
in the strictly convex case  maximal elementary subgroups
 partition the non-trivial elements of $\Gamma$.
 In the properly convex case this partition breaks down.
 A component of the thin
part has preimage in $\Omega$ which contains a union of subsets
 each consisting of the convex hull of the set of points moved a distance $3^{-n}\mu_n$ by some particular
element of $\Gamma$. Points in this convex hull are moved at most $\mu_n$, \ref{CHbound}. The union of these sets is {\em starshaped}
and this yields the topology of the components of the thin part.  

Suppose $M$ is a strictly convex projective $n$-manifold.
The {\em injectivity radius} $\inj(x)$ at a point $x$ in $M$ is 
the supremum of the radii  of embedded metric balls in $M$ centered at $x.$
Since metric balls are convex, this equals half the length of the shortest
non-contractible loop based at $x$.

 The {\em local fundamental group} at $x$ is the  subgroup 
$\pi_1^{loc}(M,x)$ of 
$\pi_1(M,x)$ generated by the homotopy classes of loops based at $x$ with length less than 
the $n$-dimensional Margulis constant $\mu=\mu_n.$
 The local fundamental  group at $x$ is  
trivial if the injectivity  radius at $x$ is larger than $\mu/2.$
The Margulis lemma  \ref{projectivemargulis-technical} 
implies that the local fundamental group is always virtually nilpotent and  by \ref{elementarypartition}:
\begin{lemma}\label{cusporcyclic} Suppose that $M$ is a strictly convex projective $n$-manifold. 
Then $\pi_1^{loc}(M,x)$ is elementary or trivial for all $x$.
\end{lemma}

\noindent Given $\epsilon>0$ the  {\em open $\epsilon$-thin} part of $M$ is
$$\thin_{\epsilon}(M) = \{\ x\in M\ :\ \inj(x)< \epsilon\ \}.$$

\begin{lemma}[thin holonomy is elementary]
\label{cusplemma} 
Suppose that $M=\Omega/\Gamma$ is a
 strictly convex projective $n$-manifold  and $N$ is a component of $\thin_{\mu/2}(M)$. 
Then the holonomy, $\Gamma_N,$  of $N$ is elementary and either hyperbolic
or parabolic.\end{lemma}
\begin{proof} Let $\pi:\Omega\longrightarrow M$ be the natural projection
and let $\tilde{N}\subset\Omega$ be a component of $\pi^{-1}(N).$ For each
$\tilde{x}\in\tilde{N}$ let $\Gamma(\tilde{x})$ be the subgroup of $\Gamma$
generated by isometries which move $\tilde{x}$ less than $\mu.$ This group
may be identified with the local fundamental group at $\pi(\tilde{x}).$
Since $N\subset \thin_{\mu/2}(M)$ this group is nontrivial.  By
\ref{projectivemargulis} it is virtually nilpotent, and so by \ref{elementary}
it is elementary.  By \ref{elementarypartition} there is a unique maximal elementary 
group, $E(\tilde{x})$,  containing $\Gamma(\tilde{x})$.

If two points $\tilde{x}_1,\tilde{x}_2$ in $\tilde{N}$
are sufficiently close then $\Gamma(\tilde{x}_1)$ and $\Gamma(\tilde{x}_2)$
have nontrivial intersection, so $E(\tilde{x}_1)=E(\tilde{x}_2)$.
 It follows that $\tilde{N}$ is partitioned into clopen
subsets with the property that  on each subset,  
$E(\tilde{x})$ is constant.
Since $\tilde{N}$
is connected it follows that  $E(\tilde{x})$ is
constant as $\tilde{x}$ varies over $\tilde{N}$. Thus there is a unique maximal
elementary group $E(\tilde{N})=E(\tilde{x})$ which contains $\Gamma(\tilde{x})$ for every $\tilde{x}\in\tilde{N}$.

Let $G$ be the normal subgroup of $\Gamma_N$ generated by unbased 
loops in $N$ of length less
than $\mu.$ Then $G$ is a nontrivial normal subgroup of $\Gamma_N$ and the
argument of the preceding paragraph shows that $G \subset E(\tilde{N})$ and in
particular is elementary.  Normality implies that
$\Gamma_N$ preserves the set of fixed point of $G$, and by strict convexity
there are at most two fixed points.  Arguing as in \ref{elementary} it
follows that $\Gamma_N$ fixes each of these points and is therefore
elementary. This group is hyperbolic or parabolic by \ref{elementaryimpliesvirtnil}.
 \end{proof}

In a space of negative sectional curvature, (or more generally, in a space satisfying Busemann's definition
of negative curvature, see \cite{Bus55} Chap.\thinspace 5),  the set of points moved a distance at most $R$ by an isometry is 
convex. However we do not
know if this is true for Hilbert metrics which need not satisfy Busemann's definiton.
The convex hull of this set is used to overcome this.
 
\begin{lemma}[Carath\'eodory's Theorem]\label{convexhull} Suppose that $S$ is a non-empty subset of a 
properly convex domain $\Omega$. 

Then the convex hull of $S$ in $\Omega$ is the union of the projective simplices with vertices in $S.$
\end{lemma}
\begin{proof} This follows from the fact that the projective convex hull is the Euclidean convex hull, and this 
statement is  due to Carath\'eodory (see Berger \cite {Be1} (11.1.8.6)) in the latter case.
\end{proof}

\begin{lemma}[convex hull bound]\label{CHbound} Suppose that $\tau$ is an isometry of a properly convex 
domain $\Omega$ and that $N$ is the subset of $\Omega$ of all points moved a distance at most $R$ by $\tau$. 

Then every point in the convex hull of $N$ is moved a distance at most $3^n\cdot R,$ where $n={\dim(\Omega)}$.
\end{lemma}

\begin{proof} By \ref{convexhull} it suffices to show that if the vertices of an $n$-simplex $\Delta$ are moved a 
distance at most $R$ then every point in $\Delta$ is moved a distance at most $3^n R.$ We prove this by 
induction on $n.$ For $n=1$ a $1$-simplex $\Delta=[a,b]$ is a segment. Then $\tau[a,b]=[c,d]$ is another 
segment. The image of $x\in[a,b]$ is a point $\tau(x)\in[c,d].$ By assumption $d_{\Omega}(a,\tau a)\le R$ 
and $d_{\Omega}(b,\tau b)\le R.$ 
The domain of the function $f:[c,d]\longrightarrow{\mathbb R}$ given by $f(x)=d_{\Omega}(x,[a,b])$ is compact and convex. 
Since $f(c),f(d)\le R$ it follows by the maximum principle \ref{maxprinciple} every point of $[c,d]$ is within 
$R$ of some point on $[a,b].$ Thus 
for $x\in[a,b]$ we see that $\tau(x)\in[c,d]$ is within distance $R$ of some point $y\in[a,b],$ 
$$d_{\Omega}(\tau(x),y)\le R.$$ Without loss of generality, assume $y$ is between $x$ and $b.$ Then from the triangle inequality we get
$$d_{\Omega}(a,y)\le d_{\Omega}(a,\tau(a))+d_{\Omega}(\tau(a),\tau(x))+d_{\Omega}(\tau(x),y).$$
Using that $\tau$ is an isometry gives $d_{\Omega}(\tau(a),\tau(x))=d_{\Omega}(a,x).$ Also $x$ is between $a$ and $y$ 
so  $$0\le d_{\Omega}(a,y)-d_{\Omega}(a,x)\le d_{\Omega}(a,\tau(a))+d_{\Omega}(\tau(x),y) \le 2R.$$ Since $x$ is on 
the segment $[a,y]$ from this we get $$d_{\Omega}(x,y)\le 2R.$$
Now $d(y,\tau(x))\le R$ so applying the triangle inequality again gives
$$d_{\Omega}(x,\tau(x))\le d_{\Omega}(x,y)+d_{\Omega}(y,\tau(x))\le 3R.$$
This proves the inductive statement for $n=1$.

Suppose $\Delta'$ is an $(n-1)$ simplex and  $\Delta=a*\Delta'.$ Consider a point $x$ in $\Delta.$ Then $x$ lies on 
a segment $[a,b]$ with $b\in\Delta'.$ By induction $d_{\Omega}(b,\tau(b))\le 3^{n-1}R.$ Also 
$d_{\Omega}(a,\tau(a))\le R\le 3^{n-1}R.$ By induction applied to the $1$-simplex $[a,b]$ we get that every 
point on $[a,b]$ is moved a distance at most $3\cdot\left( 3^{n-1}R\right).$ This completes the proof.
\end{proof}

If $M=\Omega/\Gamma$ is a strictly convex projective $n$-manifold  then a {\em Margulis tube} is a  tubular 
neighborhood, $N,$ of  a simple geodesic $\gamma$ in $M$ such that  at every point in $\partial N$ the 
injectivity radius is at least $\iota_n=3^{-n-1}\mu_n$. In the following the dimension $n$ is
fixed and we use $\iota=\iota_n$ and $\mu=\mu_n$.

\begin{proof}[Proof of Theorem \ref{thickthin}] We adapt  the discussion of the 
thick-thin decomposition of 
hyperbolic manifolds in Thurston \cite{wpt1} \S 4.5.  to construct $A.$

Suppose $M=\Omega/\Gamma$ is strictly convex. For a nontrivial 
element $\gamma \in \Gamma$ let $T(\gamma)$ be the open subset of $\Omega$ which is the
interior of the convex hull of all points moved by $\gamma$ a distance less
than $3\iota.$ By \ref{CHbound} every point in $T(\gamma)$ is moved a
distance at most $\mu$ by $\gamma$.   We note for later use that if $\gamma$ is parabolic 
it is easy to see that $T(\gamma)$ is starshaped at $p$.  

If $y$ is a point in the intersection of
$T(\gamma_1)$ and $T(\gamma_2)$ then $\gamma_1$ and $\gamma_2$ both move
$y$ at most $\mu$, so that by \ref{cusporcyclic}, $\gamma_1$ and $\gamma_2$ are
contained in the same elementary subgroup  $S\le \Gamma.$ 
In fact we claim the converse also holds: If $\gamma_1$ and $\gamma_2$ are 
contained in the same elementary
group $E$ then $T(\gamma_1)$ and $T(\gamma_2)$ intersect, provided they
are both nonempty.

First suppose that $E$ is hyperbolic.  Then it is cyclic generated by some
element $\gamma.$ Each $\gamma_i$ is a power of this element $\gamma$ and $T(\gamma_i)$
contains the axis of $\gamma.$ Hence $T(\gamma_1)\cap T(\gamma_2)$ contains
this axis.  The other case is that $E$ is parabolic.  By \ref{nonhypsmalldist} 
there is a point $x$ in $\Omega$ moved less than $3\iota$ by both $\gamma_1$ and $\gamma_2$.
Thus $x\in T(\gamma_1)\cap
T(\gamma_2)$ which proves the claim.

 Write $T(\gamma_1)\sim T(\gamma_2) $ if their intersection is not
 empty, the argument of the previous paragraph shows that this defines an 
equivalence relation.

Let $\widetilde{U}\subset\Omega$ be the union of all the $T(\gamma)$ for
nontrivial $\gamma.$  To each $T(\gamma)$ we may assign a 
maximal elementary subgroup of $\Gamma$, by assigning 
to each point $p$ in $\widetilde{U}$ the maximal elementary subgroup which 
stabilizes the component of $\widetilde{U}$ containing $p.$ 
This map is constant on connected components and induces a
bijection between those components and ${\mathcal E}$, 
a certain subset of the maximal elementary subgroups of $\Gamma$.
 Let $\theta:\widetilde{U}\longrightarrow{\mathcal E}$ be this function, so
 that connected components of $\widetilde{U}$ correspond to
elements of ${\mathcal E}.$

Clearly $\widetilde{U}$ is preserved by $\Gamma$. Also,  if $\widetilde{V}$
 is a component of $\tilde{U}$ then $\tilde{V}$ is preserved by the
 elementary group $E=\theta(\widetilde{V})$ and if for $\gamma \in \Gamma$, 
 $\gamma\widetilde{V}$ intersects
 $\widetilde{V}$ then it equals $\widetilde{V}.$ The image 
  of $\widetilde{U}$ in $M$ is an open submanifold, $U$, of the $\mu_n/2$-thin part of $M$ 
 and each $V =\widetilde{V}/E$ is a  component of $U$.

We will determine the topology of $V$ and  construct $A$ by removing 
from $V$ an open
collar, to give a metrically complete submanifold with smooth
boundary.
By \ref{cusplemma}, $E$ is elementary, and either
 hyperbolic or parabolic.

 The first case is that $E$ is parabolic and we claim that $V$ is an open cusp.
 There is a parabolic fixed point
$p.$  As noted above $\tilde{V}$ is the union of sets
which are starshaped at $p$ and is therefore starshaped at $p$.
It only remains to show that $V$ is a thickening of a convex cusp.
By \ref{elementaryimpliesvirtnil} $E$ contains a nilpotent subgroup $E'$ of finite index.
Let $\gamma$ be a non-trivial element in the center of $E'$. Then $T(\gamma)$ is
convex and preserved by $E'$. Let $\delta_1,\cdots,\delta_k$ be a set of left coset
representatives of $E'$ in $E$. Each group element $\gamma_i=\delta_i\gamma\delta_i^{-1}$
preserves a convex set $T_i=T(\gamma_i)=\delta_i T(\gamma)$. The action of $E$ permutes these sets. By \ref{nonhypsmalldist}
there is $x\in\Omega$ moved a distance less than $3\iota$ by each of $\gamma_1,\cdots,\gamma_k$.
It follows that $K=T_1\cap\cdots\cap T_k$ is not empty. It is convex and preserved by $E$.
Thus $K/E$ is a convex core for $V$. This proves $V$ is an open cusp.

Otherwise $E$ is hyperbolic and  infinite cyclic with
 some generator $\gamma$ that has axis $\ell.$ Here is a sketch of the argument:
We show that $\widetilde{V}$  is a union of open convex sets each of which 
contains $\ell.$ This will  imply that 
 $\widetilde{V}$ is star-shaped with respect to points on $\ell$ and hence an
 ${\mathbb R}^{n-1}$-bundle over $\ell.$ The bundle structure is preserved
 by $E.$ This in turn implies that $\widetilde{V}/E$ is diffeomorphic to an ${\mathbb
 R}^{n-1}$-bundle over the circle which is the short geodesic $\ell/E.$
 Hence $V$ in this case is a Margulis tube.

 Here are the details: There is a projection
 $\pi_{\ell}:\Omega\longrightarrow\ell$ given by \ref{isometryfoliation}.
 The fibers of the restriction $\pi_{\ell}|:\tilde{V}\longrightarrow\ell$
 are not copies of ${\mathbb R}^{n-1}$ but only open \& star-shaped.  An open
 star-shaped set is diffeomorphic to Euclidean space.  We must identify the
 fibers smoothly with Euclidean space as we move around in this bundle.
 
 Choose a smooth complete Riemannian metric on $V$ and lift it to an
 $E$-equivariant Riemannian metric $ds$ on $\tilde{V}.$ The pencil of
 hyperplanes from \ref{isometryfoliation} intersects along a
 codimension-$2$ projective hyperplane, $Q.$ Pass to the $2$-fold cover
 $S^n$ of the ${\mathbb R}P^n$ which contains $\Omega$.
 The preimage of $Q$ is a codimension-$2$ sphere
 $S^{n-2}.$ Let $\pi_S:\Omega\setminus\ell\longrightarrow S^{n-2}$ be
 radial projection along the (cover of the) pencil.  This map is smooth: it
 is the projectivization of a linear map.
  
 Define $h: V \longrightarrow{\mathbb R}$ as follows.  Given $x\in V$
 there is a unique segment $[x,y]$ in $\Omega$ contained in one of the
 hyperplanes in the pencil and with $y\in\ell.$ Define $h(x)$ to be the
 $ds$-length of this segment.  Then $h$ is smooth except along $\ell.$
 Regard $S^{n-2}$ as the unit sphere in ${\mathbb R}^{n-1}$
 centered on $0.$ The hyperbolic $\gamma$ preserves $Q$ and acts on it as a
 projective transformation.  The map $g:\widetilde{V}\longrightarrow {\mathbb
 R}^{n-1}$ defined by $g(x)=h(x)\cdot \pi_S(x)$ restricted to a fiber of $\pi_{\ell}$ is a diffeomorphism and is
 $E$-equivariant.  Hence the map $k:\widetilde{V}\longrightarrow
 \ell\times{\mathbb R}^{n-1}$ given by $k(x)=(\pi_{\ell}(x),g(x))$ is an
 $E$-equivariant diffeomorphism.  Thus it covers a diffeomorphism
 $V\longrightarrow (\ell\times{\mathbb R}^{n-1})/E.$ The target is the
 desired smooth vector bundle.

Next we show that the thick part is not empty. It follows from \ref{unboundeddisplacement}
that $M$ can't consist of a single Margulis tube, and it follows from \ref{fatcusps}
that $M$ cant consist of a single cusp contained in the thin part. Hence $M\ne U$.

It remains to describe the manifold $A$, as a submanifold of $U.$ If a component $V$ of
$U$ is diffeomorphic to an ${\mathbb R}^{n-1}$ bundle, choose the smallest
sub-bundle with fiber the closed ball of radius $R$ centered at $0$ subject
to the condition it contains all points moved at most
$(2/3)3\iota=2\iota.$ (Here one could replace $2/3$ by any number $0 < \lambda < 1$.)
Thus on the boundary the injectivity radius is at least $(1/2)(2\iota)=\iota$.
 If $V$ is an open cusp it follows from \ref{productcusp}(C6)
 that it contains a closed cusp  satisfying the the same condition. To apply (C6)
 one needs a slightly smaller open cusp. To obtain this, perform the above construction, but
 using the convex hull of points moved a distance $2\iota$.
\end{proof}
\noindent
{\bf Remark.} With more work one can show that in the cusp case $V$ is $K$ with a collar attached.
Then using Siebenmann's open collar theorem \cite{Sieb} it follows that in  dimensions greater than four
 $V/E$ is $K/E$ with an open collar attached. Thus in dimension $\ne4$ the interior of a cusp component of the thin part is diffeomorphic to a full cusp. 

\gap
For some applications it is useful to have the components of the thin part be
convex. This is possible if control of the injectivity radius on the boundary is loosened:
\begin{proposition}[Convex and thin]\label{convexthin} 
Suppose that $E$ is a component of the thin part of a strictly convex $n$-manifold $M=\Omega/\Gamma$
of finite volume.

Then the interior of $E$ contains a closed subset $C$ which is a convex submanifold 
such that the closure of $E\setminus C$ is a collar of $\partial E$.

 Furthermore, there is a constant,
$\mu'=\mu'(n,d),$ depending only on dimension and $d=\diam(\partial E)$
 such that the injectivity radius at every point
of $\partial C$ is greater than $\mu'.$ Either $C$ is a horocusp or a metric $r$-neighborhood of a geodesic.
\end{proposition}
\begin{proof} Let $\pi:\Omega\longrightarrow M$ be the projection
and $\tilde{E}$  a component of $\pi^{-1}E.$
The first case is that $E$ is a cusp.
 There
is a unique parabolic fixed point $p\in\bOmega$
in the closure of $\tilde{E}.$ Let ${\mathcal B}_t$ be the horoballs
centered at $p$ parameterized so that 
${\mathcal B}_t\subset \tilde{E} \Leftrightarrow t\le 0.$
The horocusp  $C=\pi({\mathcal B}_{-1})$ is contained in the interior of $E.$

Let $\ell_q$ be a line with endpoints $p\ne q\in\bOmega$.
This line meets both $\partial\tilde{E}$ and $\partial{\mathcal B}_t$ in unique points.
It follows that the region between $\partial \tilde{E}$ and $\partial{\mathcal B}_{-1}$
is foliated by intervals each contained in such a line and thus
 the region between  $\partial E$ and $C$ is a collar
of $\partial E.$

Since $\partial \tilde{E}$ separates ${\mathcal B}_{-1}$ from ${\mathcal B}_d$ every line
 $\ell_q$ meets $\partial\tilde{E}$ between ${\mathcal B}_{-1}$ and ${\mathcal B}_d$. It
  follows that every point in 
${\mathcal B}_{-1}$ is within a distance $d+1$ of  $\tilde{E}.$ Projecting it follows that
every point in $\partial C$ is within a distance $d+1$ of a point in $\partial E.$
By the uniform bound on decay, the injectivity radius at each point
of $\partial C$ is bounded above and below in terms of $\mu$ and $d.$ This completes the cusp
case.

 The other case is that $E$ is a Margulis tube. Let $\gamma$ be the core geodesic.
 Then $\tilde{E}$ is a neighborhood of a line $\tilde{\gamma}$ covering $\gamma.$
 Let $r$ be the smallest distance between  a point on $\partial E$ and  $\gamma$.
  Let ${\mathcal B}_t$ denote the set of points in $\Omega$ distance $(r+t)$ from $\tilde{\gamma}.$
 By \ref{convexnbhds} this set is convex. Set $\delta=\min(1,r/2)$ then 
 ${\mathcal B}_{-\delta}$ not empty and is contained in the interior
 of $\tilde{E}.$ Thus ${\mathcal B}_{-\delta}\subset \tilde{E}\subset{\mathcal B}_d$ and we define
  $C=\pi({\mathcal B}_{-\delta})$.
  Let $p:\Omega\longrightarrow\tilde{\gamma}$ be the nearest point projection.
 The fibers of this map are lines. 
 The argument for cusps is easily adapted to this setting with
 the lines $\ell_q$ replaced by fibers of $p$ to show that $C$ has the required properties.
  \end{proof}

In particular every cusp component of the thin part of a {\em finite volume} manifold 
contains a horocusp. 
The thin part of $M={\mathbb H}^4/\langle\gamma\rangle,$ where $\gamma$ is a parabolic
that induces a Euclidean screw-motion on a horosphere,  contains no horocusp. The set of
points moved a distance at most $d$ by a Euclidean screw motion in ${\mathbb E}^3$ 
is a tubular neighborhood of a line. Thus  the thin part of $M$ intersects a horomanifold
in a Euclidean solid torus. The radius of this solid torus increases moving towards the parabolic fixed point
but is bounded above.

\section{Maximal Cusps are Hyperbolic}
\label{maximal_cusps_hyperbolic}

This section proves Theorem \ref{maximalcuspsstandard}: a maximal cusp in a 
properly convex projective orbifold is projectively equivalent to a cusp in a 
complete (possibly infinite volume) {\em hyperbolic} orbifold. 
It follows that a cusp cross-section is diffeomorphic to a compact Euclidean
orbifold.

A parabolic in $O(n,1)$ is a {\em pure translation} if every eigenvalue is $1$. 
The starting point is a  characterization of ellipsoids in projective space (cf \cite{Socie2}):
\begin{theorem}[ellipsoid characterization]\label{ellipsoid} 
 Suppose that $\Omega$ is strictly convex of dimension $n$ and that   $W \subset SL(\Omega,p)$ 
is a nilpotent group which   acts simply-transitively on $\bOmega\setminus\{p\}.$ 

   Then $\bOmega$ is an ellipsoid and $W$ is conjugate to the subgroup of pure translations in some
   parabolic subgroup of $O(n,1)$.
\end{theorem}

Here is a sketch of the proof of Theorem \ref{maximalcuspsstandard}.  Suppose $\Gamma$ is the
holonomy of a maximal cusp.  Then $\Gamma$ preserves some properly convex
set $\Omega$ and fixes a point $p\in\bOmega.$ Following Fried \&
Goldman,  a {\em syndetic hull} of a discrete subgroup $\Gamma$ of
a Lie group $G$ is defined as a connected Lie subgroup $H$ containing $\Gamma$ with
$H/\Gamma$ compact.  This is used to show in \ref{maxcuspislattice} that there is
a subgroup, $\Gamma_0,$ of finite index in $\Gamma$ with a nilpotent simply
connected syndetic hull $W\subset SL(n+1,{\mathbb R}).$
By \ref{domainforliegroup} there is another domain $\Omega'$ which is
strictly convex and contains $p$ in its boundary 
and  $W$ acts simply transitively on 
$\bOmega'\setminus\{p\}.$ The characterization implies that $\bOmega'$ is an ellipsoid
and therefore $\Gamma_0$ is conjugate into $O(n,  1).$ 
An easy algebraic argument, given in (\ref{finiteindex}), implies $\Gamma$ is
conjugate into $O(n,1)$ completing the proof.
\begin{proof}[Proof of \ref{ellipsoid}] 

 Lemma \ref{unipotent} implies that $W$ is conjugate to a group of  upper-triangular
  unipotent matrices. In particular, every nontrivial element of $W$ is parabolic.
The proof is  by induction on $n=\dim W=\dim\bOmega.$ Using the parabolic model of
hyperbolic space,
the inductive hypothesis is that there are parabolic coordinates for $\Omega$
centered on $p$ such that 
 $\bOmega$ is the graph of the convex function
$f:U\longrightarrow{\mathbb R}$ given by $f({\bf u})=\frac{1}{2}||{\bf u}||^2$,
where $U$ designates ${\mathbb R}^{n}$ equipped with an inner product; 
and also that $W$ is the group with elements $S_{\bf u}$ corresponding
to ${\bf u}\in U$ given by
$$S_{\bf u}(x)=x+{\bf u}+<{\bf u},x>e_0 +\frac{1}{2} ||x||^2 e_0$$

In the case $n=1$ the Lie group $W$ is one-dimensional. The classification of parabolics
given in \ref{parabolic2and3}
implies that $W$ is conjugate to a parabolic subgroup of $O(2,1)$  and $\bOmega$ 
is the orbit of a point under this subgroup. The conclusion now follows for $ n = 1$.

Inductively assume the statement is true for $n$.  
Since $\Omega$ is strictly convex,
radial projection ${\mathcal D}_p$ identifies $\bOmega\setminus p$ with 
${\mathcal D}_p\Omega$ by \ref{spacedirections}(3). 
The hypothesis that $W$ acts simply transitively on $\bOmega\setminus p$ implies
 ${\mathcal D}_p\Omega/W$ is a single point and thus compact.
Then  
 \ref{maxparabolicround} implies that 
$p$ is a round  point of $\bOmega.$

Consider a domain $\Omega$ with $\dim\bOmega=n+1,$
so $\Omega\subset{\mathbb R}P^{n+2}.$
There is a basis $e_0,\cdots,e_{n+2}$ of ${\mathbb R}^{n+3}$  in which $W$ is upper-triangular.
 In these coordinates $p=[e_0]$ and the projective hyperplane $P$,  given by
the subspace spanned by $e_0,\cdots,e_{n+1}$, is the supporting
hyperplane to $\Omega$ at $p.$ We can choose $e_{n+2}$ so that
it represents any point  $q\in\bOmega\setminus\{p\}.$ The affine patch
${\mathbb R}^{n+2}$ given dehomogenising by $x_{n+2}=1$ gives parabolic coordinates
for $\Omega$ with $P$ at infinity and $q$ at the origin. Furthermore, the
hyperplane, $U\subset{\mathbb R}^{n+2}$ given by $x_{0}=0$  is tangent to $\Omega$
at $q$ and $\bOmega$  is the graph of a
non-negative convex function $f:U\longrightarrow {\mathbb R}\cdot e_0$ defined on {\em all}
of $U$ because $p$ is a $C^1$ point; as in \S\ref{horospheresection}.  We refer
to $U$ as {\em horizontal} and the $x_0$-axis as {\em vertical}. 

Since $p$ is round, $P$ is unique, so that the group $W$ acts on ${\mathbb R}^{n+2}$ as a group of affine transformations.
It sends vertical lines to vertical lines and therefore induces an action on $U.$
It follows that this induced action on $U$ is simply transitive. Regarding an element of $W$
as a matrix in the chosen basis, by  \ref{directioneigenvalues}, the matrix for this induced 
action on $U$ is given by deleting the first row and column which correspond to $e_0$, the 
vector in the vertical direction.

 There is a codimension-1
foliation of ${\mathbb R}^{n+2}$ given by the vertical hyperplanes $P_c$ defined by $x_{n+1}=c.$  
This foliation is preserved by $W$. Indeed, $W$
is unipotent and upper-triangular, so the $(n+2,n+3)$-entry gives a 
homomorphism $\phi:W\longrightarrow{\mathbb R}$
and for $w\in W$ it follows that $w(P_c)=P_{c+\phi{w}}.$

Consider the horizontal subspace $V=U\cap P_0$ with basis $(e_1,\cdots,e_{n}).$
Let $W_V=\ker\phi$ and $\Omega_V=\Omega\cap P_0$ then $\bOmega_V$ is the graph of $f|V.$
Observe that $\Omega_V$ is a strictly convex set in ${\mathbb R}P^{n+1}$ and
$W_V$ preserves $\Omega_V$ and acts simply transitively on $\bOmega_V.$ 
By induction, there is an inner product on $V$ so that $\Omega_V$ is the graph of
$f(v)=\frac{1}{2}||v||^2$ for $v\in V$, and the group $W_V$ consists of elements $T_{\bf v}$ for
${\bf v}\in V$  given by
$$T_{\bf v}(x)=x+{\bf v}+<{\bf v},x>e_0 +\frac{1}{2} ||x||^2 e_0.$$

\noindent  In the basis $e_0$ followed by an orthonormal basis of $V$ followed by $e_{n+2}$,
 the matrix of $T_{\bf v}$  is
 $$\left( \begin{array}{ccccccc}
1  &  v_1 & v_2 & ..... & v_{n} &  \frac{1}{2}\sum_{i=1}^{n} v_i^2\\ 
0 &1 & 0 & 0 & 0 &  v_1 \\ 
0 &0 & 1 & 0 & 0 &  v_2 \\ 
.... &.... & .... & .... & .... & ....\\ 
0 &0 & 0 & 0 & 1 &  v_{n} \\ 
0 &0 & 0 & 0 & 0 &  1 \end{array}\right) $$ 
so that the Lie algebra, ${\mathfrak w}_V$ of $W_V$ is
$$\left( \begin{array}{ccccccc}
0  &  v_1 & v_2 & ..... & v_{n} &  0\\ 
0 &0 & 0 & 0 & 0 &  v_1 \\ 
0 &0 & 0 & 0 & 0 &  v_2 \\ 
.... &.... & .... & .... & .... & ....\\ 
0 &0 & 0 & 0 & 0 &  v_{n} \\ 
0 &0 & 0 & 0 & 0 &  0 \end{array}\right) $$ 

\noindent It follows that the general element
of the Lie algebra, ${\mathfrak w}$ is an $ (n+3) \times (n+3)$   matrix of the form

$$ \alpha =  \left( \begin{array}{ccccccc}
0  &  x_1 & x_2 & ..... & x_n &  t_{0}  & 0\\ 
0 &0 & 0 & 0 & 0 & t_1 & x_1 \\ 
0 &0 & 0 & 0 & 0 &  t_2  & x_2 \\ 
0 &0 & 0 & 0 & 0 &  t_3  & x_3 \\ 
.... &.... & .... & .... & .... & .... & ....\\ 
0 &0 & 0 & 0 & 0 &  t_n & x_n \\ 
0 &0 & 0 & 0 & 0 & 0 & x_{n+1} \\ 
0 &0 & 0 & 0 & 0 & 0 & 0 \end{array}\right) $$
 These Lie algebra elements satisfy $\alpha^4=0,$ so the general group element in $W$ is 
 $a= \exp(\alpha)= I + \alpha + \alpha^2/2+\alpha^3/6$. 
 Because the induced action of $W$ on $U$ is simply transitive it follows that
$x_1,\cdots ,  x_{n+1}$ are coordinates for ${\mathfrak w}$ and the remaining
entries in $\alpha$ are linear functions of these coordinates.

The orbit of the origin gives $\bOmega$ and is given by the last column of $a$, which is
the transpose of 
$${\bf y}=(f(x_1,......,x_{n+1}),x_1,x_2,\cdots,x_{n+1},0)+x_{n+1}(0,t_1,\cdots t_n,0,0),$$
where the first entry of ${\bf y}$ is the function $f:{\mathbb R}^{n+1}\rightarrow{\mathbb R}$
so that $\bOmega$ is the graph of $f(x_1,......,x_{n+1})$. Notice that these computations
show that this function  is a polynomial of degree at most $3$ in the coordinates $x_1,\cdots,x_{n+1}$.
Moreover, since $f({\bf x})>0$ for all non-zero ${\bf x}$
the linear and cubic parts are both zero, and  it follows that $f$ is a positive definite quadratic form.

 Choose an inner product on ${\mathbb R}^{n+2}$
so that $f({\bf x}) = ||{\bf x}||^2/2$. It now follows that $\bOmega$ is projectively equivalent to the round ball and $W$
is conjugate into a parabolic subgroup of $O(n+1,1).$ Since $W$ is unipotent, this is the parabolic
subgroup of pure translations, which completes the inductive step.
\end{proof}

\begin{lemma}\label{unipotent} Suppose that $\Omega$ is strictly convex and 
$W \subset SL(\Omega,p)$  is nilpotent and  acts  simply-transitively on 
$\bOmega\setminus\{p\}.$ 

 Then $W$ is unipotent and conjugate in $\SL(n+1,{\mathbb R})$ into the group of upper triangular matrices. 
\end{lemma}

\begin{proof} As above,  every non-trivial element of $W$ is parabolic and  $p$ is a round point of $\bOmega.$ 
The idea of the proof is to show 
that if $W$ is not unipotent, then there  is a proper projective subspace, $Q,$ that is preserved by $W$, which  
contains $p$ and another point in  $\bOmega$.  Since $Q$ is a proper subspace $Q\cap\bOmega$ is a 
proper non-empty subset which is preserved by $W$  which contradicts the transitivity assumption. 

Recall some standard facts about nilpotent Lie algebras and their representations.
Let $\rho : \wp \longrightarrow End(V)$ be a representation of a nilpotent Lie algebra in a finite dimensional
vector space $V$. A linear function $\lambda : \wp \longrightarrow {\mathbb C}$ is a
{\em weight} of  $\wp$, if there is some nonzero vector  ${\bf v} \in \wp$ and an integer $m = m({\bf v})$
so that  $(\rho(X) - \lambda(X) I)^m{\bf v} = {\bf 0}$ for all $X \in \wp$. The set of such vectors together with ${\bf 0}$
forms a linear subspace of $V$, this is the {\em weight space} of $\rho$ corresponding to the weight $\lambda$
and is denoted $V_{\rho, \lambda}$.

Then in \cite{raja}, Theorem 3.5.8 it is shown that if $\wp$ is a nilpotent Lie algebra and 
$\rho : \wp \longrightarrow End(V)$ is a representation in a finite dimensional vector space $V$
over an algebraically closed field, then the weight spaces corresponding to distinct weight are linear 
independent  and there is a decomposition
\begin{equation}\tag{$*$}
{\mathbb C}^n = \bigoplus_{\lambda} V_{\rho, \lambda}
\end{equation}
exhibiting the algebra $\wp$ as block matrices.

We apply these ideas to the Lie algebra $\mathfrak w$ of $W$; differentiating the inclusion $W \longrightarrow GL(n,{\mathbb R})$
yields a representation of  $\mathfrak w \longrightarrow End({\mathbb R}^n)$. Moreover, $W$ is simply connected so that
the exponential map $exp : \mathfrak w \longrightarrow W$ is an analytic diffeomorphism (see \cite{raja} Theorem 3.6.2)
and the decomposition of $(*)$ gives rise to a block decomposition of  ${\mathbb C}^n$ as a 
direct sum of $W$-invariant subspaces; we suppress $\rho$ and write $ V_{\rho, \lambda} = X_\lambda$. Each weight
space gives rise to a homomorphism $\mu : W \longrightarrow {\mathbb C}^*$, since  if $g \in W$ is written $g = exp({\bf w})$,
we may define $\mu(g) = exp(\lambda({\bf w}))$, i.e. we associate to $g$, the eigenvalue which appears in the
block $X_\lambda$. In this way $ X_\lambda$ is defined 
as the intersection over all $g$ in $W$ of the kernel of $(g - \mu(g)I)^n.$ The action of $W$ on 
$X_{\lambda}$ is given by $\mu(g)\cdot U(g)$ where $U(g)$ is unipotent.

Now recall that $ W \subset GL(n,{\mathbb R}).$ For each weight $\mu$ there is a complex conjugate 
weight $\overline\mu.$ This yields a direct sum decomposition over ${\mathbb R}$ 
\[{\mathbb R}^n = \bigoplus_{\{\mu,\overline{\mu}\}} V_{\mu,\overline{\mu}},\]
where $V_{\mu,\overline{\mu}}=(X_{\mu}+ X_{\overline\mu})\cap{\mathbb R}^n$. 
This sum is direct if $\mu\ne\overline{\mu}$.

This follows from the following elementary fact.  Suppose $U$ is a complex
vector subspace of ${\mathbb C}^n$ which is invariant under the involution
$v\mapsto\overline{v}$ given by coordinate-wise complex conjugation, so
that $U=\overline{U}.$ Then $U=(U\cap{\mathbb R}^n)\otimes_{\mathbb
R}{\mathbb C}.$ Observe that $X_{\overline\mu}=\overline{X_{\mu}}$ and
apply this with $U=X_{\mu}\oplus X_{\overline\mu}.$

Because every non-trivial element of $W$ is parabolic, it has $1$ as an
eigenvalue with algebraic multiplicity at least $3.$ Suppose some element
$A$ of $W$ has an eigenvalue other than $1$.  Every eigenvalue of every
element of $W$ has complex modulus 1.  Since $A$ is in a $1$-parameter
subgroup there is some element, $B,$ of $W$ which has a non-real
eigenvalue.  By combining the $V_{\mu,\overline{\mu}}$ subspaces into two
sets, one with $\mu(B)=\pm 1$ and the other with $\mu(B)\ne \pm1$ we get a
$G$-invariant decomposition \[{\mathbb R}^n=U\oplus V\] with $V$ generated
by the set with $\mu(B)\ne\pm1.$ If $\mu(B)$ is complex then $X_{\mu}$ and
$X_{\overline{\mu}}$ are both non-trivial, so that $\dim(V)\ge 2$.  On the
other hand since $B$ has eigenvalue $1$ with algebraic multiplicity at
least $3$ it follows that $\text{codim}(V)\ge 3$.  Furthermore, we observe
that $e_1\in U$.

Let $V'$ be the subspace spanned by $V$ and $e_1.$ Then
$\text{codim}(V')\ge 2$ thus $V'$ is a {\em proper} subspace.  The
projective subspaces obtained from $U$ and $V'$ intersect in one point,
namely $p=[e_1]\in\bOmega.$ Since $p$ is a smooth point of $\partial
\Omega$, there is a unique supporting tangent hyperplane, $P$ say, to
$\Omega$ at $p$.  If both $P(U)$ and $P(V')$ are contained in $P$ then $P$
 contains the projectivization of $U+V'={\mathbb R}^n$
contradicting that $P$ has codimension $1.$
  
It follows that at least one of $U$ and $V'$ contains a point in the
interior of $\Omega$.  However, both subspaces are proper and we thus
obtain a proper non-empty $G$ invariant subset of
$\bOmega\setminus\{p\}.$ This contradicts the transitivity
assumption.  This is a contradiction, which proves that $W$ is
unipotent.
\end{proof}

This completes the proof of the characterization of ellipsoids. It remains to apply this to show
maximal cusps are hyperbolic, following the outline:

\begin{proposition}[Discrete nilpotent virtually has simply connected syndetic hull]
\label{maxcuspislattice} 
Suppose that $\Gamma$ is a finitely generated,  discrete nilpotent subgroup of $GL(n,{\mathbb R})$. 

Then $\Gamma$ contains a subgroup of finite index $\Gamma_0$, which has a syndetic hull $W \leq  GL(n,{\mathbb R})$ 
that is nilpotent, simply-connected and a subgroup of the Zariski closure of $\Gamma_0.$
\end{proposition}
\begin{proof}  Since $\Gamma$ is finitely generated  and linear, by Mal'cev-Selberg's lemma it has a torsion-free subgroup, 
$\Gamma_1,$ of finite index.
By a theorem of Mal'cev (\cite{LiegrpsII} p45, thm 2.6) there is a simply connected nilpotent Lie group $\widetilde{W}$
which contains $\Gamma_1$ as a cocompact lattice. 
By the super-rigidity theorem for lattices in nilpotent groups
(the nilpotent case we need is
due to \cite{GOR}, see also \cite{Witte} Theorem 6.8$^\prime$ as well as the paragraph above (1.3) and (1.4) therein) 
after possibly passing to a finite index subgroup $\Gamma_0\subset  \Gamma_1,$
 the inclusion map  $i :  \Gamma_0 \rightarrow GL(n,{\mathbb R})$ extends to a homomorphism
$\pi:\widetilde{W}\rightarrow GL(n,{\mathbb R}).$
 Furthermore, $W=\pi \widetilde{W}$ is
contained in the Zariski closure of $\Gamma_0.$ 

The map $\pi:\widetilde{W}\rightarrow W$ is the universal cover and since these are both
nilpotent groups the group of covering transformations is a discrete free abelian group. However, $\pi$ restricted to 
$\Gamma_0$ is an inclusion
map, i.e. $\Gamma_0 \cap ker(\pi) = \{ 1 \}$. But $\pi^{-1}(i \Gamma_0)$
is a lattice in $\widetilde{W}$ which contains $\Gamma_0$;  this is is impossible
unless $ ker(\pi)$ is trivial, so that $\pi$ is injective. Thus we may identify $\widetilde{W}$ with $W.$

Now $W/\Gamma_0$ 
is a compact 
subset of $GL(n,{\mathbb R})/\Gamma_0$ and thus closed.
Hence $W$ is a closed subgroup
and thus a Lie group. 
\end{proof}\noindent
{\bf Remarks.} (i) In general $\pi$ is not birational and $W$ need not be
an algebraic subgroup.  For example, let $\Gamma$ be the cyclic subgroup of
$GL(2,{\mathbb R})$ generated by the diagonal matrix $\diag(2,3).$ The
Zariski closure of $\Gamma$ is the diagonal subgroup of rank 2, but $W$ is
a one-parameter subgroup.\\
(ii) If $\Gamma\subset SL(n,{\mathbb R})$ then the Zariski closure of $\Gamma$ 
(and hence $W$) is in $SL(n,{\mathbb R}).$%

\gap
By the above the hypothesis of the next result holds for a finite index subgroup of  a cusp group of maximal rank.
\begin{proposition}\label{domainforliegroup} Suppose that  $\Omega$ is properly convex and 
  $\Gamma\subset SL(\Omega,H,p)$ is a torsion-free  cusp group of maximal rank.
  Also suppose that $\Gamma$ is a cocompact lattice in a simply connected nilpotent Lie subgroup 
$W$ of $SL(n+1,{\mathbb R}).$ 
Further assume that $W$ is contained in the Zariski closure of $\Gamma.$  

Then there is a strictly convex domain $\Omega'$ with $p\in\bOmega'$ and which is preserved by $W$ and
$W$ acts simply transitively  on $\bOmega'\setminus\{p\}.$ 

In particular, the non-trivial elements of $W$ are all parabolic and $p$ is a round point of $\bOmega'$.
\end{proposition}
\begin{proof}  
The condition that $\Gamma$ preserves $p$ and $H$ is algebraic, therefore the Zariski
closure of $\Gamma$, and hence $W$, also preserves them.  
There is a natural action of $W$ on ${\mathcal D}_p{\mathbb R}P^n\cong {\mathbb R}P^{n-1}$  by
projective transformations. This action preserves the image of $H$ and so gives an affine action on 
${\mathbb A}^{n-1}.$ Radial projection ${\mathcal D}_p$ 
identifies ${\mathbb A}^{n-1}$ with an $(H,p)$-horosphere because $p$ is a round point
 by \ref{maxparabolicround}. Hence the action of $\Gamma$ on ${\mathbb A}^{n-1}$ is
 properly discontinuous. Thus ${\mathbb A}^{n-1}/\Gamma$ is a Hausdorff manifold

The action of $W$  on  ${\mathbb A}^{n-1}$ is transitive because
   $W$ is a simply connected nilpotent Lie group, so it  is contractible and 
  $\Gamma$ is a lattice, so that $W/\Gamma$ is a compact  manifold
 which is homotopy equivalent to the compact manifold
 ${\mathbb A}^{n-1}/\Gamma$.   Both manifolds are Hausdorff.
  Furthermore, there is a $\Gamma$-equivariant
 map $\tilde{\theta}:W\rightarrow{\mathbb A}^{n-1}$ given by  sending $w\in W$ 
 to $w\cdot x_0.$ This map covers a homotopy equivalence 
 $\theta:W/\Gamma\rightarrow{\mathbb A}^{n-1}/\Gamma$ between compact manifolds. 
 Therefore $\theta$ is surjective. It follows that the $W$-orbit of $x$ is all of
 ${\mathbb A}^{n-1}.$

 The map $\tilde{\theta}$ is injective because $\tilde{\theta}$ is a local diffeomorphism 
 at some point since it is a smooth surjection between manifolds of the same dimension. 
 By transitivity it is a local diffeomorphism everywhere. Thus $\theta$ also has this property 
 and is therefore a covering map. Thus
 $\tilde{\theta}$ is also a covering map. But $W$ and ${\mathbb A}^{n-1}$ are simply
 connected so the covering is trivial. Thus $\tilde{\theta}$ is injective as claimed. It
 follows that $W$ acts freely on ${\mathbb A}^{n-1}.$ 
 
 Choose a point $x\in{\mathbb A}^n.$ Define $\Omega'$ as the interior of the convex
hull of $W\cdot x$. We claim this is a properly convex domain.
Since $p$ is a round point, ${\mathbb A}^{n-1}$ is foliated by generalized horospheres ${\mathcal S}_t$
and the generalized horoballs  ${\mathcal B}_t$ fill ${\mathbb A}^{n-1}$.
Since $\Gamma$ is a parabolic group it preserves every horosphere and horoball.
 There is a compact subset $D\subset W$
such that $W =\Gamma\cdot D.$ Then $D\cdot x$ is a compact set in ${\mathbb
A}^n.$ Thus it is contained in some horoball ${\mathcal B}_t.$ Thus $W
\cdot x =\Gamma\cdot (D\cdot x)$ is also contained in ${\mathcal B}_t$.
It follows that the convex hull of this set is contained in ${\mathcal B}_t$ and is therefore properly convex.

  Clearly $p
\in\bOmega'$ and since $\Omega'$ is contained in a generalized horoball
 $P$ is a supporting tangent hyperplane to $\Omega'$
at $p.$ Also $\Omega'$ is $W$-invariant.
It remains to prove that $\Omega'$ is strictly convex. 

We may regard $\Omega'$ as the interior of a compact convex 
set $K$ in Euclidean space. 
As noted earlier, $K$ is the convex hull of its extreme points.
Therefore there is an extreme point $q \in \bOmega'$  other than $p.$ The action of $W$ on $\bOmega'\setminus\{p\}$ is 
transitive, since this set is identified with  ${\mathcal D}_p\Omega.$ 
The orbit of $q$ under $W$ consists of extreme points, hence 
with the possible exception of $p$, every point of $\bOmega'$ is an extreme point. However it follows immediately from the 
definition that if every point but one of  $\bOmega'$ is extreme, then every point of  $\bOmega'$ is extreme. 
This proves that $\Omega'$ is strictly convex.

Finally, if $1\ne w\in W$  then $w$ fixes $p$.  Since $W$ acts freely on ${\mathbb A}^n = 
{\mathcal D}_p\Omega= {\mathcal D}_p\Omega'$,  it acts freely on $\bOmega'\setminus\{p\}$
and so fixes no point other than $p$ in $\partial \Omega'$. Thus $w$ is not hyperbolic.
If $w$ were elliptic, it would fix a point in $\Omega$ and hence fix
 every point on the line $\ell$ containing $p$ and $q.$ But
this line meets $\bOmega'$ in a second point, giving the same contradiction. 
Thus $w$ is parabolic. \end{proof}
\begin{lemma}\label{finiteindex}
 Suppose that $\Gamma\subset GL(n+1,{\mathbb R})$ contains a 
parabolic subgroup of finite index $\Gamma_0\subset O(n,1)$ which preserves the ball
$\Omega$ and fixes the point $p\in\bOmega.$  
Also suppose that $p$ is a bounded parabolic fixed point for $\Gamma$.
Then $\Gamma\subset O(n,1).$  
\end{lemma}
\begin{proof} By passing to a subgroup of finite index
we may assume that $\Gamma_0$ is a normal subgroup of $\Gamma.$ 
Let $P$ be the supporting hyperplane to $\Omega$ at $p.$ Then $P$
is the unique codimension-1 hyperplane preserved by $\Gamma_0.$ 
Since $\Gamma_0$ is normal in $\Gamma$ it follows that $P$ is also
preserved by $\Gamma.$ If $x\in\bOmega\setminus\{p\}$ then the compactness of
$(\bOmega\setminus\{p\})/\Gamma_0$ implies the orbit $\Gamma_0\cdot x$ is
Zariski dense in $\bOmega.$

Since $\Gamma$ preserves $P$ it follows that $\gamma x\notin P$ for all $\gamma\in\Gamma.$
Since $\Gamma_0$ preserves $\bOmega$ the $\Gamma_0$ orbit of any point $x\notin P$
contains a generalized horosphere, ${\mathcal S}_x,$ for $\Omega$ centered at $p.$ Using normality gives
$$\Gamma_0\cdot (\gamma x)=\gamma(\Gamma_0\cdot x).$$

The Zariski closure of $\Gamma_0\cdot (\gamma x)$ is ${\mathcal S}_{\gamma x}$ and the Zariski closure
of $\gamma(\Gamma_0\cdot x)$ is $\gamma {\mathcal S}_x.$ It follows that $\gamma$ preserves the family
of generalized horospheres centered at $p.$ For some $n>0$ we have $\gamma^n\in\Gamma_0.$
We claim that it follows that $\gamma$ preserves each ${\mathcal S}_x.$ For otherwise, after replacing $\gamma$ by 
$\gamma^{-1}$ if needed we may assume $\gamma({\mathcal S}_x)$ is contained the interior of the
horoball bounded by ${\mathcal S}_x.$ But then the same is true for $\gamma^n{\mathcal S}_x.$ In particular $\gamma_n$
does not preserve ${\mathcal S}_x.$ This contradicts that $\gamma^n\in\Gamma_0.$

Thus every element of $\Gamma$ preserves the ball $\Omega$ and it follows from classical results
of Beltrami \& Klein (see for example Theorem 6.1.2 of Ratcliffe \cite{ratty})  that  
$\Gamma\subset O(n,1)$.    \end{proof}

It follows from \ref{thincusps} that:
\begin{proposition}[Maximal cusps have finite volume]
\label{improvedmarguliscusps} If $C$ is a maximal cusp in a properly convex projective manifold 
then $C$ has finite volume
  \end{proposition}

An irreducible representation into $GL(n+1,{\mathbb R})$ is determined up
to conjugacy by its character.  It follows that non-elementary hyperbolic
manifolds are isometric iff they are projectively equivalent.  A {\em
hyperbolic cusp} is a cusp of a hyperbolic manifold.  The preceding
argument fails for cusps since the character is the constant function with value
$(n+1)$ for every cusp with cross-section a codimension one torus.  The
next result says that maximal hyperbolic cusps are equivalent in the
projective sense iff they are equivalent in the hyperbolic sense.

\begin{proposition}[Hyperbolic cusps]
\label{hyperboliccusps} 

Suppose $\Gamma_1,\Gamma_2\subset PO(n,1)$ are two groups of parabolic isometries
so that the quotients $C_i={\mathbb H}^n/\Gamma_i$ are maximal cusps.  

Then $\Gamma_1$ and $\Gamma_2$ are conjugate subgroups of $PO(n,1)$ iff they are
conjugate subgroups of $PGL(n+1,{\mathbb R}).$ 
Thus $C_1$ and $C_2$ are isometric iff they are projectively equivalent.  
\end{proposition}
\begin{proof} The symmetric bilinear form $\langle,\rangle$ of signature $(n,1)$ is
preserved by $O(n,1).$ Let $S$ be the projectivization of the set of
non-zero lightlike vectors for this form.  Then $S$ is the boundary of the
projective model of ${\mathbb H}^n.$ By means of conjugacy within $O(n,1)$
we may assume the groups $\Gamma_1,\Gamma_2$ have the same parabolic
fixed-point $p=[a]\in S.$  Since $C_1$ and $C_2$ are projectively equivalent,  
$\Gamma_2= \gamma.\Gamma_1.\gamma^{-1}$  for  an element $\gamma \in GL(n+1,{\mathbb R})$. 

The function $f:{\mathbb R}P^n\setminus S\longrightarrow{\mathbb R}$ given by 
$f(x)=\langle a,x\rangle^2/\langle x,x\rangle$ has level sets in ${\mathbb H}^n$ that are the horospheres centered at $p.$ 
Thus a horosphere is a quadric.

Choose some point $x$ in ${\mathbb H}^n$ and consider the orbit $\Gamma_1\cdot x$.
Since $C_1$ is a maximal cusp the Zariski closure of this orbit is the horosphere ${\mathcal S}_1$ 
centred at $p$ that contains $x$ and  is the quadric hypersurface $\{ y: f(y)=f(x)\}.$ 

We may assume $\gamma(x)$ is in ${\mathbb H}^n$ and therefore one may define 
${\mathcal S}_2$ to be the unique horosphere centred at $p$ which contains the point $\gamma(x)$.
Since $\Gamma_2$ acts by hyperbolic isometries, ${\mathcal S}_2$ contains the orbit 
 $\Gamma_2\cdot\gamma(x)$. Note that ${\mathcal S}_2$ is the unique quadric which 
 contains the orbit $\Gamma_2\cdot\gamma(x)$.

Now projective transformations send quadrics to quadrics, so that $\gamma {\mathcal S}_1$
is the unique quadric which contains $\gamma(\Gamma_1\cdot x)$. Since 
$\Gamma_2\cdot\gamma(x) = \gamma(\Gamma_1\cdot x)$, it follows that 
$\gamma {\mathcal S}_1 = {\mathcal S}_2$.

 Let ${\mathcal B}_i$ be the open horoball ball bounded by ${\mathcal S}_i.$ 
 The Hilbert metric on ${\mathcal B}_i$ is isometric to ${\mathbb H}^n.$
Furthermore ${\mathcal B}_i/\Gamma_i$ is isometric to ${\mathbb H}^n/\Gamma_i.$ Also
$\gamma$ is an isometry of ${\mathcal B}_1$ onto ${\mathcal B}_2.$ Hence, using $\cong$ to denote
isometry of Hilbert metrics, we get 
\begin{equation*}
{\mathbb H}^n/\Gamma_1\cong
{\mathcal B}_1/\Gamma_1\cong {\mathcal B}_2/\Gamma_2\cong{\mathbb H}^n/\Gamma_2.
\qedhere
\end{equation*}
\end{proof}

\section{Topological Finiteness}\label{top_finiteness}

This section contains finiteness properties about families of properly or strictly convex
manifolds, including a finite bound on the number of homeomorphism classes under various hypotheses.

There is a fundamental difference between the strictly convex and properly
convex cases. In the  strictly convex case the thick part is non-empty
and all that is required is an upper bound on
volume. 
However in the properly convex case the entire manifold might be thin
and one needs an upper bound on diameter and a lower
bound on the  injectivity radius at one point.

In dimension greater than 3 there are finitely many isometry
classes of complete, hyperbolic manifolds with volume less than $V$. If a closed hyperbolic manifold contains
a totally geodesic codimension-1 embedded submanifold then the hyperbolic
structure can be deformed to give a one parameter family of strictly convex structures.
Therefore there is no bound on the number of isometry (= projective equivalence) classes of strictly convex manifolds
with bounded volume.
Marquis has similar examples for hyperbolic manifolds with cusps \cite{marquis1}. 

An important tool that is of independent interest is that for properly convex manifolds 
 there is a uniform upper bound on how quickly 
injectivity radius at a point decreases as the point moves  \ref{decay}. 
This result, which is well known for Riemannian manifolds  with bounded curvature, was exploited by Cheeger for his finiteness theorem \cite{Ch}.

In dimension at least $4$, for closed strictly convex manifolds,
the diameter is bounded above by an explicit constant times the volume. 

\begin{proposition}[decay of injectivity radius]
\label{decay} For each dimension $n\ge 2$ there is a nowhere zero function $f:{\mathbb
R}^+\times{\mathbb R}^+\longrightarrow {\mathbb R}^+$ which is decreasing
in the second variable with the following property:

If $M$ is a properly convex projective $n$-manifold and $p,q$ are two points in $M$ 
then 
$${\rm inj}(q)>f({\rm inj}(p),d_M(p,q)).$$
\end{proposition} 
\begin{proof} Here is a sketch of a standard argument.
There is an upper bound, $V,$ on the volume of the ball of radius $R$ centered 
at a point where the injectivity radius is $\epsilon.$ There is a lower bound on the 
volume, $v,$ of an embedded ball of radius $\delta.$ If $v>V$ then a point where the 
injectivity radius is less than $\epsilon$ can't be within distance $R-\delta$ of a point with
injectivity radius $\delta.$ Thus for $R$ and $\delta$ fixed $\epsilon$
cannot be too small.  The details now follow:

The manifold is $M=\Omega/\Gamma.$ Suppose that the injectivity radius at $q$ is
$\epsilon/2.$ Then there is $\gamma$ in $\Gamma$ and $\tilde{q}\in\Omega$
covering $q$ such that $\gamma$ moves $\tilde{q}$ a distance $\epsilon.$ By
\ref{isometryfoliation} there is a hyperplane $H\subset\Omega$ that
contains $q$ and which is disjoint from $\gamma H.$ The latter contains
$\gamma\tilde{q}.$

\begin{figure}[ht]
 \begin{center}
	 \psfrag{Om}{$\Omega$}
	 \psfrag{q}{$\tilde{q}$}
	 \psfrag{X}{$X$}
	 \psfrag{e}{$\epsilon$}
	 \psfrag{B}{$B$}
	 \psfrag{R}{$R$}
	 \psfrag{H}{$H$}
	 \psfrag{gH}{$\gamma H$}
	 \psfrag{gq}{$\gamma\tilde{q}$}
\includegraphics[scale=1]{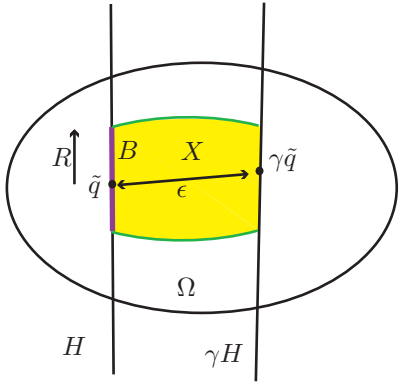}
	 \label{decayfig}
 \end{center}
\caption{Decay of Injectivity Radius}
\end{figure}

Let $X$ be the subset of $\Omega$ between $H$ and $\gamma H$ consisting of
all points distance at most $R$ from either $\tilde{q} $ or
$\gamma\tilde{q}.$ The image of $X$ in $M$ is the ball of radius $R$ around
$q.$ We claim that the Hilbert volume of $X$ is bounded above by a function
$V(\epsilon,R)$ which is independent of $\Omega,H$ and $\gamma.$ Clearly
this function is decreasing in $\epsilon$ and increasing in $R.$ We claim
that for each $R$ we have $\lim_{\epsilon\to0} V(\epsilon,R)=0.$

Assuming this, the proposition follows
from \ref{hilbertballs} since if ${\rm inj}(p)>\delta,$ then
the volume of the ball of radius $\delta$ center $p$ is bounded below by a
function $v(\delta)$ depending only on $\delta.$ If the distance in $M$
from $p$ to $q$ is $R-\delta$ then this ball is contained in $X$ so
$V(\epsilon,R)>v(\delta).$ The claim implies that as $\epsilon\to0$ then
$R\to\infty$, proving the proposition.

The proof of the claim follows from Benzecri's compactness theorem.  If the
claim is false there is $R>0$ and $V_0>0$ and for each $n>0$ there is a
domain $\Omega_n$ containing a point $\tilde{q}_n$ and a pair of
hyperplanes in $\Omega_n,$ as described, with $\epsilon=1/n$ and with the
volume of $X$ at least $V_0.$ We put $(\Omega_n,\tilde{q}_n)$ in Benzecri
position and pass to a convergent subsequence.  In the limit the two planes
coincide.  Just before that the Euclidean volume of $X$ is arbitrarily
small which contradicts Lemma \ref{hilbertballs}.   \end{proof}
\noindent
{\bf Remark.} With a bit more work the function $f$ in this result can be made explicit.\\

\begin{proof}[{\bf Proof of Proposition \ref{deeptubes} (Uniformly deep tubes)}] 
Suppose $p$ is a point on the boundary of a Margulis tube in a projective $n$-manifold
$M.$ Then the injectivity radius at $p$ is at least $\iota_n.$ Suppose
the core of the Margulis tube is a geodesic $\gamma$ of length $\epsilon.$

Then the injectivity radius at points on $\gamma$ is $\epsilon/2.$ By
\ref{decay} it follows that the distance of $p$ from $\gamma$ increases to
infinity as $\epsilon\to 0.$\end{proof}


\begin{proof}[Proof of \ref{GHcompact}] 

Let $\tilde{{\mathcal H}}$ denote the set of isometry classes 
of pointed metric spaces $(\Omega,x)$ with $\Omega$ an open properly convex set 
in $\RPn$ and equipped with the Hilbert metric. These metric spaces are obviously proper.
There is an isometry taking $\Omega$ into Benzecri position and $x$ to the origin.
The set of Benzecri domains is compact in the Hausdorff topology and this implies
these metric spaces are {\em uniformly totally bounded}: that is for every $\epsilon>0$ there is $N>0$
such that every metric space in the family is covered by $N$ balls of radius $\epsilon$.

The universal cover of a properly convex projective manifold is isometric to a properly convex domain
with its Hilbert metric. These domains are proper metric spaces which are uniformly
totally bounded. Hence the elements of ${\mathcal H}$ are uniformly totally bounded
 proper metric spaces.  Gromov's compactness theorem implies that ${\mathcal H}$ is precompact.
We will show that every sequence $(M_k,x_k)$ in ${\mathcal H}$ has a subsequence 
which converges to a point in ${\mathcal H}.$ It then follows from Gromov's compactness 
theorem that ${\mathcal H}$ is compact. 

We may isometrically identify the universal cover of
$M_k$ with a properly convex domain $\Omega_k$ in Benzecri position so that the 
origin $p\in\Omega_k$  covers $x_k.$ This provides an identification of $\pi_1(M_k,x_k)$ 
with a discrete subgroup  $\Gamma_k$ in $PGL(n+1,{\mathbb R}).$
The set of Benzecri domains is compact in the Hausdorff topology therefore there is a 
neighborhood $U$ of the identity in $PGL(n+1,{\mathbb R})$ such that every element in 
$U^{-1}U$ which preserves some Benzecri domain, $\Omega,$ moves $p$ a distance less than 
$\epsilon$ in $\Omega.$ Every non-trivial element of $\Gamma_k$ moves $p$  a distance 
at least $\epsilon,$ hence $\Gamma_k\cap U=\{1\}.$ 
It follows that for every $\delta\in PGL(n+1,{\mathbb R})$ that 
$|\Gamma_k\cap \delta U| \le 1$, for if $\alpha,\beta\in\Gamma_k\cap \delta U$ then 
$\alpha^{-1}\beta\in U^{-1}U.$ This implies $\alpha^{-1}\beta=1.$

Let $K_m$ be an increasing family of compact subsets with union $PGL(n+1,{\mathbb R}).$ 
Each $K_m$ is the union of a finite number, $c_m$ say, of left translates of $U.$ It follows 
that $K_m$ contains at most $c_m$ elements of $\Gamma_k.$ We may now subconverge 
so that the $\Omega_k$ converge in the Hausdorff topology to a Benzecri domain $\Omega_{\infty}$, 
and so that for each $m$  the sets $K_m\cap\Gamma_k$ converge to a finite 
set $S_m.$ Then $\Gamma_{\infty}=\bigcup_m S_m$ is a discrete group of projective transformation 
which preserves $\Omega_{\infty}.$ It is clear that $\Gamma_{\infty}$ is the limit in the Hausdorff topology 
on closed subsets of $PGL(n+1,{\mathbb R})$ of the sequence $\Gamma_n.$ We obtain a properly convex $n$-manifold
 $M_{\infty}=\Omega_{\infty}/\Gamma_{\infty}$ with basepoint $x_{\infty}$ 
which is the projection of $p.$   We show below that $(M_k,x_k)$ subconverges in the based Gromov-Hausdorff 
topology to $(M_{\infty},x_{\infty}).$ It follows that ${\mathcal H}$ is compact.

Since ${\Omega}_k$ converges in the Hausdorff topology  to ${\Omega}_{\infty},$ 
given a compact subset $K\subset\Omega_{\infty}$ it follows that  for all $k$ sufficiently 
large $K\subset\Omega_k.$ The restriction to $K$ of the Hilbert metric on $\Omega_k$ 
 converges as $k\to\infty$ to the restriction to $K$ of the Hilbert metric on $\Omega_{\infty}.$ 
Let $\pi_k:\Omega_k\longrightarrow M_k$ and $\pi_{\infty}:\Omega_{\infty}\longrightarrow M_{\infty}$ 
be the natural projections. Let  $R_k\subset \pi_k(K)\times \pi_{\infty}(K)$ be the relation 
induced by the identity on $K.$ Thus $\pi_k(x)R_k\pi_{\infty}(x)$ for all $x\in K.$ 
Since $\Gamma_k$ converges in the Hausdorff topology to $\Gamma_{\infty}$ 
it follows for each $y\in int(K)$ the partial orbits $K\cap(\Gamma_k\cdot y)$ converges 
to $K\cap(\Gamma_{\infty}\cdot y).$  The Hilbert metrics restricted to $K$
almost coincide, thus for $\epsilon>0$ and all $k$ sufficiently large,
 $R_k$ is an $\epsilon$-relation. This gives Gromov-Hausdorff convergence.

\end{proof}

This gives another proof of the uniform decay of injectivity radius. 


\subsection{The closed case.}


Recall that if $K$ is a simplicial complex and $C \subset |K|$, then the 
{\em  simplicial neighborhood} of $C$ is the union of all 
simplices in $K$ which are a face of a simplex that contains some point of $C.$
The  {\em open simplicial neighborhood} $U$ is the interior of this set.

\begin{proof}[Proof of \ref{properly_convex_finiteness}]  We show that $M$ has a triangulation with at most $s=s(d,\epsilon)$ 
simplices and is therefore homeomorphic to one of a finite number of PL-manifolds.

By decay of injectivity radius, \ref{decay}, there is $\delta=\delta(\epsilon,d)>0$ such that if $M$ 
satisfies the hypotheses of the proposition,  then at every point in $M$ the injectivity
radius is larger than $2\delta.$

By \ref{hilbertballs} metric balls of radius $\delta$ in properly convex domains are uniformly bilipschitz
to Euclidean balls, so there is $r=r(\delta)>0$ with $r<<\delta$ such that
every ball of radius at most $r$ in a properly convex domain is contained in a projective simplex
of diameter less than $\delta/10.$

From \ref{GHcompact} the manifolds satisfying the hypotheses
are uniformly totally bounded. Since $M$ has diameter at most $d,$ it follows that there is $N=N(r,d)>0,$ such that $M$ is covered by $N$ balls of radius $r$ and hence by $N$
embedded projective simplices each of diameter less than $\delta/10.$ 

List these simplices and inductively assume there is an embedded simplicial complex $K_m$ in $M$ 
which contains subdivisions of the first $m$ simplices in the list, and that the number of 
simplices in $K_m$ is bounded above by a function $s(m).$  

For the inductive step, choose a point $x$ in $\sigma=\sigma_{m+1}$ and
ball neighborhood, $B(x,\delta).$ This is an embedded ball in $M$ and lifts
to an affine patch.  The simplices in $K_m$ have diameter at most
$\delta/10$ so this ball contains the simplicial neighborhood of $\sigma$ in
$K_m.$ Apply Lemma \ref{subdivide} below in this affine patch to subdivide $\sigma$ and
$K_m$ to produce a simplicial complex $K_{m+1}$ containing subdivisions of
$\sigma$ and $K_m$ and with at most $s(m+1)$ simplices.  Observe that
simplices outside the ball are not subdivided, therefore this process is
local and therefore can be done in $M$. It follows that $M$ can be triangulated with at most
$s(N)$ simplices. \end{proof}

\begin{lemma}\label{subdivide} Suppose that $K$ is a finite simplicial
complex in Euclidean space, consisting of affine simplices.  Suppose that
$\sigma$ is an affine simplex in Euclidean space.  Let $L$ be the
simplicial neighborhood of $\sigma$ in $K.$ 

Then there is simplicial
complex $P$ containing simplicial subdivisions of $K$ and of $\sigma$ such
that simplices in $K\setminus L$ are not subdivided and so that the number
of simplices in $P$ is bounded in terms of the number of simplices in $L$.
\qed
\end{lemma}

The {\em diameter}, $\diam(X)$ of a metric space $X$ is the supremum of the distance between points.
\begin{proposition}(Margulis tube geometry)\label{deeptube} 
Suppose $T$ is a Margulis tube with depth $r$ in a strictly convex projective $n$-manifold $M=\Omega/\Gamma$. 

If dimension $n\ge 4$ then $\diam(\partial T)\ge r$ and $\diam(T)\le 4\cdot \diam(\partial T)$
\end{proposition}
\begin{proof} There is a unique closed geodesic $\gamma$ in $T$ and the depth of $T$ 
is the minimum distance
of points on $\partial T$ from $\gamma.$ 
By abuse of notation $\gamma \in GL( n+1, {\bf R})$ is the generator of the fundamental group of $T$ 
with fixed points $a$ and $b$ in $\partial \Omega$; which correspond to a pair of 
eigenvectors in ${\bf R}^{n+1}$ where the eigenvalues are positive and  are of largest and smallest modulus.

Since $n \geq 4$, the matrix of $\gamma$ has (at least) one further invariant vector subspace $W^-$, either of dimension one or two,
so that  by adjoining the eigenvectors corresponding to $a$ and $b$, we obtain a $\gamma$-invariant subspace $W^+$ 
of dimension three or four and hence an invariant projective subspace $W={\mathbb P}(W^+)$ 
of dimension two or three which 
contains the axis of $\gamma$.  Choose any projective hyperplane $V$ of codimension one which contains $W$. 

Since balls are strictly convex, there is a nearest point retraction $\pi : \Omega \longrightarrow V\cap\Omega.$ Then 
$\pi^{-1}(\pi z)$ is the line through $z$ consisting of the set of points in $\Omega$ with the property that
their closest point to $V$ is $\pi z$.  This map is distance non-increasing and surjective. 

Observe that  $\partial T$ separates $axis(\gamma)$ from $\partial \Omega.$
Pick some point $z \in axis(\gamma)$ then $\pi^{-1}(z)$ is a line which meets $\partial T$ in two points.
Let $x$ be one of these points. There is  a point $y \in W\cap\partial T$. 

 Since $W$ is $\gamma$-invariant   $\gamma^ky\in W \leq V$, and it follows
that $d(x, \gamma^k(y)) \geq r$ for every $k$.  The distance in $T$ 
between images of $x$ and $y$ is  $\min_k d(x, \gamma^k(y))$.  This proves $\diam(\partial T)\ge r$.

The second inequality follows from the following observations. 
Since $\pi$ is distance nonincreasing
$\diam(\gamma)\le \diam(\partial T)$. 
Every point in $T$ lies on a {\em vertical} line segment $\ell$ with one endpoint
on $\gamma$ and the other on $\partial T$ such that $\pi(\ell)$ is a single point. By the triangle inequality
$\diam(\ell)\le \diam(\partial T)+r+\diam(\gamma)\le 3\diam(\partial T)$. Given
two points, $x,y$ in $T$ let $\ell_x,\ell_y$ be the vertical arcs containing them. 
Choose two shortest arcs $\alpha\subset\gamma$ 
and $\beta\subset\partial T$ each connecting $\ell_x$ and $\ell_y.$
Then $\delta=\ell_x\cdot\alpha\cdot\ell_y\cdot\beta$ is a loop containg $x$ and $y$ made of these four
arcs. The length of $\delta$ is at most $\diam(\partial T)+\diam(\gamma)+2(3\diam(\partial T))\le 8\diam(\partial T)$.
Thus $x$ and $y$ are connected by an arc in this loop of length at most half this number.
\end{proof}

\begin{theorem} [Volume bounds diameter]
\label{boundedvolume_gives_bounded_diamater}
For each dimension $n \geq 4$ there is a constant $c_n>0$ such that if $M^n$ is either 
(i) a closed strictly convex real projective manifold or (ii) a Margulis tube,
then $\diam(M)\le c_n\cdot Volume(M).$ 
Furthermore, in the closed case,  $\diam(M)\le 9\diam(\thick(M))$.
\end{theorem}

\begin{proof}  We begin with the proof in the closed case. 

Let $M=A\cup B$ be a thick-thin decomposition of $M$ as given by \ref{thickthin}, where $B=\thick(M)$.  
Set $r=\diam(M)/18.$  Then every point in $B$ has injectivity radius at least $\iota_n$ and $A$
 is a disjoint union of Margulis tubes. 
 
 A point in a Margulis tube $T$ of $M$ is within a distance at most $\diam(T)$ of a point
  in $B$. By \ref{deeptube} $\diam(T)\le 4\cdot \diam(\partial T)\le 4\cdot \diam(B)$. Since $B$ is connected
   any two points in $M$ are connected by a path of length at most $(4+1+4)\diam(B)$.
 Thus $\diam(M)\le 9\cdot \diam(B)$. 
 
Hence $\diam(B)\ge 2r$  and the injectivity radius at every point in $B$ is at 
least $\iota_n$ there are $r/\iota_n$ disjoint embedded balls each of radius $\iota_n$ centered at points in $B.$
It follows from
the Benzecri compactness theorem that  the volume of a ball of 
radius $R$ in an n-dimensional properly convex set is bounded below by $v=v(n,R).$
Set $v=v(n,\iota_n).$
The volume of $M$ is at least the sum of the volumes of
these balls and this is bounded below by $(r/\iota_n)\cdot v.$ Then $c_n=v^{-1} \iota_n$ satisfies the conclusion
of the theorem.

In the second case when $M=T$ is a Margulis tube, the balls we exhibit are
centered on points of $\partial T$ and therefore not fully contained in
$T.$ To remedy this, use a slightly smaller Margulis tube $T'\subset T.$ We
leave the details to the reader.  \end{proof}


Combining \ref{boundedvolume_gives_bounded_diamater} and \ref{properly_convex_finiteness} gives:
\begin{theorem}
\label{topfinite_closedcase}
For fixed $n \geq 4$ and $K$, there are only finitely many homeomorphism types of closed, strictly convex real projective
$n$-manifolds of volume $< K$.
\end{theorem}

\begin{corollary}
For fixed $n \geq 5$ and $K$, there are only finitely many diffeomorphism types of closed, strictly convex real projective
$n$-manifolds of volume $< K$.
\end{corollary}
\noindent
{\bf Proof.}   For $n \geq 5$, it is classical that a given closed topological $n$-manifold has only finitely many 
smooth structures. For example, by Kirby-Siebenmann there are only finitely many $PL$-manifolds in each homeomorphism
class and by Milnor-Kervaire-Hirsch-Cairns, each such structure gives rise to a finite  number of smooth structures. (see \cite{AMB} Chapter $7$). \qed


\subsection{Topological finiteness: The general case. }

Here is an outline of the proof of topological finiteness of manifold with volume at
most $V$  in  the general case of a strictly convex 
manifold with cusps.

Using \ref{convexthin}  we can replace the thin part by  finitely many disjoint convex submanifolds, 
namely horocusps and tubes which are equidistance neighborhoods of closed geodesics.  
The injectivity radius on the boundary of these convex manifolds
is bounded below in terms $V.$ 
This is because the injectivity radius on the boundary
of the thin part is at least $\iota_n$ and combined with the upper bound on volume
  this bounds above the diameter of the boundary of the thin part. In what follows
  we use these convex thin manifolds and refer to their complement as the {\em thick part}.

The  volume bound now provides an upper bound on the diameter of the thick part in all dimensions.
As in the closed case it follows that there is a simplicial complex $K$ with a number of 
simplices bounded by some function of the volume, so that $|K|$ is a submanifold which 
{\em contains} the thick part. Now we observe the following: 
Using only the fact that the thin part is  convex it follows from \ref{convexsubcomplex} that there is a 
subcomplex of the second derived subdivision $K''$ of $K$  which is homeomorphic to the compact 
manifold obtained by removing the interior of thin part. This gives finitely many topological types for 
the thick part in all dimensions.

In dimension at least $4$,  a volume bound gives an upper bound on the diameter of
Margulis tubes, and thus a lower bound on their injectivity radius. We can then modify 
the above argument so that $K$ contains the Margulis tubes as well, omitting only the cusps.
This establishes there are only finitely many topological types of finite volume
strictly convex manifold in dimensions other than $3.$

The reason that dimension $3$ is different is that the group of self homeomorphisms mod homotopy of
$S^1\times S^n$ is finite unless $n = 1$, see Gluck \cite{Gluck} for $n=2$ and Browder
 \cite{Brow} for $n\ge 5$. Thus, except in this dimension, there are only finitely many ways
to attach a Margulis tube to the thick part. Of course in dimension $3$ there are known to be infinitely many
closed hyperbolic $3$ manifolds with volume less than $3$ and these are strictly convex. 
This completes the outline. 

\smallskip

{\bf Remark.} Some caution is required when there are cusps in view of the
following: Suppose $M$ is a manifold with a boundary component $T.$ One
might have a non-trivial h-cobordism $N\subset M$ with $\partial N = T\cup
T'$ and with $T'$ homeomorphic to $T.$ Thus it is not enough to prove there
are only finitely many possibilities for $M\setminus N $ unless one also
knows there are only finitely many possibilities for $N$ and for the
attaching map.
\gap

We begin with some definitions.  Suppose $\sigma_1$ is a face
of a simplex $\sigma.$ The {\em complementary} face $\sigma_2$ to
$\sigma_1$ is the simplex spanned by the vertices of $\sigma$ not in
$\sigma_1.$ Thus $\sigma=\sigma_1*\sigma_2$ is the {\em join} of $\sigma_1$
and $\sigma_2.$ This gives a line-bundle structure on
$|\sigma|\setminus(|\sigma_1|\cup|\sigma_2|)$ which we refer to as the {\em
simplex line-bundle} for $(\sigma,\sigma_1).$
 
 A fiber is the interior of a straight line segment connecting
 $x_1\in\sigma_1$ to $x_2\in\sigma_2.$ We orient these lines so they point
 towards $\sigma_1.$ This structure is completely determined by the choice
 of $\sigma_1$ and $\sigma.$ Observe that if $\tau$ is a face of $\sigma$
 and which intersects $\sigma_1$ but is not contained in $\sigma_1$ then
 the simplex line bundle for $(\tau\cap\sigma,\tau\cap\sigma_1)$ is the
 restriction of the simplex line bundle for $(\sigma,\sigma_1).$

A subcomplex $L$ of a simplicial complex $K$ is called {\em full} if for every $k > 0$,  
$L$ contains every $k$-simplex
$\sigma$ in $L$ having the property that $\partial\sigma\subset L.$

\begin{lemma}\label{simplexbundle} Suppose that $L$ is a full subcomplex of a simplicial complex $K.$
Let $U$ be the open simplicial neighborhood of $L$ in $K.$ 

Then $U\setminus |L|$
 is a line bundle whose restriction to each simplex in $U$ is a simplex line bundle.
  This bundle is a product.
  \end{lemma}

\begin{proof} Suppose that $\sigma$ is a simplex in $K$ which intersects
$U\setminus |L|.$ Since $U$ is in the open simplicial neighborhood $\sigma$
contains a point of $L.$ The condition that $L$ is full subcomplex implies that
$\sigma_1=\sigma\cap L$ is a simplex.  Since by hypothesis $\sigma$ is not
a simplex of $L$,   it follows that $\sigma_1\ne\sigma.$ This determines a
simplex line bundle for $(\sigma,\sigma_1).$ As remarked above, these
bundles are compatible on intersections, therefore this gives a global line
bundle.  The lines are oriented pointing towards $L$ and so the bundle is a
product.  \end{proof}

 In what follows we interpret the interior of a 0-simplex to be itself.  A
{\em derived subdivision}, $K',$ of a simplicial complex $K$ is determined
by a choice, for each simplex $\sigma$ of $K,$ of a point $\hat{\sigma}$,
called the {\em barycenter}, in the interior of $\sigma.$ Suppose that $C$
is a subset of $|K|.$ A derived subdivision of $K$ is said to be {\em
adapted} to $C$ if it satisfies the condition: for every simplex $\sigma$
of $K$ if $C$ contains a point in the interior of $\sigma$ then the
barycenter $\hat{\sigma}$ is in the interior of $C.$ Such a subdivision
exists iff whenever the interior of a simplex of $K$ contains a point of
$C$ then it also contains a point in the interior of $C.$ If $K''$ is a
derived subdivision of $K'$ (as above) adapted to $C$ we say $K''$ is a {\em second
derived subdivision of $K$ adapted to $C.$}
 
A subset $C$ of the underlying space of a simplicial complex $K$ is called
{\em locally convex} if $C\cap\sigma$ is empty or convex for every simplex
$\sigma$ in $K.$ It is {\em strongly locally convex} if, in addition,
whenever $C\cap\sigma$ is not empty, then  $C\cap\sigma$ contains an open subset
of $\sigma.$ It follows that there is a derived subdivision, $K,$ adapted
to $C$ and, moreover, $C$ is strongly locally convex relative to $K'.$

Observe that if $C_1$ and $C_2$ are both strongly locally convex and no
simplex of $K$ contains points in both $C_1$ and $C_2$ then $C_1\cup C_2$
is strongly locally convex.

 \begin{lemma}\label{convexsubcomplex} 
 Suppose that $M$ is a compact $n$-manifold triangulated by   a simplicial complex $K$  and that
  $C$ is a compact, strongly locally convex submanifold of $M$ which is a neighborhood of $\partial M.$ 
 Let $K'$ and $K''$  be a derived and second-derived subdivision of $K$ adapted to $C.$
  Let $L$ be the subcomplex of $K''$ consisting of those simplices contained entirely in $C.$

  Then there is a homeomorphism of $M$ to itself taking $C$ to $|L|.$
  \end{lemma}
 \begin{proof} 
 \label{moving_C}    Let $\partial' C$ be the closure of $\partial C\setminus\partial M.$ We will show that the closure of 
 $C\setminus |L|$ is homeomorphic to a collar $I\times\partial' C$ in $C$ of $\partial' C.$ 
 Since $\partial' C$ is bicollared in $M$  this implies the result.
 
 Let $W$ be the subcomplex of $K'$ consisting of all simplices which are entirely contained in $C.$
 Since $C$ is locally convex, $W$ is a full subcomplex. Furthermore, $W$ is contained in the interior
  of $C$ because each vertex of $W$ is the barycenter of a simplex in $K$ and these barycenters 
  are in the interior of $C.$ A simplex of $W$ is the convex hull of its vertices and 
  therefore contained in the interior of $C.$ 
  
   Let $U$ be the open simplicial neighborhood of $W$ in $K'.$ Then  $U$ contains $C.$ 
  This is because if $x$ is a point in $C$ then there is a simplex $\sigma$ in $K$ whose interior 
  contains $x.$ Since $K'$ is adapted to $C$ it follows that the barycenter $\hat{\sigma}$ is in $C$ 
  and therefore in $W.$ The interior of $\sigma$ is the open star of $\hat{\sigma}$ in $K'$ 
  which is in $U.$ Thus $x$ is in $U.$
 
 By \ref{simplexbundle} $U\setminus |W|$ is a line bundle.
Now $U$ contains $C$ and $W$ is contained in the interior of $C$ hence $\partial' C\subset U\setminus |W|.$

Each of the lines, $\ell,$ in the line bundle is the interior of a straight line with one endpoint, $x,$ in $W$ 
and the other, $y,$ in the boundary of the closure of $U.$  Thus $x\in int(C)$ and $y\notin C$ and it follows that $\ell$
contains a point of $\partial' C.$  A line segment
in a convex set is either contained in the boundary of the convex set, or else contains at most one boundary point.
Thus $\ell$ contains a unique point of $\partial' C.$
Since $K''$ is a derived subdivision of $K'$ adapted to $C$ it follows that $L$ is the simplicial 
neighborhood of $W$ in $K''.$ Hence $\ell$ also meets $\partial |L|.$ By considering the second  
derived subdivison of a simplex one sees that 
$\ell$ also meets $\partial |L|$ in a single point. 
It follows that the closure of $C\setminus |L|$ is a product $I\times\partial' C$ as claimed.
  \end{proof}

The proof of the remaining topological finiteness results as outlined above requires only one 
more ingredient: To apply Lemma \ref{convexsubcomplex} we must ensure that the  intersection of the thin 
part of $M$ with $|K|$ is strongly locally convex. To do this we replace $C$ by a convex simplicial
complex and then move $K$ into general position with respect to $C$:

 We claim that  there is a homeomorphism arbitrarily close to the identity of $M$ to itself
 which takes the thin part of $M$ to the underlying space of a simplicial complex, $L,$
 such that each component of $C=|L|$ is convex. We then move  $K$ into general position with 
 respect to $L.$ This implies $C$ is strongly locally convex relative to $K.$
 
Let $A\subset M$ be the convex-thin part given by \ref{convexthin}. 
 We replace each component of $A$ by  a slightly larger convex simplicial
neighborhood to obtain $C,$ possibly triangulated with an extremely large number of simplices. 
Since $A$ and $C$ are both convex
there is a homeomorphism of $M$ to itself
which is the identity outside a small neighborhood of $C$ and takes $C$ onto $A.$

We can assume the simplices of $K$ are small enough that no simplex intersects two components of $C$.
Now use general position to move $K$ so that each component of $C$ is  strongly locally convex with respect to $K.$  
\section{Relative Hyperbolicity}  

A {\em geodesic} in a metric space is a rectifiable path such
that  the length of every sufficiently short subpath equals the distance between its endpoints.
A metric space $X$ is a {\em geodesic metric space} if every pair of points is connected by a 
geodesic.  A {\em triangle} in a metric space consists of three geodesics arranged in the usual way.

A triangle is {\em $\delta$-thin} if every point on each side of the triangle is within a distance
$\delta$ of the union of the other two sides. A triangle is called {\em $\delta$-fat}  if it
is not $\delta$-thin.
 If $X$ is a locally compact, complete
geodesic metric space and every triangle in  $X$ is $\delta$-thin 
then $X$ is called {\em $\delta$-hyperbolic.} 

These ideas can be applied to a properly convex domain with the Hilbert metric. 
Some care is required with terminology
in view of the fact that if $\Omega$ is
strictly convex then geodesics are precisely  projective line segments, otherwise
if $\Omega$ is only properly convex, there may be geodesics which are not
segments of projective lines, and triangles with geodesic sides which are not planar. 
A {\em straight triangle} in projective space is a disc in a projective plane
bounded by three sides  that are segments of projective lines. A straight triangle is $\delta$-thin if its boundary is $\delta$-thin. In view of this the following is re-assuring:

\begin{lemma}[straight-thin implies thin]\label{straightthin} 
If every straight triangle in a properly convex domain $\Omega$ is $\delta$-thin, 
then $\Omega$ is strictly convex.
\end{lemma}
\begin{proof} Suppose that  there is a line segment $\ell$ in the boundary of $\Omega$. 
Choose a sequence $x_n\in\Omega$ which converges to a point in the
interior of $\ell.$ It is easy to see that (a sub-triangle in $\Omega$ of) the straight triangle $T_n$ that is the convex 
hull of $x_n$ and $\ell$ becomes arbitrarily fat as $n\to \infty$ , which contradicts the hypothesis that $\Omega$
is  $\delta$-thin.
\end{proof}
\begin{proposition}[maximal cusps bilipschitz hyperbolic]\label{thincusps}
Suppose that $C$ is a maximal rank cusp in a strictly convex manifold of finite volume.

Then $C$ is bilipschitz homeomorphic  to a cusp of a hyperbolic manifold. In particular 
the universal cover of $C$ is $\delta$-hyperbolic.
\end{proposition}
\begin{proof} By Theorem \ref{maximalcuspsstandard}, maximal rank cusps are hyperbolic, so that the cusp $C$ can be viewed 
as a submanifold of $\Omega/\Gamma$ with $\Gamma <  PO(n,1)_p < PO(n,1)$, 
where $PO(n,1)_p$ is the group of parabolics that fixes a point $p\in\bOmega.$ 
Let $\widetilde{C}$ denote the preimage of $C$ in $\Omega.$ 
By \ref{maxparabolicround} $p$ is a round point of $\Omega$ so there is a 
unique supporting hyperplane $H$ to $\Omega$ at $p.$ 

Parabolic coordinates centered on $(H,p)$ give an affine patch ${\mathbb A}^n$. 
Since $\Gamma \le \SL(H,p),$ the round (open)
ball, ${\mathbb H}^n$, which is preserved by $PO(n,1)$
is contained in this affine patch. Moreover, this patch is the union of generalized horoballs ${\mathcal B}_t$ for
$PO(n,1)_p$. It is first shown that there are two of these horoballs such that
${\mathcal B}_s\subset\widetilde{C} \subset  \Omega \subset {\mathcal B}_t$.

\begin{figure}[ht]
 \begin{center}
 \psfrag{C}{$\widetilde{C}$}
\psfrag{Om}{$\Omega$}
\psfrag{D1}{${\mathcal B}_s$}
\psfrag{D2}{${\mathcal B}_t$}
\psfrag{x1}{$p$}
\psfrag{P}{$H$}
\psfrag{G}{$\Gamma$}
\psfrag{Kp}{$\color{blue} K'$}
\psfrag{q}{$\color{blue} K$}
	 \includegraphics[scale=0.9]{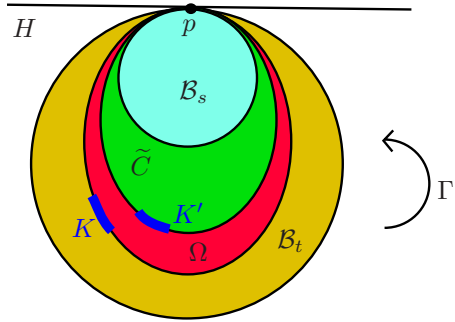}
 \end{center}
\caption{Comparing $\Omega$ to ${\mathcal B}_t$ and ${\mathcal B}_s$.}	 \label{cuspcompare}

\end{figure}

Refer to figure \ref{cuspcompare} (which is drawn in a different affine patch).
 Because the cusp has maximal rank,  $\bOmega\setminus p$ contains a compact fundamental 
domain $K$ for the action of $\Gamma$ and
$K$ is contained in ${\mathcal B}_t$ for some $t$. It follows that ${\mathcal B}_t$ contains
the $\Gamma$ orbit of $K$ and thus contains $\Omega.$ Similarly there is a compact
fundamental domain $K'$ for the action of $\Gamma$ on $\partial\widetilde{C}.$
Then for some $s$ the generalized horoball ${\mathcal B}_s$ 
 is disjoint from $K'$ and hence from $\partial\widetilde C.$ This proves the inclusions.

 The Hilbert metric on ${\mathcal B}_t$ is isometric to 
hyperbolic space ${\mathbb H}^n.$ Using the above parabolic coordinates it is easy to see that the Hilbert metrics
on $\Omega$ and ${\mathcal B}_t$  restricted to ${\mathcal B}_s$ are  bilipschitz. Since $p$ is a bounded parabolic fixed point, there are a constant $k$ and a maximal rank cusp $C' \subset C$ such that $\widetilde{C}' \subset {\mathcal B}_s\subset\widetilde{C}$ and $d_\Omega(x,C')\le k$ for all $x\in C.$
Thus $C$ is bilipschitz homeomorphic to ${\mathcal B}_s/\Gamma$ for both Hilbert metrics. 
Since ${\mathbb H}^n$ is $\delta$-thin and this property is preserved by quasi-isometry, the result follows.
\end{proof}
\noindent
{\bf Remark.} The metric on $C$ is asymptotically hyperbolic in the sense that  if ${\mathcal B}_s$ is 
sufficiently small the two metrics on ${\mathcal B}_s$ are $(1+\epsilon)$-bilipschitz.
\gap

There are several equivalent definitions of the term {\em relatively hyperbolic}. We will use 
Gromov's original definition \cite{Gromov, Bumagin} in the context of a properly convex 
projective manifold, $M,$ of finite volume which is the interior of a compact
manifold whose ends are cusps.

 Recall that each end of  $M$ is a horocusp which is covered by a family of disjoint 
 horoballs in the universal cover. Part of Gromov's definition requires the ends of $M$ have this structure.
Then, following Gromov, one says that $\pi_1M$  is {\em relatively hyperbolic} relative to the collection 
of subgroups $\{\pi_1A\}$ (where $A$ ranges over 
the boundary components of $M$) if the following conditions are satisfied:\\
\nb $\widetilde{M}$ is $\delta$-hyperbolic \\
\nb $M$ is quasi-isometric to the union of finitely many copies of $[0,\infty)$ joined at $0.$

\gap 
By Proposition \ref{thincusps}, each cusp in $M$ is bilipschitz to a maximal hyperbolic cusp. The latter 
is foliated by compact horomanifolds (intrinsically Euclidean) whose diameter decreases 
as one goes into the cusp. In particular
such a cusp is quasi-isometric to $[0,\infty).$ Now $M$ with the cusps  deleted is compact and connected thus
quasi-isometric to a point.
It follows that this  second condition is always satisfied in our context, so that for such manifolds:\\[\baselineskip]
\begin{enumerate}
\item[($\star$)] $\widetilde{M}$ is $\delta$-hyperbolic implies $\pi_1M$ is relatively hyperbolic.\\[\baselineskip]
\end{enumerate}
Following Benoist, a {\em properly embedded triangle} or {\em PET} in a convex set $\Omega$ 
is a straight triangle $\Delta$ with interior in $\Omega$ and boundary in $\bOmega.$ 
A {\em hex plane} is  any metric space isometric to the metric in example E(ii) of \S\ref{newparabolics}. 

 If $C$ is a circle of maximum radius in a straight triangle, a center of $C$
is called an {\em incenter} and the radius of $C$ is the {\em inradius.} 
The following is an easy exercise:
\begin{lemma}\label{incenter}
 A straight triangle $T$ in  a properly convex domain $\Omega$ has a unique incenter. 
 If $T$ is $\delta$-fat the inradius is at least $\delta/2.$
 \end{lemma}
\begin{lemma}[fat triangle limit is PET] Suppose that $\Omega$ is properly convex and  $T_n$ is a sequence
of straight triangles in $\Omega.$ Suppose that $x_n\in T_n$ and $d(x_n,\partial T_n)\to\infty$  and $x_n\to x\in\Omega$.

Then there is a subsequence of the triangles
which converges (in the Hausdorff topology on closed subsets of $\RPn$) to a PET
in $\Omega$ containing $x.$
\end{lemma}
\begin{proof} The sequence of straight triangles has a subsequence converging to a (possibly degenerate) straight triangle $T$ 
containing $x.$
Since $d(x_n,\partial T_n)\to\infty$  the distance of $x$ from $\partial T$ is infinite.
Hence $\partial T\subset\bOmega.$\end{proof}
Combining this with Benzecri's compactness theorem gives:
\begin{lemma} 
Given a sequence $T_n\subset\Omega_n$ of straight triangles in properly convex domains
for which  $x_n\in T_n$ and $d(x_n,\partial T_n)\to\infty.$ 

Then after taking a subsequence and applying suitable projective transformations:
\begin{itemize}
\item $(\Omega_n,T_n,x_n)\to(\Omega,T,x)$ in the Hausdorff topology
on subsets of $\RPn$,
\item  $\Omega$ is properly convex,
\item $T$ is a PET in $\Omega.$
\end{itemize}
\end{lemma}
This implies that inside a large circle centered at a point in the interior of
any straight triangle far from the boundary, the metric is very close to the hex metric; for if this was not the case,
we could find a sequence of triangles and domains $(\Omega_n,T_n,x_n)$ with the property
that $d(x_n,\partial T_n)\to\infty$, but the metric  on large balls about $x_n$ does not
become close to the hex metric. We then apply the Lemma and obtain a contradiction.

Notice that such a large circle contains a very fat straight triangle.

\begin{theorem} 
\label{generalized_benoist}
Suppose that $M=\Omega/\Gamma$ is a properly convex complete projective manifold 
of finite volume which is the interior of a compact manifold $N$ and the holonomy of each 
component of $\partial N$ is parabolic. 
Then the following are equivalent:\\
(1) $(\Omega,d_{\Omega})$ is $\delta$-hyperbolic,\\
(2) $\Omega$  is strictly convex, \\
(3) $\Omega$ does not contain a PET,\\
(4) $\Omega$ does not contain a PET which projects into a compact submanifold $B$ of $M,$\\
(5) $\pi_1M$ is relatively hyperbolic,\\
(6) $\partial \Omega$ is $C^1.$
\end{theorem}
\begin{proof}  Each component of $\partial N$ is compact and therefore each end of  $M$ is a maximal rank cusp.
That (1) $\implies$ (2) follows from  \ref{straightthin}.
It is clear (2) $\implies$ (3) $\implies$ (4). 

For (4) $\implies$ (1), assume (1) is false.  Then by \ref{straightthin} for each $n>0$ there is an  $n$-fat straight triangle  $\Delta_n$ in $\Omega.$ 
 Let $D_n$ denote the disc in $\Delta_n$ of radius $n/2$ center at the incenter $x_n.$
Let $\pi:\Omega\longrightarrow M$ be the projection. 
By hypothesis $M$ is the union of a compact submanifold, $B,$ and finitely many cusps.
Furthermore, every cusp is covered by a horoball which is $\delta$-thin.

We claim that $B$ may be chosen so that $\pi( D_n)\subset B$ for all $n.$
For otherwise  there is a subdisc $D_n'\subset D_n$  with
radius $r_n\to\infty$ and $\pi(D_n')$ eventually leaves every compact set. After taking a subsequence $\pi D_n'$
are all contained in the same cusp $C$ of $M.$ 
There is an $r_n'$-fat triangle $\Delta_n'\subset D_n'$ and $r_n'\to\infty.$
Choose a horoball $\widetilde{C}$ which is a component of $\pi^{-1}C.$
A translate of $\Delta_n'$ by some element of $\Gamma$ is contained in $\widetilde{C}.$ 
Since $r_n'\to\infty$ this contradicts that $\widetilde{C}$
is $\delta$-thin, proving the claim.

Since $B$ is compact we may choose $\gamma_n\in\Gamma$ so that $\gamma_n(x_n)$ converges to
a point $x_{\infty}\in\Omega$ and $\gamma_n(D_n)$ converges in the Hausdorff topology on closed
subsets of $\RPn$ to a planar disc $D_{\infty}$ with interior in $\Omega.$ This also the Hausdorff limit of
the sequence of straight triangles $\gamma_n(\Delta_n')$. Hence this limit is a PET and this implies (4) is false.
This completes the proof that the first 4 conditions are equivalent.

Condition ($\star$) above shows  (1)  $\implies$ (5).

For (5)  $\implies$ (4)  assume (4) is false, so that $\Omega$ contains a PET $\Delta,$ which projects into  $B.$
It follows from Dru{\c{t}}u \cite{drutu} Theorem 1.4 and condition 
$(\beta_3)$ of Theorem 1.6   that if  (5) were true then every quasi isometric embedding of a Euclidean plane 
into $\widetilde{B}$ lies within a bounded neighborhood of one boundary component of $\widetilde{B}.$
This would imply $\Delta$ lies within a bounded distance of a horoball covering a cusp.
By \ref{thincusps} a horoball covering a cusp is $\delta$-thin. 
 A $K$-neighborhood
of such a horoball is quasi-isometric to  the horoball and therefore $\delta'$-thin. Therefore $\Delta$ cannot
be in this neighborhood, so (5) is false.

Since $M$ is of finite volume, so is $M^*$ by \ref{cor:finvol-iff-dual-finvol}. Whence
(6) $\Leftrightarrow$ ($\Omega^*$ is strictly convex)$\Leftrightarrow$ (5).
\end{proof}

\small
\bibliography{finitevolumemaster.bbl} 
\bibliographystyle{abbrv} 

\end{document}